\documentclass{amsart}

\usepackage{amssymb}
\usepackage[all]{xy}

\def\comment#1{{\sf{[#1]}}}

\def\Z{{\mathbb Z}}
\def\Q{{\mathbb Q}}
\def\C{{\mathbb C}}
\def\P{{\mathbb P}}

\def\H{{\mathbb H}}
\def\L{{\mathbb L}}
\def\V{{\mathbb V}}

\def\Aff{{\mathbb A}}

\def\A{{\mathcal A}}

\def\cC{{\mathcal C}}

\def\cG{{\mathcal G}}
\def\cH{{\mathcal H}}

\def\K{{\mathcal K}}
\def\cL{{\mathcal L}}
\def\M{{\mathcal M}}

\def\cP{{\mathcal P}}

\def\U{{\mathcal U}}
\def\X{{\mathcal X}}
\def\e{\epsilon}
\def\w{\omega}

\def\G{\Gamma}

\def\a{{\mathfrak a}}
\def\d{{\mathfrak d}}
\def\g{{\mathfrak g}}

\def\k{{\mathfrak k}}
\def\n{{\mathfrak n}}
\def\p{{\mathfrak p}}
\def\r{{\mathfrak r}}

\def\u{{\mathfrak u}}
\def\z{{\mathfrak z}}

\def\cc{\mathbf{c}}
\def\dd{\mathbf{d}}
\def\ee{\mathbf{e}}
\def\x{\mathbf{x}}

\def\deltatilde{{\tilde{\delta}}}
\def\rhotilde{{\tilde{\rho}}}
\def\rhohat{{\hat{\rho}}}
\def\sigmatilde{{\tilde{\sigma}}}
\def\phihat{\hat{\phi}}
\def\etabar{{\overline{\eta}}}
\def\wtilde{{\tilde{\w}}}
\def\what{{\hat{\w}}}

\def\Ghat{\widehat{\G}}

\def\uhat{\hat{\u}}
\def\ghat{\hat{\g}}
\def\stilde{\tilde{s}}

\def\Utilde{\widetilde{U}}
\def\Ttilde{\widetilde{T}}

\def\Uhat{\widehat{\U}}
\def\cGhat{\widehat{\cG}}

\def\Gtilde{\widetilde{G}}

\def\Mo{\overset{\ \circ}{\vphantom{n}\smash{\M}}}

\def\kbar{{\overline{k}}}
\def\Kbar{{\overline{K}}}

\def\ubar{\overline{u}}
\def\xbar{{\overline{x}}}

\def\thetadual{\check{\theta}}

\def\Ql{{\Q_\ell}}
\def\Zl{{\Z_\ell}}
\def\Zlx{{\Z_\ell^\times}}

\def\Qp{{\Q_p}}

\def\Gm{{\mathbb{G}_m}}
\def\Sp{{\mathrm{Sp}}}

\def\GSp{{\mathrm{GSp}}}

\def\un{\mathrm{un}}

\def\geom{\mathrm{geom}}
\def\cts{\mathrm{cts}}

\def\top{\mathrm{top}}
\def\orb{\mathrm{orb}}
\def\an{\mathrm{an}}

\def\et{\mathrm{\acute{e}t}}

\def\nab{\mathrm{nab}}

\def\Het{H_\et}
\def\Hnab{H_\nab}

\def\dot{\bullet}
\def\bmu{\pmb{\mu}}
\def\tambo{\boxplus}
\def\blank{\phantom{x}}
\def\red#1{{\langle #1 \rangle}}
\def\bracket{\text{\sf bracket}}
\def\cupp{\text{\sf cup}}

\def\sect{\text{\sf sect}}

\def\bs{\backslash}

\newcommand\im{\operatorname{im}}
\newcommand\id{\operatorname{id}}
\newcommand\ad{\operatorname{ad}}

\newcommand\Hom{\operatorname{Hom}}
\newcommand\Ext{\operatorname{Ext}}
\newcommand\Homcts{\operatorname{Hom_{\mathrm{cts}}}}
\newcommand\Spec{\operatorname{Spec}}
\newcommand\Diff{\operatorname{Diff}}
\newcommand\Aut{\operatorname{Aut}}
\newcommand\Out{\operatorname{Out}}
\newcommand\Der{\operatorname{Der}}

\newcommand\Gr{\operatorname{Gr}}

\newcommand\Jac{\operatorname{Jac}}
\newcommand\Ind{\operatorname{Ind}}
\newcommand\Char{\operatorname{char}}
\newcommand\Gal{\operatorname{Gal}}
\newcommand\Pic{\operatorname{Pic}}

\newcommand\codim{\operatorname{codim}}
\newcommand\Sect{\operatorname{Sect}}

\newtheorem{theorem}{Theorem}[section]
\newtheorem{lemma}[theorem]{Lemma}
\newtheorem{proposition}[theorem]{Proposition}
\newtheorem{corollary}[theorem]{Corollary}
\newtheorem{bigtheorem}{Theorem}

\theoremstyle{definition}
\newtheorem{definition}[theorem]{Definition}
\newtheorem{example}[theorem]{Example}

\theoremstyle{remark}
\newtheorem{remark}[theorem]{Remark}

\begin{document}

\title{Rational Points of Universal Curves}

\author{Richard Hain}
\address{Department of Mathematics\\ Duke University\\
Durham, NC 27708-0320}
\email{hain@math.duke.edu}

\thanks{Supported in part by grant DMS-0706955 from the National Science
Foundation and by MSRI}

\date{\today}

%\begin{abstract}
%
%\end{abstract}

\maketitle

\tableofcontents

\section{Introduction}

Associated to a smooth, projective, geometrically connected curve $C$ over a
field $K$ is the short exact sequence
$$
1 \to \pi_1(C\otimes_K \Kbar,\etabar) \to \pi_1(C,\etabar) \to G_K \to 1
$$
of \'etale fundamental groups, where $\Kbar$ is a separable closure of $K$, $G_K
= \Gal(\Kbar/K)$ and where $\etabar$ is a geometric point of $C$. Each
$K$-rational point $x$ of $C$ induces a section $s_x : G_K \to \pi_1(C,\etabar)$
of the right-hand map, which is well defined up to conjugation by an element of
the geometric fundamental group $\pi_1(C\otimes_K \Kbar,\etabar)$.
Grothendieck's Section Conjecture \cite{grothendieck} asserts that, if $K$ is an
infinite field that is finitely generated over its prime field, and if the genus
of $C$ is 2 or more, then the function
$$
C(K) \to \{\text{sections of }\pi_1(C,\etabar) \to G_K\}/
\text{conjugation by } \pi_1(C\otimes_K \Kbar,\etabar)
$$
that takes $x$ to the conjugacy class of $s_x$ is a bijection. In this paper we
prove the section conjecture for the restriction of the universal curve $\cC \to
\M_{g}$ to its generic point $\Spec(k(\M_g))$ for all $g\ge 5$ when $\Char k
=0$.\footnote{Here and throughout the introduction, $\M_g$, $\M_{g,n}$ and
$\M_{g,n}[m]$ are considered to be stacks over the field $k$.} When $g\ge 5$ and
$n\ge 1$ we prove a modified version of the Section Conjecture for the pullback
of the universal curve to the generic point of $\M_{g,n}$ in which
$\pi_1(C\otimes_K \Kbar,\etabar)$ is replaced by its $\ell$-adic prounipotent
completion, where the prime number $\ell$ is chosen so that the image of the
$\ell$-adic cyclotomic character $\chi_\ell : G_k \to \Zlx$ is infinite. In
particular, this version of the ``unipotent section conjecture'' holds whenever
$k$ is a number field or a finite extension of $\Qp$, where $p\neq \ell$.

The proof proceeds in two stages. We first show that, when $g\ge 3$, the
universal curve over $\M_{g,n}$ has no more $k(\M_{g,n})$-rational points other
than the obvious ones --- namely, the rational points that are the restrictions
of the $n$ tautological sections of $\cC \to \M_{g,n}$. This follows directly
from results in Teichm\"uller theory proved by Hubbard \cite{hubbard} and Earle
and Kra \cite{earle-kra} in the 1970s. We give an algebraic proof of this
consequence of their results which also applies to ample subvarieties of
$\M_{g,n}$ of dimension $\ge 2$ and which should also hold in positive
characteristic.

The second step is to study conjugacy classes of sections of the group extension
$$
1 \to \pi_1(C\otimes_K \Kbar,\etabar)^\un_{/\Ql} \to
G_K\ltimes\pi_1(C\otimes_K \Kbar,\etabar)^\un_{/\Ql} \to G_K \to 1
$$
where $K=k(\M_{g,n})$, $n\ge 1$, and where $C$ is the pullback of the universal
curve to $\Spec K$. This is done using the non-abelian cohomology developed by
Kim in \cite{kim:coho} for extensions of a profinite group by a prounipotent
group, which proves to be very effective in this case.

The principal tool used to prove these results is the theory of weighted
completions of profinite groups
\cite{hain-matsumoto:weighted,hain-matsumoto:mcgs}, which was developed with
Makoto Matsumoto. Other notable ingredients used include the computation of the
relative completion of mapping class groups in genus $\ge 3$ in
\cite{hain:torelli},  a ``density theorem'' \cite{hain:density}, and Kim's
non-abelian cohomology mentioned above.

This work is inspired by work of Ellenberg \cite{ellenberg} and Kim
\cite{kim:program} who independently tried to bound the number of rational
points of hyperbolic curves over a number field by bounding the number of
conjugacy classes of sections of the surjective homomorphism
$\pi_1(C,\etabar)/L^r \to G_K$, where $L^r$ denotes the $r$th term of the lower
central series of $\pi_1(C\otimes_K\Kbar,\etabar)$.

We now describe the main results in more detail. A natural generalization of the
universal curve $\cC \to \M_{g,n}$, introduced in Section~\ref{sec:top_amp}, is
that of a {\em topologically ample family of curves of type $(g,n)$}. Roughly
speaking, a family $C \to T$ of smooth projective curves over a geometrically
connected $k$-variety $T$ is topologically ample of type $(g,n)$, where $g\ge
3$, if it has $n$ disjoint sections and if the corresponding $k$-morphism $T \to
\M_{g,n}$ induces an isomorphism of the fundamental group of $T$ with a finite
index subgroup of $\pi_1(\M_{g,n},\etabar)$ that contains the profinite Torelli
group. (See Definition~\ref{def:top_ample}.) Examples of topologically ample
curves of type $(g,n)$ include the universal curve over the moduli space
$\M_{g,n}[m]$ of smooth projective curves of type $(g,n)$ with a level $m$
structure, where $g+n\ge 4$ or $g\ge 3$ and $m\ge 3$, and its restriction to a
generic linear section of $\M_{g,n}[m]$ of dimension $\ge 2$.

\begin{bigtheorem}
\label{thm:rational}
Suppose that $C\to T$ is a topologically ample family of curves of type $(g,n)$
over a field $k$ of characteristic zero. If $g\ge 3$, then the only
$k(T)$-rational points of $C$ are its $n$ tautological points: $C(k(T)) =
\{x_1,\dots,x_n\}$.
\end{bigtheorem}

The theorem is false when $g=2$. Since every genus 2 curve is hyperelliptic, 
$$
|C(k(\M_{2,n}[2m]))| \ge 6 + 2n
$$
whenever $m+n\ge 2$ as, apart from the $n$ tautological sections,
$C(k(\M_{2,n}[2m]))$ contains the $2g+2=6$ Weierstrass points of the universal
curve and the $n$ hyperelliptic conjugates of the tautological
points.\footnote{It is necessary to take $m+n\ge 2$ as the stack $\M_2[2]$ does
not have a function field. The Weierstrass points are rational as the level is
even, which implies that the $2$-torsion of the jacobian is trivialized.} In
fact, the result of Earle-Kra \cite{earle-kra} implies that there are no
additional rational points, so that one has equality above.

The following result establishes the Section Conjecture for the pullback of the
universal curve to the generic point of $\M_{g}[m]$ when $g\ge 5$, provided
that there exists a prime number $\ell$ such that the image of the $\ell$-adic
cyclotomic character $\chi_\ell : G_k \to \Zlx$ is infinite. 

\begin{bigtheorem}
\label{thm:zero_points}
Suppose that $k$ is a field of characteristic zero, that $\ell$ is a prime
number, and that $m\ge 1$. Set $K = k(\M_{g}[m])$ and let $C/K$ be the pullback
of the universal genus $g$ curve to $\Spec K$. Fix a geometric point $\etabar$
of $C$. If $g\ge 5$, and if the image of the $\ell$-adic cyclotomic character
$\chi_\ell : G_k \to \Zlx$ is infinite, then the extension
$$
1 \to \pi_1(C\otimes_K\Kbar,\etabar) \to \pi_1(C,\etabar) \to G_K \to 1
$$
does not split.
\end{bigtheorem}

When $n\ge 1$, we can only prove a weaker version of this result. As in the
previous result, $k$ is a field of characteristic zero. Fix $m\ge 1$ and set $K
= k(\M_{g,n}[m])$. Let $C/K$ be the pullback of the universal curve to $\Spec
K$.

Each $K$-rational point $x_j \in \{x_1,\dots,x_n\}$ of $C$ induces a continuous
homomorphism
$$
s_j : G_K \to \Aut \pi_1(C\otimes_K \Kbar,\xbar_j)
$$
where $\xbar_j : \Spec \Kbar \to C$ is a geometric point that lies over $x_j$.
This induces a homomorphism
$$
s_j^\un : G_K \to \Aut \pi_1(C\otimes_K \Kbar,\xbar_j)^\un_{/\Ql}
$$
into the automorphism group of the $\Ql$-unipotent completion of
$\pi_1(C\otimes_K \Kbar,\xbar_j)$.

\begin{bigtheorem}
\label{thm:sec_pos}
Suppose that $g\ge 5$ and that $n\ge 1$. Suppose that $k$ is a field of
characteristic zero and that $\ell$ is a prime number for which the $\ell$-adic
cyclotomic character $\chi_\ell : G_k \to \Zlx$ has infinite image. With the
notation introduced above, if $H^1(G_k,\Ql(r))$ is finite dimensional for all $r
> 1$, then then there is a natural bijection between the
$\pi_1(C\otimes_K\Kbar,\xbar_1)^\un(\Ql)$ conjugacy classes of splittings of
$$
1 \to \pi_1(C\otimes_K\Kbar,\xbar_1)^\un(\Ql) \to
G_K\ltimes_{s_1^\un} \pi_1(C\otimes_K\Kbar,\xbar_1)^\un(\Ql) \to G_K \to 1
$$ 
and elements of $C(K)$.
\end{bigtheorem}

The theorem applies, for example, when $k$ is a number field and when $k$ is a
finite extension of $\Qp$ provided that $\ell\neq p$. It should also hold when
$k$ is a finite field of characteristic $p$ and $\ell\neq p$. Truncated versions
of this result are proved in Section~\ref{sec:thm3}. The theorem also holds when
$\M_{g,n/k}[m]$ is replaced by a generic section of it by a complete
intersection $T$ of dimension $\ge 3$ and sufficiently high multi-degree. The
details of the proof of this extension, Theorem~\ref{thm:variant}, are similar
to those Theorem~\ref{thm:sec_pos} and are left to the interested reader.

When $C$ is an affine curve, the story is more subtle as the second weight
graded quotient of $\pi_1^\un(C\otimes_K\Kbar,\etabar)$ contains at least one
copy of $\Ql(1)$. Since $H^1(G_K,\Ql(1))$ is infinite dimensional, this implies
that $H^1(G_K,\Gr^W_{-2}\pi_1^\un(C\otimes_K\Kbar))$ is infinite dimensional
when $C$ is affine. To avoid unnecessary technicalities, the affine case has
been postponed.

\bigskip

As remarked above, the principal tool used to prove the main results is the
theory of weighted completion of profinite groups developed with Makoto
Matsumoto in \cite{hain-matsumoto:weighted,hain-matsumoto:mcgs}. Given a curve
$C/T$ satisfying certain natural conditions, one can take the weighted
completion of the groups $\pi_1(C,\etabar_C)$ and $\pi_1(T,\etabar)$ with
respect to their natural monodromy representations in $\GSp(H) :=
\GSp(\Het^1(C_{\etabar},\Ql(1)))$ to obtain $\Ql$-affine proalgebraic groups
$\cG_C$ and $\cG_T$, both of which are extensions of $\GSp(H)$ by a prounipotent
group. The projection $C \to T$ induces a surjective homomorphism
\begin{equation}
\label{eqn:homom}
\cG_C \to \cG_T.
\end{equation}
Sections of $\pi_1(C,\etabar_C) \to \pi_1(T,\etabar)$ induce sections of
(\ref{eqn:homom}). Theorem~\ref{thm:rational} is proved by exploiting the
presentation of the Lie algebra $\u_{g,n}$ of the pronilpotent radical of the
relative completion of the mapping class group $\pi_1(\M_{g,n},x)^\an$ that was
computed in \cite{hain:torelli}.

The Lie algebras $\u_T$, $\u_C$ and $\u_{g,n}$ of the prounipotent radicals of
the fundamental groups of $T$, $C$ and $\M_{g,n}$ all have a natural weight
filtration
$$
\u = W_{-1}\u \supseteq W_{-2}\u \supseteq W_{-3}\u \supseteq \cdots
$$
whose graded quotients are $\GSp(H)$-modules. To prove
Theorem~\ref{thm:rational}, one need only use their quotients $\u_C/W_{-3}$ and
$\u_T/W_{-3}$ by $W_{-3}$. The principal defining property of a topologically
ample family $C/T$ of curves of type $(g,n)$ is that the morphism $T \to
\M_{g,n}$ induce an isomorphism $\u_T/W_{-3}\to \u_{g,n}/W_{-3}$. This clearly
holds when $T \to \M_{g,n}$ induces an isomorphism on fundamental groups, such
as when $T$ is a generic linear section of $\M_{g,n}$ of dimension $\ge 2$.

To prove Theorems~\ref{thm:zero_points} and \ref{thm:sec_pos}, one needs to
understand how $\u_T/W_{-3}$ changes when one replaces $T=\M_{g,n}$ by a
Zariski open subset $U$. It turns out that when $g\ge 5$, the induced Lie
algebra homomorphism
$$
\u_U/W_{-3} \to \u_{g,n}/W_{-3}
$$
is close to being an isomorphism. More precisely, its weight $-1$ quotient does
not change when one replaces $\M_{g,n}$ by $U$, and its weight $-2$ quotient is
an isomorphism modulo the trivial representation of $\Sp(H)$. This is proved
using a corollary of the ``density theorem'' from \cite{hain:density}, which
asserts that if $X$ is a complex algebraic manifold, and $f:X \to \M_{g,n}^\an$
a morphism whose image is a divisor, then the image of the natural monodromy
representation $\pi_1(X,x_o) \to \Sp_g(\Z)$ has finite index in $\Sp_g(\Z)$. The
proof of this density theorem makes essential use of Goresky and MacPherson's
stratified Morse theory \cite{smt} and a recent result \cite{putman:pic} of
Putman on the Picard groups of level coverings of $\M_{g,n}$ for $g\ge 5$.

This invariance of $\u/W_{3}$ is encoded in the cohomological computations of
Sections~\ref{sec:coho_comps} and \ref{sec:lie_coho_comps}.
Theorems~\ref{thm:zero_points} and \ref{thm:sec_pos} are then proved using Kim's
non-abelian cohomology \cite{kim:coho} and provide an example where his methods
effectively find all of the rational points of a curve. These non-abelian
computations are also related to, and might illuminate, work of Jordan Ellenberg
\cite{ellenberg} and Kirsten Wickelgren \cite{wickelgren} on rational points of
curves over number fields.

The computations and arguments in Sections~\ref{sec:coho_comps} through 
\ref{sec:thm3} introduce certain cohomology classes of degree $\le 2$ in the
cohomology of $\M_{g,n}$ with coefficients in local systems corresponding to 
certain $\GSp(H)$ modules $V$. Since a smooth projective curve of genus $g$ over
an extension field $K$ of $k$ with $n$-known rational points corresponds to a
map $\Spec K \to \M_{g,n}$, these classes specialize to classes in
$H^\dot(G_K,V)$, and can thus be regarded as characteristic classes of curves
$C$ over $K$ with $n$ known rational points. They should be useful in
investigating whether or not $C$ has any additional $K$-rational points.

\bigskip

\noindent{\em Convention:} Suppose that $k$ is a field. By a $k$-variety we will
mean an integral, separated scheme of finite type over $k$. All $k$-schemes will
be integral. By a family of smooth varieties over $X$ we mean a flat morphism $X
\to T$, each of whose closed fibers is smooth. When we want to emphasize that
$\M_{g,n}$ is regarded as a stack over $\Spec k$, we will denote it by the
unfortunately cumbersome notation $\M_{g,n/k}$.

\bigskip

\noindent{\em Acknowledgments:} Much of this work was done in 2009 during visits
to MSRI, the Universit\'e de Nice and the Newton Institute. I am very grateful
to these institutions for their hospitality and support and to Duke University
for the sabbatical leave during which this work was completed. I am also
grateful to the many mathematicians who took an interest in this work and with
whom I had many helpful discussions, particularly Jordan Ellenberg, Minhyong Kim
and Makoto Matsumoto. I am also grateful to Arnaud Beauville and Naria Kawazumi,
both of whom pointed out the existence and relevance of the works of Hubbard
\cite{hubbard} and Earle and Kra \cite{earle-kra}. Finally, I owe a huge debt of
gratitude to the referee for his/her very thorough reading of the manuscript and
his/her numerous detailed comments and corrections.

\section{Fundamental Groups and Galois Groups}

We will have occasion to use various fundamental groups. In this section we
enumerate some of them and summarize their basic properties. Unless mentioned
otherwise, throughout this section $X$ is a smooth, geometrically connected
variety defined over a field $k$ of characteristic zero.

\subsection{The topological fundamental group} When $k$ is a subfield of $\C$,
we shall denote the corresponding analytic variety by $X^\an$. For $x\in X(\C)$,
we have the topological fundamental group $\pi_1^\top(X^\an,x)$, which we shall
denote simply by $\pi_1^\top(X,x)$.

If $X$ is a DM stack over $k$, the associated analytic space $X^\an$ is an
orbifold (i.e., a stack in the category of topological spaces). In this case, we
will denote the orbifold fundamental group of $X^\an$ by $\pi_1^\orb(X,x)$,
where $x$ is any suitable basepoint. Such fundamental groups are defined in
\cite{noohi}, for example. Typically in this paper, $X$ will be the quotient of
a smooth variety be a finite group. In this case the orbifold fundamental group
is easy to describe directly. See, \cite[\S3]{hain:elliptic}, for example.

\subsection{The \'etale fundamental group}

Denote the \'etale fundamental group of $X$ by $\pi_1(X,\etabar)$, where
$\etabar$ is a geometric point of $X$. The \'etale fundamental group of a DM
stack $X$ can also be defined --- see \cite{noohi}. We will also denote it by
$\pi_1(X,\etabar)$. Fix an algebraic closure $\kbar$ of $k$. The fundamental
group of $\Spec k$ with respect to the geometric point $\Spec \kbar \to \Spec k$
is simply $\Gal(\kbar/k)$. We denote it by $G_k$. The structure morphism $X \to
\Spec k$ induces a homomorphism $\pi_1(X,\etabar) \to G_k$. One has the
canonical short exact sequence
\begin{equation}
\label{eqn:ses}
1 \to \pi_1(X\otimes_k \kbar,\etabar) \to \pi_1(X,\etabar) \to G_k \to 1.
\end{equation}
Each $k$-rational point $x\in X(k)$ induces a section of $\pi_1(X,\etabar) \to
G_k$ that is well-defined up to conjugation by an element of the geometric
fundamental group $\pi_1(X\otimes_k \kbar,\etabar)$.

\subsection{Topological versus \'etale fundamental groups}

When $k= \C$ and $x\in X(\C)$, there is a natural isomorphism
$$
\pi_1(X,x) \cong \pi_1^\top(X,x)^\wedge
$$
between the \'etale fundamental group of $X$ and the profinite completion of
its topological fundamental group. When $k$ is a subfield of $\C$, the exact
sequence (\ref{eqn:ses}) becomes
$$
1 \to \pi_1^\top(X,x)^\wedge \to \pi_1(X,x) \to G_k \to 1.
$$

\subsection{Relation to Galois groups}

Set $K = k(X)$ and let $\etabar$ be any geometric point that lies over the generic
point of $X$. This also serves as a basepoint for any Zariski open subset of
$X$. We therefore have an inverse system of profinite groups:
$$
\big\{\pi_1(X-D,\etabar)\big\}_D
$$
where $D$ ranges over the divisors in $X$ that are defined over $k$.

Denote the Galois group $\Gal(\Kbar/K)$ of $K$ by $G_K$. There are natural
surjections
$$
G_K \to \pi_1(X-D,\etabar).
$$
for each divisor $D$ in $X$.

\begin{proposition}
The natural homomorphism
$$
G_K \to \varprojlim_D \pi_1(X-D,\etabar)
$$
is an isomorphism.
\end{proposition}

\section{Points and Sections}
\label{sec:points}

Suppose that $T$ is a smooth variety over $k$, a field of characteristic zero,
and that $f : C \to T$ is a family of smooth projective curves. That is, $f$ is
a proper flat family of geometrically connected smooth projective curves over
$T$. Denote the function field $k(T)$ of $T$ by $K$. The following assertion
follows directly from the valuative criterion for properness \cite[7.3.8]{ega}. It
can also be proved by an elementary direct argument.

\begin{proposition}
There is a natural 1-1 correspondence between $K$-rational points $x\in C(K)$
and equivalence classes of sections $s : U \to C$ of $f$ defined on a Zariski
open subset $U$ of $T$ defined over $k$ where each component of $T-U$ has
codimension $\ge 2$ in $T$. Two such sections are defined to be equivalent if
they agree on the intersection of their domains.
\end{proposition}

Fix a prime number $\ell$ and let $\# \in \{\et, \top\}$ if $k=\C$ and $\#=\et$
in general. Fix geometric points $t_o$ of $T$ and $x_o$ of $C$ such that $f(x_o)
= t_o$. Denote the fiber of $C$ over $t_o$ by $C_o$.

\begin{corollary}
\label{cor:splitting}
Each $x\in C(K)$ induces a splitting 
$x_\ast : \pi_1^\#(T,t_o) \to \pi_1^\#(C,x_o)$
of the natural homomorphism
$f_\ast : \pi_1^\#(C,x_o) \to \pi_1^\#(T,t_o)$
that is unique up to conjugation by an element of $\pi_1^\#(C_o,x_o)$.
\end{corollary}

\begin{remark}
\label{rem:stix}
A result of Stix \cite{stix} implies that rational points $x\in C(K)$
correspond to global sections of $C \to T$. When $k=\C$, this follows from
Teichm\"uller theory. Indeed, if $s : U \to C$ is a holomorphic section $C\to T$
defined in the complement of an analytic subset $Z$ of $T$, each of whose
components has codimension $\ge 2$, then the inclusion $U\hookrightarrow T$
induces an isomorphism on fundamental groups, and thus an inclusion of
universal coverings $\Utilde \hookrightarrow \Ttilde$. The section $s$ lifts to
a holomorphic mapping $\tilde{s} : \Utilde \to \X_{g,1}$ into the Teichm\"uller
space $\X_{g,1}$ of marked Riemann surfaces of type $(g,1)$. Since $\X_{g,1}$ is
a domain in $\C^{3g-2}$, Hartog's Theorem implies that $\stilde$ extends to a
holomorphic mapping $\Ttilde \to \X_{g,1}$, which implies that $s$ extends to a
holomorphic section of $C\to T$.
\end{remark}

\section{Monodromy Representations}

\subsection{The groups $\Sp(H)$ and $\GSp(H)$}
\label{sec:gsp}

Suppose that $g \ge 1$ and that $A$ is a commutative ring. Suppose that $H_A$ is
a free $A$-module of rank $2g$ endowed with a unimodular, skew symmetric
bilinear form
$$
\beta : H_A \otimes H_A \to A.
$$
For an $A$-algebra $R$, set $H_R = H_A\otimes_A R$. The general symplectic group
is defined by
$$
\GSp(H_R) := \{\phi : H_R \to H_R :
\phi^\ast \beta = \tau(\phi)\beta \text{ for some } \tau(\phi) \in R^\times\}.
$$
Taking $\phi$ to $\tau(\phi)$ defines a homomorphism $\tau : \GSp(H_R) \to
R^\times$ whose kernel is the {\em symplectic group} $\Sp(H_R)$.

We will view $\GSp(H_A)$ and $\Sp(H_A)$ as group schemes over $A$ whose groups
of $R$-rational points are $\GSp(H_R)$ and $\Sp(H_R)$, respectively. We will
omit the ground ring $A$ from the notation when it should be clear from the
context. There is an exact sequence of algebraic groups
$$
1 \to \Sp(H) \to \GSp(H) \overset{\tau}{\to} \Gm \to 1.
$$

The pairing $\beta$ can be replaced by a $\GSp(H)$-pairing $\theta$ as follows.
For $r\in \Z$, denote by $A(r)$ the representation of $\GSp(H)$ whose underlying
$A$-module is free of rank $1$ and where $\GSp(H_A)$ acts via $\tau^r$. Each
choice of a generator $a_o$ of $A(1)$ gives rise to the $\GSp(H)$-invariant
bilinear pairing
$$
\theta : H \otimes H \to A(1)\qquad \theta: x\otimes y \mapsto \beta(x,y)a_o.
$$
With this choice
$$
\GSp(H_R) := \{\phi : H_R \to H_R : \phi^\ast \theta  = \theta\}.
$$

For a $\GSp(H_A)$-module $V$, denote by $V(r)$ the representation $V\otimes_A
A(r)$. Since $\beta$ is unimodular, $\theta$ induces an isomorphism
$$
\xymatrix{
H_A \ar[r]^(.22)\simeq & H_A^\ast(1) := \Hom_A(H_A,A)(1)
}
$$
The representation $\Lambda^2 V$ will be regarded as a {\em quotient} of
$V^{\otimes 2}$; the image of $v_1\otimes v_2$ will be denoted $v_1\wedge v_2$.
The pairing
$$
(\Lambda^2 H_A^\ast)\otimes (\Lambda^2 H_A)\to A \qquad 
(\phi_1\wedge\phi_2)\otimes (x_1\wedge x_2) := \det(\phi_i(x_j))
$$
is perfect and thus induces isomorphisms
$$
(\Lambda^2 H_A)^\ast(2) \cong \Lambda^2 \big(H_A^\ast(1)\big)
\cong \Lambda^2 H_A
$$
The dual pairing $\thetadual : H_A^\ast \otimes H_A^\ast \to A(-1)$ will be
regarded as a $\GSp(H)$-equivariant map $\thetadual : A(1) \to \Lambda^2 H_A$
or, equivalently, as an element of $\Lambda^2 H_A(-1)$.

\subsection{Monodromy representations}

Suppose that $T$ is a smooth variety over a field $k$ of characteristic zero and
that $f:C\to T$ is a family of smooth projective curves over $T$. Fix a
geometric point $\etabar : \Spec F \to T$ of $T$ and denote the fiber of $C$
over it by $C_\etabar$. For a prime number $\ell$ we set
$$
H_\Zl = H_\et^1(C_\etabar,\Zl(1)).
$$
This is endowed with the cup product pairing $\theta : \Lambda^2 H_\Zl \to
\Zl(1)$, which is unimodular. Let $\xbar$ be a geometric point of $C$ that lies
over $\etabar$.

\begin{lemma}
\label{lem:injective}
If $g\ge 2$ or if $g=1$ and $f:C\to T$ has a section, then the homomorphism
$\pi_1(C_\etabar,\xbar) \to \pi_1(C,\xbar)$ induced by $C_\etabar \to C$ is
injective.
\end{lemma}

\begin{proof}
The result is clear when $f$ has a section. So suppose that $g\ge 2$. Since
$\pi_1(C\otimes\kbar,\xbar) \to \pi_1(C,\xbar)$ is injective, it suffices to
consider the case where $k$ is algebraically closed. We may assume that $k=\C$
and that $\etabar \in T(\C)$ and $\xbar\in C(\C)$. Standard algebraic topology
implies that $\pi_1^\top(C,\xbar)$ acts on $\pi_1^\top(C_\etabar,\xbar)$ and
that the composite
$$
\pi_1^\top(C_\etabar,\xbar) \to \pi_1^\top(C,\xbar) \to
\Aut\big(\pi_1^\top(C_\etabar,\xbar)\big)
$$
is the conjugation action. Taking profinite completions, it follows that there
is a continuous action of $\pi_1(C,\xbar)$ on $\pi_1(C_\etabar,\xbar)$ and that
$$
\pi_1(C_\etabar,\xbar) \to \pi_1(C,\xbar) \to
\Aut\big(\pi_1(C_\etabar,\xbar)\big)
$$
is the conjugation action. Since $\pi_1(C_\etabar,\xbar)$ has trivial center
\cite{anderson}, it follows that $\pi_1(C_\etabar,\xbar) \to \pi_1(C,\xbar)$ is
injective.
\end{proof}

Suppose that $g\ge 2$ or that $g=1$ and that $f$ has a section. Since $H_\Zl
\cong \Hom(\pi_1(C_\etabar,\xbar),\Zl(1))$, the conjugation action of
$\pi_1(C,\xbar)$ on $\pi_1(C_\etabar,\xbar)$ induces a natural monodromy action
$$
\rho_\etabar : \pi_1(T,\etabar) \to \GSp(H_\Zl)
$$
such that the diagram
$$
\xymatrix{
\pi_1(T,\etabar) \ar[r]^{\rho_\etabar}\ar[d] & \GSp(H_\Zl) \ar[d]^\tau \cr
G_k \ar[r]^{\chi_\ell} & \Gm(\Zl)
}
$$
commutes, where the left-hand map is the canonical projection, the right-hand
vertical map $\tau$ is the natural surjection, and where $\chi_\ell$ is the
$\ell$-adic cyclotomic character. Often we will extend scalars from $\Zl$ to
$\Ql$.

Denote the projection by $f : C \to T$. For $A=\Zl$ or $\Ql$, set
$$
\H_A := R^1f_\ast A(1)
$$

\begin{proposition}
If $g\ge 2$ or if $g=1$ and $f:C\to T$ has a section, then the monodromy
representation of $\pi_1(T,\etabar)$ on the fiber $H_A$ of $\H_A$ over $\etabar$
is $\rho_\etabar : \pi_1(T,\etabar) \to \Aut H_A$.
\end{proposition}

\begin{proof}
The proof of the topological analogue of the assertion is an exercise in
topology and is left to the reader.\footnote{Reduce to the case of a mapping
torus.} We prove the arithmetic version. Let $K=k(T)$. Fix an algebraic closure
$\Kbar$ of $K$. We may assume that $\etabar = \Spec \Kbar$. Then one has the
diagram
$$
\xymatrix{
1 \ar[r] & \pi_1(C_\etabar,\xbar) \ar[r]\ar@{=}[d] &
\pi_1(C\times_T \Spec K,\xbar) \ar[r]\ar[d] & G_K \ar[d]\ar[r] & 1 \cr
1 \ar[r] & \pi_1(C_\etabar,\xbar) \ar[r] & \pi_1(C,\xbar) \ar[r] &
\pi_1(T,\etabar) \ar[r] & 1
}
$$
in which the vertical maps are surjective. The Galois action $G_K \to \Aut H_A$
is induced by the conjugation action of $\pi_1(C\times_T \Spec K,\xbar)$ on
$\pi_1(C_\etabar,\xbar)$. The result now follows from naturality of the
monodromy action under base change and the commutativity of the diagram.
\end{proof}

\section{Mapping Class Groups and Moduli Stacks of Curves}
\label{sec:mcgs}

Suppose that $g$ and $n$ are non-negative integers satisfying $2g-2+n > 0$. Fix
a closed, oriented surface $S$ of genus $g$ and a finite subset $\x =
\{x_1,\dots,x_n\}$ of $n$ distinct points in $S$. The corresponding mapping
class group will be denoted
$$
\G_{S,\x} = \pi_0\Diff^+ (S,\x),
$$
where $\Diff^+(S,\x)$ denotes the group of orientation preserving
diffeomorphisms of $S$ that fix $\x$ pointwise. By the classification of
surfaces, the diffeomorphism class of $(S,\x)$ depends only on $(g,n)$.
Consequently, the group $\G_{S,\x}$ depends only on the pair $(g,n)$. It will be
denoted by $\G_{g,n}$ when we have no particular marked surface $(S,\x)$ in
mind.

Denote the moduli stack over $\Spec k$ of smooth projective curves of genus $g$
with $n$ distinct marked points by $\M_{g,n/k}$. The corresponding complex
analytic orbifold will be denoted by $\M_{g,n}^\an$. Recall that
$\pi_1^\orb(\M_{g,n},\eta)$ denotes the orbifold fundamental group of
$\M_{g,n}^\an$.

If $(C,\x)$ is an $n$-pointed smooth complex projective curve, there is
a natural isomorphism
$$
\pi_1^\orb(\M_{g,n},[C,\x]) \cong \G_{C,\x}.
$$
For each geometric point $\etabar$ of $\M_{g,n/k}$ there is therefore an
isomorphism
$$
\pi_1(\M_{g,n/\kbar},\etabar) \cong \G_{g,n}^\wedge,
$$
which is well-defined up to inner automorphisms, and an exact sequence
$$
1 \to \G_{g,n}^\wedge \to \pi_1(\M_{g,n/k},\etabar) \to G_k \to 1.
$$

The Lefschetz trace formula implies that when $n>2g+2$, the automorphism group
of each $n$-marked, genus $g$ curve is trivial. In this case, $\M_{g,n/k}$ is a
quasi-projective $k$-scheme, \cite{knudsen}. When $n\le 2g+2$, we can replace
$\M_{g,n}$ by $\Mo_{g,n}$, the locus of $n$-pointed genus $g$ curves in
$\M_{g,n}$ with trivial automorphism group. Since the locus of $n$-pointed
curves in $\M_{g,n}$ with a non-trivial automorphism is a closed subscheme of
$\M_{g,n}$, Knudsen's result \cite{knudsen} implies that
$\Mo_{g,n}$ is a smooth $k$-scheme.

\begin{proposition}
\label{prop:autos}
If $g\ge 2$, then each component of the complement of $\Mo_{g,n}$ in $\M_{g,n}$
has codimension $\ge g+n-2$. Consequently, if $g+n \ge 4$, the inclusion
$\Mo_{g,n} \to \M_{g,n}$ induces an isomorphism $\pi_1^\top(\Mo_{g,n},x) \cong
\G_{g,n}$ of fundamental groups.
\end{proposition}

\begin{proof}[Sketch of Proof]
We will prove the case $n=0$.  The case $n>0$ is an immediate consequence as
marking points on a curve $C$ cannot create new automorphisms, only kill them.
First note that if a curve has a non-trivial automorphism, it has one of prime
order. Suppose that $S$ is a compact oriented surface of genus $g$ and that $G$
is a finite cyclic group of prime order $p$ that acts effectively on $S$ and
preserves the orientation. Since the stabilizer of each point is cyclic, $G\bs
S$ is a compact surface and the projection $\pi:S \to G\bs S$ is a branched
covering. Let $h$ be the genus of $G\bs S$.

The Riemann-Hurwicz formula implies that
\begin{equation}
\label{eqn:rh}
h-1 = \frac{g-1}{p} - \frac{1}{2}\sum_{y \in A} \Big(1-\frac{1}{p}\Big)
\end{equation}
where $A$ is the set of critical values of $\pi$. Let $a=|A|$. Teichm\"uller
theory implies that the locus of curves in $\M_g$ where $G\subseteq \Aut C$ and
the action of $G$ on $C$ is conjugate to the given action of $G$ on $S$ is the
image of a morphism $\M_{h,a}\to\M_g$. This locus has codimension
$$
\dim \M_g - \dim \M_{h,a} = (3g-3)-(3h-3+a).
$$
as $g\ge 2$ implies that $(h,a)$ cannot be $(1,0)$. Equation~(\ref{eqn:rh})
implies that the codimension of this locus is
$$
\bigg(3g-3 + \Big(\frac{3}{2}-\frac{1}{1-p^{-1}}\Big)a\bigg)(1-1/p).
$$
When $p=2$ the codimension is $(3g-3 - a/2)/2$, which is $\ge g-2$ as $a\le
2g+2$ with equality if and only if $a=2g+2$ and $h=0$, in which case $C$ is
hyperelliptic. When $p\ge 3$, then $3/2 \ge (1-p^{-1})^{-1}$. In this case,
$$
\codim \ge (3g-3)(1-p^{-1}) \ge 2g-2
$$
with equality if and only if $a=0$ and $p=3$.
\end{proof}

\subsection{Level structures}

\subsubsection{In topology}
Suppose that $S$ is a compact oriented reference surface of genus $g$. Set $H_\Z
= H_1(S,\Z)$ endowed with its intersection pairing. For a positive integer $m$,
the level $m$ subgroup $\Sp(H_\Z)[m]$ of $\Sp(H_\Z)$ is defined to be the kernel
of the quotient homomorphism
$$
\Sp(H_\Z) \to \Sp(H_{\Z/m\Z}).
$$
The group $\Sp(H_\Z)[m]$ is torsion free for all $m\ge 3$.

The action of $\Diff^+(S,\x)$ on $S$ induces a surjective homomorphism
$$
\rho : \G_{S,\x} \to \Sp(H_\Z).
$$ 
Define the level $m$ subgroup of $\G_{g,n}$ to be the kernel of the composite:
$$
\G_{g,n} \to \G_g \to \Sp(H_{\Z/m\Z}).
$$
Denote it by $\G_{g,n}[m]$. It is torsion free when $m\ge 3$.

\subsubsection{In arithmetic}
\label{sec:level}
Suppose that $m$ is a positive integer and that $k$ is a field of characteristic
zero that contains all $m$th roots of unity $\bmu_m(\kbar)$. Fix an isomorphism
$\psi : \Hom(\bmu_m(\kbar),\Q/\Z) \to \Z/m\Z$. Such isomorphisms correspond to
choices of a primitive $m$th root of unity. A level $m$ structure on a smooth
projective curve $C$ of genus $g$ over $k$ is an isomorphism
$$
\phi : \Het^1(C\otimes_k\kbar,\Z/m\Z) \to (\Z/m\Z)^{2g}
$$
such that the diagram
$$
\xymatrix{
\Het^1(C\otimes_k\kbar,\Z/m\Z)^{\otimes 2} \ar[r]^{\text{cup}}
\ar[d]_{\phi^{\otimes 2}} &
\Hom(\bmu_m(\kbar),\Q/\Z) \ar[d]_\psi \cr
(\Z/m\Z)^{2g} \otimes (\Z/m\Z)^{2g} \ar[r] & \Z/m\Z
}
$$
commutes, where the bottom row is the standard symplectic inner product on
$(\Z/m\Z)^{2g}$. 

When considering the moduli of curves with a level $m$ structure over $k$, we
will assume, as above, that $k$ contains the $m$th roots of unity
$\bmu_m(\kbar)$ and that the isomorphism $\psi$ is specified. This ensures that
the moduli stack $\M_{g,n/k}[m]$ is geometrically connected. It is a principal
$\Sp_g(\Z/m\Z)$-bundle over $\M_{g,n/k}$.  When $m\ge 3$, $\M_{g,n}[m]$ is a
smooth quasi-projective variety. This follows from \cite[Thm.~1.10]{git} using
the GIT setup from \cite[\S2]{geemen-oort}.

When $k$ is an algebraically closed subfield of $\C$, there is a natural
conjugacy class of isomorphisms $\pi_1^\orb(\M_{g,n}[m],\ast) \cong
\G_{g,n}[m]$.

\section{Relative and Weighted Completion of Profinite Groups}

This section is a terse review of the relative and weighted completions of
profinite groups. The reader is referred to the surveys
\cite{hain-matsumoto:survey,hain:morita}, and the articles
\cite{hain:comp,hain:torelli,hain-matsumoto:weighted} for a more detailed
discussion and details.

\subsection{Relative completion of a discrete group} Suppose that $\G$ is a
discrete group and that $R$ is a reductive $F$-group, where $F$ is a field of
characteristic zero (the {\em coefficient field}). The completion of $\G$ with
respect to a Zariski dense representation $\rho : \G \to R(F)$ is a proalgebraic
$F$-group $\cG$, which is an extension
$$
1 \to \U \to \cG \to R \to 1
$$
of $R$ by a prounipotent group $\U$, and a Zariski dense representation
$\rhotilde : \G \to \cG(F)$ whose composition with $\cG(F) \to R(F)$ is $\rho$.
It is characterized by the following universal mapping property: if $G$ is a
proalgebraic $F$-group that is an extension
$$
1 \to U \to G \to R \to 1
$$
of $R$ by a unipotent $F$-group $U$, and if $\phi : \G \to G(F)$ is a
homomorphism whose composition with $G(F) \to R(F)$ is $\rho$, then there is a
homomorphism $\cG \to G$ of $F$-groups that commutes with the projections to $R$
and such that $\phi$ is the composite $\G \to \cG(F) \to G(F)$.

When $R$ is the trivial group $\cG=\U$, which is a prounipotent group. In this
case, $\U$ and $\G \to \U(F)$ comprise the {\em unipotent completion} of $\G$
over $F$. It will be denoted by $\G^\un_{/F}$.

A more interesting example is where $\G$ is the mapping class group $\G_{g,n}$,
where $2g-2+n>0$ and $g\ge 1$. In this case, $\G_{g,n/F}^\un$ is trivial. A
better choice is to take $R=\Sp(H_\Q)$, where $H_A = H_1(S,A)$ is the first
homology of the compact reference surface $S$ and $\rho : \G_{g,n} \to
\Sp(H_\Q)$ to be the representation of the mapping class group on the first
homology of the surface. Since this has image $\Sp(H_\Z)$, it is Zariski dense.
Denote the completion of $\G_{g,n}$ with respect to $\rho$ by $\cG_{g,n}^\geom$
and its prounipotent radical by $\U_{g,n}^\geom$. These are proalgebraic
$\Q$-groups.

One can restrict the homomorphism $\G_{g,n} \to \cG_{g,n}^\geom(\Q)$ to the
level $m$ subgroup.

\begin{proposition}[{\cite[Prop.~3.3]{hain:torelli}}]
\label{prop:level_indep}
If $g\ge 3$, then for all $m\ge 1$, the group $\cG_{g,n}^\geom$ and the
homomorphism $\G_{g,n}[m]\hookrightarrow \G_{g,n}\to \cG_{g,n}^\geom(\Q)$ is the
completion of $\G_{g,n}[m]$ relative to the natural homomorphism $\G_{g,n}[m]\to
\Sp(H_\Q)$.
\end{proposition}

We have been vague about the role played by the coefficient field $F$. This is
because relative completion behaves well under base change.

\begin{theorem}[\cite{hain-matsumoto:mcgs}]
\label{thm:base_change}
Assume the notation of this subsection. If $E$ is an extension field of $F$ and
if the image of $\rho : \G \to R(E)$ is Zariski dense in $R\otimes_F E$, then
$\cG\otimes_F E$ and $\rho_E : \G \to R(E)= (R\otimes_F E)(E)$ comprise the
completion of $\G$ relative to $\rho_E$.
\end{theorem}

This is proved directly for relative completions of mapping class groups in
\cite[Thm.~3.1]{hain:torelli}. It implies that the completion of $\G_{g,n}[m]$
relative to $\rho : \G_{g,n}[m] \to \Sp(H_\Ql)$ is $\cG_{g,n}^\geom\otimes_\Q
\Ql$.

\subsection{Continuous relative completion of a profinite group} This is the
profinite analogue of relative completion of a discrete group. In this case, we
take the coefficient field $F$ to be the field $\Ql$ for some prime number
$\ell$. Suppose that $\G$ is a profinite group, that $R$ is a reductive
$\Ql$-group, and that $\rho : \G \to R(\Ql)$ is a continuous, Zariski dense
representation. The continuous completion of $\G$ with respect to $\rho$ is a
proalgebraic $\Ql$-group $\cG$ that is an extension
$$
1 \to \U \to \cG \to R \to 1
$$
of $R$ by a prounipotent $\Ql$-group $\U$ and a continuous\footnote{A
homomorphism $\phi : \G \to G(\Ql)$ from a profinite group to the rational
points of a proalgebraic $\Ql$-group $G$ is defined to be continuous if it is
the inverse limit of continuous homomorphisms $\G \to G_\alpha(\Ql)$, where each
$G_\alpha$ is an algebraic $\Ql$-group and $G=\varprojlim G_\alpha$.} Zariski
dense representation $\rhotilde : \G \to \cG(\Ql)$ whose composition with
$\cG(\Ql) \to R(\Ql)$ is $\rho$. It is characterized by a universal mapping
property that is similar to the one that characterizes the relative completion
of a discrete group. The only difference is that all homomorphisms are required
to be continuous in the $\ell$-adic profinite case.

When $R$ is the trivial group, this reduces to the continuous unipotent
completion of a profinite group, \cite[A.2]{hain-matsumoto:weighted}.

The continuous relative completion of the profinite completion of a discrete
group can be computed from the relative completion of the discrete group. Denote
the profinite completion of the discrete group $\G$ by $\G^\wedge$. We can view
$\G$ as a topological group by defining the neighbourhoods of the identity to be
the finite index normal subgroups of $\G$.

\begin{theorem}[\cite{hain-matsumoto:weighted,hain-matsumoto:mcgs}]
\label{thm:comparison}
Suppose that $\G$ is a discrete group, $R$ is a reductive $\Ql$-group, and that
$\rho : \G \to R(\Ql)$ is a continuous, Zariski dense representation. Let
$\rho_\ell : \G^\wedge \to R(\Ql)$ be the continuous extension of $\rho$ to
$\G^\wedge$. If $\cG$ and $\rhotilde : \G \to \cG(\Ql)$ is the completion of
$\G$ relative to $\rho$, then:
\begin{enumerate}

\item $\rhotilde$ is continuous and thus induces a continuous homomorphism
$\rhohat_\ell : \G^\wedge \to \cG(\Ql)$;

\item $\cG$ and $\rhohat_\ell$ is the continuous relative completion of
$\G^\wedge$ with respect to $\rho_\ell$.

\end{enumerate}
\end{theorem}

In particular, the continuous completion of $\G_{g,n}^\wedge[m]$ with respect to
the standard representation $\rho_\ell : \G_{g,n}^\wedge[m] \to \Sp(H_\Ql)$ is
$\cG_{g,n}^\geom\otimes_\Q\Ql$. It also implies that the continuous unipotent
completion of the profinite completion $\pi^\wedge$ of the topological
fundamental group $\pi$ of and $n$-punctured, genus $g$ surface is
$\pi^\un_{/\Ql}$.

\subsection{Negatively weighted representations}

Suppose that $F$ is a field of characteristic zero. Denote $\Gm_{/F}$ by $\Gm$.
First suppose that $R$ is a reductive $F$-group and that $\w : \Gm\to R$ is a
central cocharacter. If $V$ is an irreducible $F$-rational representation of
$R$, then the restriction of $V$ to $\Gm$ via $\w$ is isotypical. That is, there
is an integer $w(V)$ such that $\Gm$ acts on $V$ via the $w(V)$th power of its
defining representation. We will call $w(V)$ the {\em weight} of $V$ (with
respect to $\w$) and say that $V$ has negative weight when $w(V)<0$. More
generally, a finite dimensional representation of $R$ is negatively weighted
when each of its irreducible components is.

We will say that an extension
$$
1 \to U \to G \to R \to 1
$$
of affine algebraic $F$-groups, where $U$ is unipotent, is {\em negatively
weighted} (with respect to $\w$) when the abelianization $H_1(U)$ of $U$ is a
negatively weighted $R$-module.  An extension of $R$ by a prounipotent $F$-group
$U$ is negatively weighted if it is the inverse limit of negatively weighted
extensions of $R$ by unipotent groups.

\begin{proposition}[{\cite[Thms.~3.9 \& 3.12]{hain-matsumoto:weighted}}]
\label{prop:weight}
Suppose that $R$ is a reductive $F$-group and that $\w : \Gm \to R$ is a central
cocharacter. If $G$ is a proalgebraic group that is a negatively weighted
extension (with respect to $\w$) of $R$ by a prounipotent group, then every
finite dimensional $G$-module $V$ has a natural weight filtration $W_\dot$:
$$
0 = W_N V \subseteq \cdots \subseteq W_{r-1}V \subseteq W_r V
\subseteq \dots \subseteq W_M V = V.
$$
It is characterized by the property that the action of $G$ on
$$
\Gr^W_r V := W_rV/W_{r-1}V
$$
factors through $G \to R$ and is a representation of $R$ of weight $r$. The
weight filtration is preserved by $G$-module homomorphisms and
the functor $\Gr^W_\dot$ on the category of $G$-modules is exact.
\end{proposition}

\begin{remark}
Each choice of a lift $\wtilde : \Gm \to G$ of the central cocharacter $\w : \Gm
\to R$ determines a vector space isomorphism $V \cong \Gr^W_\dot V$ of a
$G$-module $V$ with its associated weight graded space. This isomorphism is
natural in the sense that the diagram of vector spaces
$$
\xymatrix{
V_1 \ar[r]^\phi\ar[d]^\cong & V_2 \ar[d]^\cong \cr
\Gr^W_\dot V_1 \ar[r]^{\Gr^W_\dot\phi} & \Gr^W_\dot V_2
}
$$
commutes for all $G$-module homomorphism $\phi : V_1 \to V_2$.  Although the
isomorphism $V\cong \Gr^W_\dot V$ is  natural, it is not canonical as it depends
upon the choice of the lift $\wtilde$.
\end{remark}

Denote the Lie algebras of $G$, $R$ and $U$ by $\g$, $\r$ and $\u$,
respectively.

\begin{proposition}[{\cite[Prop.~4.5]{hain-matsumoto:weighted}}]
\label{prop:presentation}
If $G$ is a proalgebraic $F$-group that is a negatively weighted extension of
$R$ by a prounipotent group $U$, then the Lie algebra $\g$ of $G$ and the Lie
algebra $\u$ of $U$ have natural weight filtrations $W_\dot$ that satisfy
$$
\g = W_0\g,\quad \u = W_{-1}\g,\quad \r = \Gr^W_0 \g.
$$
In particular, each choice of a lift $\wtilde:\Gm\to G$ of the central
cocharacter $\w:\Gm\to R$ determines Lie algebra isomorphisms
$$
\g \cong \prod_{r<0} \Gr^W_r \g \text{ and } \u \cong \prod_{r<0} \Gr^W_r \u.
$$
\end{proposition}

A pronilpotent Lie algebra $\n$ in the category of $G$-modules is the inverse
system $(\n_\alpha)_\alpha$ of nilpotent Lie algebras in the category of
$G$-modules. The continuous cohomology of a pronilpotent Lie algebra $\n$ is
defined by
$$
H^\dot(\n) = \varinjlim H^\dot(\n_\alpha)
$$
and its homology by
$$
H_\dot(\n) = \varprojlim H_\dot(\n_\alpha).
$$
These are ind- and pro-objects, respectively, of the category of $G$-modules,
and thus have natural weight filtrations. The later can be regarded as a
topological vector space under the inverse limit topology. There are natural
$G$-module isomorphisms
$$
H^\dot(\n) \cong \Homcts(H_\dot(\n),F) \text{ and }
H_\dot(\n) \cong \Hom(H^\dot(\n),F).
$$
Exactness of $\Gr^W_\dot$ implies that homology and cohomology commute with
$\Gr^W_\dot$. That is, there are canonical $R$-module isomorphisms
$$
H^\dot(\Gr^W_\dot\n) \cong \Gr^W_\dot H^\dot(\n)
\text{ and }
H_\dot(\Gr^W_\dot\n) \cong \Gr^W_\dot H_\dot(\n).
$$

The following useful result is an analogue of Sullivan's observation
\cite{sullivan}.

\begin{lemma}
\label{lem:ses}
Suppose that $G$ is a proalgebraic $F$-group that is a negatively weighted
extension of $R$ by a prounipotent group $U$ and that $\n$ is a pronilpotent Lie
algebra in the category of $G$-modules. If $H_1(\n)$ has negative weights (i.e.,
$H_1(\n) = W_{-1}H_1(\n)$), then there is an $R$-module isomorphism $\Gr^W_{-1}
\n = \Gr^W_{-1}H_1(\n)$ and an exact sequence
$$
0 \to \Gr^W_{2}H^1(\n) \to (\Gr^W_{-2}\n)^\ast
\overset{\bracket^\ast}{\longrightarrow}
\Lambda^2\Gr^W_{1}H^1(\n) \overset{\cupp}{\longrightarrow} \Gr^W_{2}H^2(\n)
\to 0
$$
in the category of ind-$R$-modules.
\end{lemma}

\begin{proof}
Denote the standard cochain complex of $\n$ by $C^\dot(\n)$, where
$$
C^s(\n) = \Hom_\cts(\Lambda^s \n,F).
$$
Since $\n = W_{-1}\n$, this is a differential graded
$F$-algebra in the category of ind-$R$-modules. The exactness of $\Gr^W_\dot$
implies that there are natural ind-$R$-module isomorphisms
$$
\Gr^W_\dot H^\dot(\n) \cong \Gr^W_\dot H^\dot(C^\dot(\n))
\cong H^\dot(\Gr^W_\dot C^\dot(\n)).
$$
Since $H_1(\n)$ is negatively weighted, $\Gr^W_{1}C^\dot(\n)$ is the complex
$\Hom(\Gr^W_{-1} \n,F)$, which implies that $\Gr^W_{1}H^1(\n) = (\Gr^W_{-1}
\n)^\ast$. Taking duals we see that
$$
W_{-1}H_1(\n)= \big(W_1H^1(\n)\big)^\ast = W_{-1}\n
$$
The complex $\Gr^W_{-2}C^\dot(\n)$ is the dual of
$$
\xymatrix{\Lambda^2 \Gr^W_{-1}\n \ar[r]^\bracket & \Gr^W_{-2}\n.}
$$
This and the isomorphism $\Gr^W_{1}H^1(\n) = (\Gr^W_{-1} \n)^\ast$ imply the
exactness of
$$
0 \to \Gr^W_{2}H^1(\n) \to (\Gr^W_{-2}\n)^\ast
\overset{\bracket^\ast}{\longrightarrow}
\Lambda^2\Gr^W_{1}H^1(\n) \overset{\cupp}{\longrightarrow}
\Gr^W_{2}H^2(\n) \to 0.
$$
\end{proof}

\subsection{Weighted completion of a profinite group} Weighted completion of a
profinite group $\G$ is a variant of continuous relative completion. It is a
useful tool as it can be used to define weight filtrations with strong exactness
properties on the representations of $\G$ that factor through its weighted
completion.

Suppose that $\G$ is a profinite group, $R$ is a reductive $\Ql$-group endowed
with a non-trivial, central cocharacter $\w : \Gm_{/\Ql} \to R$, and that $\rho
: \G \to R(\Ql)$ is a continuous, Zariski dense representation. 

The weighted completion of $\G$ with respect to $\rho$ is a proalgebraic
$\Ql$-group $\cG$ that is a negatively weighted extension
$$
1 \to \U \to \cG \to R \to 1
$$
of $R$ by a prounipotent $\Ql$-group $\U$ and a continuous, Zariski dense
representation $\rhotilde : \G \to \cG(\Ql)$ whose composition with $\cG(\Ql)
\to R(\Ql)$ is $\rho$. It is universal with respect to continuous homomorphisms
of $\G$ to negatively weighted extensions of $R$ that lift $\rho$.

Proposition~\ref{prop:weight} implies that every object of the category of
finite dimensional $\cG$-modules has a natural weight filtration and that the
functor $\Gr^W_\dot$ to the category of graded vector spaces is exact.

\subsection{Presentations}
\label{sec:presentations}

Suppose that $F$ is a field of characteristic zero and that $R$ is a reductive
group $F$-group. A generalization of Levi's Theorem ensures that every
extension 
$$
1 \to \U \to \cG \to R \to 1
$$
of $R$ by a prounipotent group splits and that any two such splittings are
conjugate by an element of $\U$. Such a splitting induces an isomorphism
$$
\cG \cong R \ltimes \U
$$
that commutes with the projections to $R$, where the action of $R$ on $\U$ is
determined by the splitting. The prounipotent radical $\U$ with its $R$-action
is determined by its Lie algebra $\u$ and the action of $R$ on it. Thus, to give
a presentation of $\cG$, it suffices to give a presentation of $\u$ in the
category of $R$-modules.

By standard arguments (cf. \cite[\S 3.1]{hain:morita}), $\u$ has a minimal
presentation of the form
$$
\u \cong \L(H_1(\u))^\wedge/(\im \psi)
$$
in the category of $R$-modules, where $\L(V)$ denotes the free Lie algebra
generated by the vector space $V$, and where $\psi : H_2(\u) \to
[\L(H_1(\u))^\wedge,\L(H_1(\u))^\wedge]$ is an injection such that the composite
$$
H_2(\u) \to [\L(H_1(\u))^\wedge,\L(H_1(\u))^\wedge] \to \Lambda^2 H_1(\u)
$$
is dual to the cup product $\Lambda^2 H^1(\u) \to H^2(\u)$.

Suppose now that $R$ is a reductive $\Ql$-group, and that $\w : \Gm \to R$ is a
central cocharacter. Suppose that $\G$ is a profinite group and that $\rho : \G
\to R(\Ql)$ is a continuous, Zariski dense representation. Denote the Lie
algebra of the weighted completion of $\G$ with respect to $\w$ and $\rho$ by
$\g$ and that of its prounipotent radical by $\u$.

\begin{proposition}[{\cite[\S 4.3]{hain-matsumoto:weighted}}]
\label{prop:homology}
If $V$ is a finite dimensional $R$-module of weight $r$, then there are natural
isomorphisms
$$
\Hom_R(H_1(\u),V) \cong \Hom_R(\Gr^W_r H_1(\u),V) \cong
\begin{cases}
H^1(\G,V) & r < 0, \cr
0 & r \ge 0
\end{cases}
$$
and a natural injection $\Hom_R(H_2(\u),V) \hookrightarrow H^2(\G,V)$ whenever
$r<0$. If $r > -2$, then $\Hom_R(H_2(\u),V) = 0$.
\end{proposition}

\begin{remark}
\label{rem:invariants}
An elementary spectral sequence argument implies that when $R$ is connected
$$
H^j(\g,V) \cong H^0(\r,H^j(\u)\otimes V) \cong \Hom_R(H_j(\u),V)
$$
for all finite dimensional $R$-modules $V$.
\end{remark}

The following example summarizes results from \cite{hain-matsumoto:weighted} and
will be used later in the paper.

\begin{example}
\label{ex:fields}
Suppose that $k$ is a field for which the image of the $\ell$-adic cyclotomic
character $\chi_\ell : G_k \to \Zlx$ is infinite or, equivalently, that
$\chi_\ell : G_k \to \Gm(\Ql)$ is Zariski dense. Define a central
cocharacter\footnote{This is the ``right choice'' because the standard
representation of $\Gm$ corresponds to $\Ql(1)$, which has weight $-2$.} $\w :
\Gm \to \Gm$ by $z\mapsto z^{-2}$. Denote the weighted completion of $G_k$ with
respect to $\chi_\ell$ and $\w$ by $\A_k$. It is a split extension
$$
1 \to \K_k \to \A_k \to \Gm \to 1
$$
of $\Ql$-groups, where $\K_k$ is prounipotent. Denote the Lie algebras of $\A_k$
and $\K_k$ by $\a_k$ and $\k_k$, respectively. Proposition~\ref{prop:homology}
implies that
$$
\Gr^W_r H_1(\k_k) =
\begin{cases}
H^1(G_k,\Ql(s))^\ast\otimes\Ql(s) & r= - 2s < 0,\cr
0 & \text{ otherwise.}
\end{cases}
$$
This implies that $\Gr^W_{-1} \a_k =0$ and that
$$
\Gr^W_{-2} \a_k \cong H^1(G_k,\Ql(1))^\ast\otimes \Ql(1).
$$
If $k$ is a number field, a finite extension of $\Qp$ or a finite field of
characteristic $p$, where $p\neq \ell$, the image of the $\ell$-adic cyclotomic
character $\chi_\ell : G_k \to \Zlx$ is infinite. When $k$ is finite, $\A_k =
\Gm$. When $k$ is a number field, $\Gr^W_\dot\k_k$ is a free Lie
algebra, \cite{hain-matsumoto:weighted}.

\end{example}

\subsection{An exactness criterion}

In general, relative and weighted completion are only right exact. In this
section we give a criterion for when weighted completion is exact. Suppose that
$$
\xymatrix{
1 \ar[r] & \pi \ar[r] & \Ghat \ar[r]^\phi & \G \ar[r] & 1
}
$$
is an exact sequence of profinite groups. Suppose that $R$ is a reductive
$\Ql$-group with a central cocharacter $\w : \Gm \to R$. and that $\rho : \G \to
R(\Ql)$ is a continuous, Zariski dense representation. Denote the weighted
completion of $\G$ and $\Ghat$ with respect to $\rho$ and $\rho\circ\phi$ by
$\cG$ and $\cGhat$, respectively. Denote the unipotent completion of $\pi$
by $\cP$.

\begin{proposition}
\label{prop:left_exactness}
Suppose that the $\G$-action on $H_1(\pi)$ induces a $\cG$-action on $H_1(\cP)$
and that the corresponding weight filtration on $H_1(\cP)$ has finite
dimensional graded quotients which vanish in weights $r\ge 0$. If $\cP$ has
trivial center, then the sequence of completions
$$
\xymatrix{
1 \ar[r] & \cP \ar[r] & \cGhat \ar[r] & \cG \ar[r] & 1
}
$$
is exact.
\end{proposition}

\begin{proof}
Denote the Lie algebra of $\cP$ by $\p$. Since the action of $\cG$ on $H_1(\cP)$
is induced by the conjugation action of $\Ghat$ on $\pi$, it follows that the
conjugation action of $\Ghat$ on $\pi$ induces an action of $\cGhat$ on
$H_1(\cP)=H_1(\p)$. Since this action is negatively weighted, it follows that
the action of $\Ghat$ on $\p$ induced by conjugation induces a negatively
weighted action of $\cGhat$ on $\p$.

Right exactness of weighted completion implies that the sequence
$$
\cP \to \cGhat \to \cG \to 1
$$
is exact. Since $\p$ is negatively weighted and since each weight graded
quotient of $H_1(\p)$ is finite dimensional, each $\p/W_r\p$ is finite
dimensional and the group $\Aut_W\p$ of $W_\dot$ preserving automorphisms of
$\p$ is proalgebraic as it is the inverse limit of the $\Aut_W(\p/W_r)$. The
conjugation action of $\Ghat$ on $\pi$ induces a homomorphism $\cGhat \to
\Aut_W\p$ such that the diagram
$$
\xymatrix{
\pi \ar[r]\ar[d] & \Ghat \ar[d]  \ar[dr] \cr
\cP(\Ql) \ar[r] & \cGhat(\Ql) \ar[r] & (\Aut\p)(\Ql)
}
$$
commutes. Since $\cP$ has trivial center, the adjoint action $\cP \to \Aut\p$ is
injective, which implies that $\cP \to \cGhat$ is injective.
\end{proof}

\section{Weighted Completion and Families of Curves}

Suppose that $k$ is a field of characteristic zero and that $C\to T$ is a family
of smooth projective curves of genus $g\ge 2$ over a smooth, geometrically
connected $k$-scheme $T$, or the generic point of a such a $k$-scheme. Let
$\etabar_T$ be a geometric point of $T$ and $\etabar_C$ a geometric point of $C$
that lies over $\etabar_T$. Denote the fiber of $C$ over $\etabar_T$ by
$C_\etabar$. Fix a prime number $\ell$. Set
$$
H = H_\Ql = \Het^1(C_\etabar,\Ql(1))
$$
and let $R$ be the Zariski closure of the image of the monodromy representation
$$
\pi_1(T,\etabar_T) \to \GSp(H_\Ql).
$$
Assume that $R$ contains the homotheties.\footnote{A theorem of Bogomolov
\cite{bogomolov} implies that this is always the case when $k$ is a number
field.} Then one has the central cocharacter
\begin{equation}
\label{eqn:standard}
\w : \Gm \to R\qquad z \mapsto z^{-1}\id_H,
\end{equation}
which shall call the {\em standard cocharacter}.\footnote{This is the
``correct'' definition as $H$ has weight $-1$.}

The surjectivity of $\pi_1(C,\etabar_C) \to \pi_1(T,\etabar_T)$ implies that the
monodromy representation $\pi_1(C,\etabar_C) \to R(\Ql)$ is also Zariski dense.
Denote the weighted completions of $\pi_1(C,\etabar_C)$ and $\pi_1(T,\etabar_T)$
with respect to $\w$ and their monodromy representations to $R$ by $\cG_C$ and
$\cG_T$, respectively; denote their prounipotent radicals by $\U_C$ and $\U_T$,
respectively.

The surjectivity of $\pi_1(C\otimes_k \kbar,\etabar_C) \to \pi_1(T\otimes_k
\kbar,\etabar_T)$ implies that their images in $R(\Ql)$ are equal. Their common
Zariski closure is a reductive subgroup $R^\geom$ of $R$. Denote the continuous
relative completions of $\pi_1(C\otimes_k \kbar,\etabar_C)$ and
$\pi_1(T\otimes_k \kbar,\etabar_T)$ with respect to their homomorphisms to
$R^\geom(\Ql)$ by $\cG_C^\geom$ and $\cG_T^\geom$, and their prounipotent
radicals by $\U_C^\geom$ and $\U_T^\geom$.

Finally, set $\pi = \pi_1(C_\etabar)$. Denote by $\pi^\un$ its continuous
unipotent completion $\pi_{/\Ql}^\un$ over $\Ql$.

\begin{proposition}
\label{prop:section}
With the notation above:
\begin{enumerate}

\item  There are exact sequences
$$
1 \to \pi^\un \to \cG_C \to \cG_T \to 1 \text{ and }
1 \to \pi^\un \to \cG_C^\geom \to \cG_T^\geom \to 1
$$
of proalgebraic $\Ql$-groups such that the diagram
$$
\xymatrix@R=10pt@C=0pt{
1 \ar[rr] && \pi_1(C_\etabar) \ar[rr]\ar'[d][dd]\ar[dr] &&
\pi_1(C\otimes_k \kbar,\etabar) \ar[rr]\ar'[d][dd]\ar[dr] &&
\pi_1(T\otimes_k \kbar,\etabar) \ar[rr]\ar'[d][dd]\ar[dr] && 1 \cr
& 1 \ar[rr] && \pi_1(C_\etabar) \ar[rr]\ar[dd] &&
\pi_1(C,\etabar) \ar[rr]\ar[dd] &&
\pi_1(T,\etabar) \ar[rr]\ar[dd] && 1 \cr
1 \ar[rr] && \pi^\un(\Ql) \ar'[r][rr]\ar[dr] &&
\cG_C^\geom(\Ql) \ar'[r][rr]\ar[dr] &&
\cG_T^\geom(\Ql) \ar'[r][rr]\ar[dr] && 1\cr
& 1 \ar[rr] && \pi^\un(\Ql) \ar[rr] && \cG_C(\Ql) \ar[rr] &&
\cG_T(\Ql) \ar[rr] && 1
}
$$
commutes.

\item Every section $s$ of $\pi_1(C,\etabar) \to \pi_1(T,\etabar)$ induces a
sections $\sigma$ of $\cG_C \to \cG_T$ and $\sigma^\geom$ of $\cG_C^\geom \to
\cG_T^\geom$ such that the diagram
$$
\xymatrix@R=10pt@C=0pt{
\pi_1(C\otimes_k \kbar,\etabar')\ar[dd]\ar[dr] &&
\pi_1(T\otimes_k \kbar,\etabar)\ar'[d][dd]\ar[ll]_s\ar[dr]\cr
&\pi_1(C,\etabar')\ar[dd] && \pi_1(T,\etabar)\ar[dd]\ar[ll]_(0.6)s\cr
\cG_C^\geom(\Ql)\ar[dr] &&
\cG_T^\geom(\Ql) \ar[dr]\ar'[l]_(0.7){\sigma^\geom}[ll] \cr
&\cG_C(\Ql) && \cG_T(\Ql) \ar[ll]_\sigma
}
$$
commutes.

\end{enumerate}
\end{proposition}

\begin{proof}[Sketch of Proof]
Exactness of the row containing $\cG_C$ follows from
Proposition~\ref{prop:left_exactness} and the fact \cite{hattori-stallings} that
$\pi^\un$ has trivial center as $H_1(\pi^\un)$ is a $\cG_C$ module of weight
$-1$. The exactness of the row containing $\cG_C^\geom$ follows from the right
exactness of relative completion and the fact that the composite $\pi^\un \to
\cG_C^\geom \to \cG_C$ is injective.

Since $\cG_C$ is a negatively weighted extension of $R$, the composition
$\pi_1(T,\xbar) \to \pi_1(C,\xbar') \to \cG_C(\Ql)$ of $\rhohat_C$ with $s$
induces a homomorphism $\sigma:\cG_T \to \cG_C$. The universal mapping property of
$\cG_T$ implies that it is a section.
\end{proof}

Denote the Lie algebras of
$$
R,\ \cG_C,\ \cG_T,\ \U_C,\  \U_T,\ \pi^\un\quad
\text{by}\quad
\r,\ \g_C,\ \g_T,\ \u_C,\ \u_T,\ \p,
$$
respectively. All are representations of $\cG_C$. So, by
Proposition~\ref{prop:presentation}, these all have natural weight
filtrations which are characterized by the property that their $m$th graded
quotient is an $R$-module of weight $m$. Moreover, $\p = W_{-1}\p$ and, for
$X=T$ and $C$
$$
\g_X = W_0 \g_X,\ W_{-1}\g_X = \u_X \text{ and }\Gr^W_0 \g_X \cong \r.
$$

\begin{corollary}
There is an exact sequence
$$
0 \to \Gr^W_\dot \p \to \Gr^W_\dot \g_C \to \Gr^W_\dot \g_T \to 0
$$
of graded Lie algebras in the category of $R$-modules. Each section of
$\pi_1(C,\etabar_C) \to \pi_1(T,\etabar_T)$ induces an $R$-linear section of
$\Gr^W_\dot \g_C \to \Gr^W_\dot \g_T$.
\end{corollary}

\begin{proof}
Proposition~\ref{prop:section} implies that the sequence
$$
0 \to \p \to \g_C \to \g_T \to 0
$$
of Lie algebras is an exact sequence of $\cG_C$-modules. The result follows from
the fact that the functor $\Gr^W_\dot$ is exact on the category of
$\cG_C$-modules. The statement about sections follows similarly.
\end{proof}

\subsection{Weighted completion of Galois groups of function fields}

Suppose that $T$ is a smooth, geometrically connected variety over a field $k$
and that $C \to T$ is a family of smooth projective curves over $T$. Set
$K=k(T)$. Let $\etabar$ be a geometric point of $C$ that lies over the generic
point of $T$.

Denote the restriction of $C$ to the Zariski open subset $U$ of $T$ by $C_U$.
Denote the restriction of $C$ to the generic point $\Spec K$ of $T$ by $C_K$.
The image of
$$
\rho_U : \pi_1(C_U,\etabar) \to \pi_1(U,\etabar) \to \GSp(H_\Ql)
$$
does not depend on $U$. Denote its Zariski closure by $R$. Assume that this
contains the scalar matrices so that the image of the standard cocharacter
(\ref{eqn:standard}), $\w: z\mapsto z^{-1}\id_H$, is defined. Denote the
weighted completion of $\pi_1(C_U,\etabar)$ with respect to $\rho_U :
\pi_1(C_U,\etabar) \to R(\Ql)$ by $\cG_{C/U}$. Denote the weighted completion of
$\pi_1(U,\etabar)$ by $\cG_U$. Then one has the morphism of inverse systems
$$
\{\cG_{C/U}\}_U\to \{\cG_U\}_U.
$$
Denote the weighted completion of $\pi_1(C_K,\etabar)$ with respect to
$\rho_K : \pi_1(C_K,\etabar) \to R(\Ql)$ by $\cG_{C/K}$, and the weighted
completion of $G_K$ with respect to the monodromy representation $G_K \to
\GSp(H)$ by $\cG_K$.

\begin{proposition}
\label{prop:generic_pt}
There are natural isomorphisms
$$
\cG_{C/K} \stackrel{\simeq}{\longrightarrow} \varprojlim_U \cG_{C/U}
\text{ and }
\cG_{K} \stackrel{\simeq}{\longrightarrow} \varprojlim_U \cG_{U}
$$
such that the diagram
$$
\xymatrix{
\pi_1(C_K,\etabar) \ar[d]  \ar[r]^(.45)\simeq &
\varprojlim_U \pi_1(C_U,\etabar) \ar[d]\cr
\cG_{C/K}(\Ql) \ar[r]^(.45)\simeq & \varprojlim_U \cG_{C/U}(\Ql)
}
$$
commutes. Each section of $\pi_1(C_K,\etabar) \to G_K$ induces a section
of $\cG_{C/K}\to \cG_K$.
\end{proposition}

\begin{proof}
This follows from Proposition~\ref{prop:homology} and the fact that
the natural homomorphisms
$$
\varinjlim_U H^\dot(\pi_1(C_U,\etabar),V) \to H^\dot(\pi_1(C_K,\etabar),V),
\quad
\varinjlim_U H^\dot(\pi_1(U,\etabar),V) \to H^\dot(G_K,V)
$$
are isomorphisms for all finite dimensional $R$-modules $V$. The statement
about sections follows from the naturality of weighted completion. 
\end{proof}

\section{Weighted Completion of Arithmetic Mapping Class Groups}
\label{sec:comp_arith_mcgs}

In this section we summarize and extend results from \cite{hain-matsumoto:mcgs}.
Throughout this section, $g$, $n$ and $m$ are non-negative integers satisfying
$2g-2+n>0$, and $m\ge 1$.  We also assume that $k$ is a field of characteristic
zero that contains $\bmu_m(\kbar)$, the $m$th roots of unity, and we fix a
primitive $m$th root of unity so that one has a geometrically connected moduli
stack $\M_{g,n/k}[m]$.

Fix a geometric point $\etabar$ of $\M_{g,n/k}[m]$ and let $C_\etabar$ be the fiber
of the universal curve over $\etabar$. For a prime number $\ell$ and a
$\Zl$-module $A$, set
$$
H_A = \Het^1(C_\etabar,A(1)).
$$
Assume that the image of the $\ell$-adic cyclotomic character $\chi_\ell : G_k
\to \Zlx$ is infinite. This implies that the image of $\chi_\ell : G_k \to
\Gm(\Ql)$ is Zariski dense. Since the diagram
$$
\xymatrix{
1 \ar[r] & \pi_1(\M_{g,n/\kbar}[m],\etabar) \ar[r]\ar[d]^{\rho^\geom} &
\pi_1(\M_{g,n/k}[m],\etabar) \ar[r] \ar[d]^\rho & G_k \ar[d]^{\chi_\ell}\ar[r] & 1 \cr
1 \ar[r] & \Sp(H_\Ql) \ar[r] & \GSp(H_\Ql) \ar[r] & \Gm(\Ql) \ar[r] & 1
}
$$
commutes and since the image of $\rho^\geom$ is $\Sp(H_\Zl)$, which is Zariski
dense in $\Sp(H_\Ql)$, the monodromy representation
$$
\rho : \pi_1(\M_{g,n/k}[m],\etabar) \to \GSp(H_\Ql)
$$
is Zariski dense.

Now take the coefficient field $F$ to be $\Ql$. Denote the weighted completion
of $\pi_1(\M_{g,n/k}[m],\etabar)$ over $\Ql$ with respect to $\rho$ and the
standard cocharacter $\w: z \mapsto z^{-1}\id_H$ by
$$
\cG_{\M_{g,n/k}[m]} \text{ and }
\rhotilde : \pi_1(\M_{g,n/k}[m],\etabar) \to \cG_{\M_{g,n/k}[m]}(\Ql).
$$
Recall that $\cG_{g,n}^\geom$ is the relative completion of
$\pi_1(\M_{g,n/\kbar}[m],\etabar)$ with respect to its standard representation
to $\Sp(H_\Ql)$.

Recall from Example~\ref{ex:fields} that $\A_k$ denotes the weighted completion
of $G_k$ with respect to the $\ell$-adic cyclotomic character $\chi_\ell : G_k
\to \Gm(\Ql)$. The following is an easily proved consequence of the right
exactness of relative and weighted completions.

\begin{proposition}[\cite{hain-matsumoto:mcgs}]
If $2g-2+n > 0$, then applying weighted completion to the two right-hand
columns and relative completion to the left-hand column of diagram
$$
\xymatrix{
1\ar[r] & \pi_1(\M_{g,n/\kbar}[m],\etabar) \ar[d]\ar[r] &
\pi_1(\M_{g,n/k}[m],\etabar)\ar[d]\ar[r] & G_k \ar[r]\ar[d] &
1 \cr
1 \ar[r] & \Sp(H_\Ql) \ar[r] & \GSp(H_\Ql) \ar[r] & \Gm(\Ql) \ar[r] & 1 
}
$$
gives a commutative diagram
$$
\xymatrix{
& \cG_{g,n} \ar[d]\ar[r] & \cG_{\M_{g,n/k}[m]} \ar[d]\ar[r] &
\A_k \ar[r]\ar[d] &
1 \cr
1 \ar[r] & \Sp(H) \ar[r] & \GSp(H) \ar[r] & \Gm \ar[r] & 1 
}
$$
whose rows are exact. 
\end{proposition}

Like the relative completion of mapping class groups, $\cG_{\M_{g,n/k}[m]}$ does
not depend on the level $m$ when $g\ge 3$.

\begin{proposition}
\label{prop:level-indep}
If $k$ is a field of characteristic zero and if $g\ge 3$, then for all $m\ge 1$
the natural homomorphism $\cG_{\M_{g,n/k}[m]} \to \cG_{\M_{g,n/k}}$ is an
isomorphism.
\end{proposition}

\begin{proof} This is a brief sketch. Complete details will appear in
\cite{hain-matsumoto:mcgs}. Set $\M=\M_{g,n}$. Denote the $\ell$-adic local
system over $\M$ that corresponds to the $\GSp(H_\Ql)$-module $V$ by $\V$. The
spectral sequence
$$
H^s(G_k,\Het^t(\M_{\kbar},\V)) \implies \Het^{s+t}(\M,\V)
$$
and the fact that the map $\Het^j(\M_{/\kbar},\V) \to \Het^j(\M_{/\kbar}[m],\V)$
is an isomorphism when $j\le 1$ and injective when $j=2$ implies that
$\Het^j(\M,\V) \to \Het^j(\M_[m],\V)$ is also an isomorphism when $j\le 1$ and
injective when $j=2$. Proposition~\ref{prop:homology} implies that for all
$\GSp(H)$-modules $V$ the columns of the commutative diagram
$$
\xymatrix{
\Hom_{\GSp(H)}(H_j(\u_{\M}),V) \ar[d]\ar[r] &
\Hom_{\GSp(H)}(H_j(\u_{\M[m]}),V) \ar[d] \cr
H^j(\M,\V) \ar[r] & H^j(\M[m],\V)
}
$$
are isomorphisms when $j=1$ and injective when $j=2$. This implies that
$H_j(\u_{\M[m]}) \to H_j(\u_{\M})$ is an isomorphism when $j=1$ and surjective
when $j=2$. The discussion of minimal presentations of pronilpotent Lie algebras
in Section~\ref{sec:presentations} implies that $\u_{\M[m]} \to \u_{\M}$ is an
isomorphism.\footnote{The last step is the analogue of Stalling's criterion
\cite{stallings} for pronilpotent Lie algebras.  It follows directly from the
discussion of minimal presentations of pronilpotent Lie algebras in
Section~\ref{sec:presentations}.}
\end{proof}

\begin{remark}
The map $\cG_{g,n}^\geom \to \cG_{\M_{g,n/k}}$ should be injective when $g\ge
3$. It is not injective when $g=1$. This issue is considered in
\cite{hain-matsumoto:mcgs} where it is related to the crystallinity of the
action of $G_k$ on $\g_{g,n}^\geom$ when $g\ge 2$.
\end{remark}

When the field $k$ is clear from context, we will denote $\cG_{\M_{g,n/k}}$ by
$\cG_{g,n}$ and its prounipotent radical by $\U_{g,n}$. Recall that
$\U_{g,n}^\geom$ denotes the prounipotent radical of $\cG_{g,n}^\geom$. Denote
the Lie algebras of $\U_{g,n}$ and $\U_{g,n}^\geom$ by $\u_{g,n}$ and
$\u_{g,n}^\geom$, respectively. Since $\u_{g,n}$ is a $\cG_{g,n}$-module (via
the adjoint action) it has a natural weight filtration. Although it is not
immediately obvious, $\u_{g,n}^\geom$ also has a natural weight
filtration.

\begin{proposition}[\cite{hain-matsumoto:mcgs}]
If $2g-2+n>0$, then the natural action of $\pi_1(\M_{g,n/k},\etabar)$ on
$\pi_1(\M_{g,n/\kbar},\etabar)$ induces an action of $\cG_{g,n}$ on
$\g_{g,n}^\geom$. That is, $\g_{g,n}^\geom$ and $\u_{g,n}^\geom$ are pro-objects
of the category of $\cG_{g,n}$-modules and thus have natural weight filtrations.
\end{proposition}

\begin{remark}
We will be appealing to computations in \cite{hain:torelli} where a weight
filtration on $\g_{g,n}^\geom$ is constructed using Hodge theory. The weight
filtration on $\g_{g,n}^\geom$ used in this paper is constructed using weighted
completion. In order to apply the computations of \cite{hain:torelli} in this
paper, we need to know that these two constructions of the weight filtration
agree. General results imply that this is the case. However, in the present
situation, a direct argument can be given when $g\ge 3$. The starting
observation is that in both constructions $H_1(\p)$ has weight $-1$. The
exactness of $\Gr^W_\dot$ then implies that $W_{-r}\p$ is the $r$th term of its
lower central series, which implies that the two constructions of the weight
filtration of $\p$ agree. In both constructions, $\u_{g,n}^\geom =
W_{-1}\g_{g,n}^\geom$ for all $n\ge 0$. Thus, to prove equality of the two
constructions, it suffices to show that they agree on $\u_{g,n}^\geom$. When
$g\ge 3$, Johnson's computation of the abelianization of the Torelli group
\cite{johnson:h1} implies that the map $H_1(\u_{g,1}) \to \Gr^W_{-1}\Der\p$ is
an isomorphism. This implies that $H_1(\u_{g,1})$ is pure of weight $-1$ in both
constructions. The injectivity of the natural map
$$
H_1(\u_{g,n}) \to H_1(\u_{g,1})^{\oplus n}
$$
induced by $\M_{g,n}\to(\M_{g,1})^n$ when $n>1$, and the surjectivity of
$H_1(\u_{g,1})\to H_1(\u_g)$ imply that $H_1(\u_{g,n})$ is also pure of weight
$-1$ for all $n\ge 0$. The exactness of $\Gr^W_\dot$ now implies that
$W_{-r}\u_{g,n}$ is the $r$th term of the lower central series of $\u_{g,n}$ for
all $n\ge 0$. The two construction therefore agree.
\end{remark}

\subsection{Variants}

The constructions and results above extend to other moduli spaces. Here we
consider one, namely the $n$th power of the universal curve $\cC_{g/k}^n[m]$
over $\M_{g/k}[m]$. Others will be constructed in subsequent sections, as
needed. As above, one has the homomorphism
$$
\rho : \pi_1(\cC_{g/k}^n[m],\etabar) \to \GSp(H_\Ql)
$$
which is Zariski dense. Denote the weighted completion of
$\pi_1(\cC_{g/k}^n[m],\etabar)$ with respect to $\rho$ and $\w$ by
$\cG_{\cC_{g/k}^n[m]}$. As in the case of $\M_{g,n/k}[m]$, it does not depend on
$m$ when $g\ge 3$. We will henceforth denote it by $\cGhat_{g,n}$ and its
prounipotent radical by $\Uhat_{g,n}$. Denote their Lie algebras by
$\ghat_{g,n}$ and $\uhat_{g,n}$, respectively.

We will denote the completion of $\pi_1(\cC_{g/\kbar}^n[m],\etabar)$ with
respect to the natural homomorphism to $\Sp(H_\Ql)$ by $\cGhat_{g,n}^\geom$. It
too is independent of $m$ when $g\ge 3$.

The inclusion $\M_{g,n/k}[m] \to \cC_{g/k}^n[m]$ induces a homomorphism
$\cG_{g,n} \to \cGhat_{g,n}$.
And one has the commutative diagram
$$
\xymatrix{
\cG_{g,n}^\geom \ar[r]\ar[d] & \cG_{g,n} \ar[r]\ar[d] & \A_k
\ar@{=}[d]\ar[r] & 1 \cr
\cGhat_{g,n}^\geom \ar[r] & \cGhat_{g,n} \ar[r] & \A_k \ar[r] & 1
}
$$
with exact rows.

\subsection{Basepoints}

Observe that $\M_{g,n/k}$ is a Zariski open subset of $\cC_{g/k}^n$. Denote
their common generic point by $\eta$ and let $\etabar$ to be a geometric point
that lies over it. Denote the image of $\etabar$ in $\M_{g/k}$ by $\etabar_o$
and the image of $\etabar$ in $\M_{g,1/k}$ under the $j$th projection
$\M_{g,n/k} \to \M_{g,1/k}$ by $\xbar_j$.

Denote the fiber of the universal curve over $\etabar_o$ by $C$. It has the
$n$-rational points $\xbar_1,\dots,\xbar_n$. The fiber of the projection
$\cC_{g/k}^n \to \M_{g/k}$ over $\etabar_o$ is $C^n$ and has basepoint
$\x_\etabar :=(\xbar_1,\dots,\xbar_n) \in C^n$. The natural map $\cC_{g/k}^n \to
(\M_{g,1/k})^n$ takes the basepoint $\etabar$ of $\cC_{g/k}^n$ to the basepoint
$\x_\etabar$ of $(\M_{g,1/k})^n$.

\subsection{Exactness}

The ``fat diagonal'' $\Delta$ of $C^n$ is the union of the divisors
$\Delta_{ij}$, where
$$
\Delta_{ij} := \{(u_1,\dots,u_n) : u_i = u_j,\ i\neq j\}.
$$
Observe that $\xbar_\eta\in C^n-\Delta$. The maps $\M_{g,n/k} \to \cC_{g/k}^n
\to (\M_{g,1/k})^n$ induce a commutative diagram
$$
\xymatrix{
1 \ar[r] & \pi_1(C^n-\Delta,\x_\etabar) \ar[r]\ar[d] & \pi_1(\M_{g,n/k},\etabar)
\ar[r]\ar[d] & \pi_1(\M_{g/k},\etabar_o) \ar@{=}[d]\ar[r] & 1 \cr
1 \ar[r] & \prod_{j=1}^n \pi_1(C,\xbar_j)\ar@{=}[d] \ar[r] &
\pi_1(\cC_{g/k}^n,\etabar)\ar[d] \ar[r] & \pi_1(\M_{g/k},\etabar_o)
\ar[d]^{\text{diag}} \ar[r] & 1 \cr
1 \ar[r] & \prod_{j=1}^n \pi_1(C,\xbar_j) \ar[r] &
\prod_{j=1}^n \pi_1(\M_{g,1/k},\xbar_j) \ar[r] & \pi_1(\M_{g/k},\etabar_o)^n
\ar[r] & 1
}
$$
with exact rows.\footnote{The exactness of both rows is follows from the fact
that the center of $\pi_1(C^n-\Delta,\x_\etabar)$ is trivial for all $n\ge 1$
when $g\ge 2$. The argument in the profinite case uses \cite{anderson} and is
essentially identical with the argument in the discrete case, which can be found
in \cite[Lem.~4.2.2]{birman}.} Denote $\pi_1(C^n-\Delta,\x_\etabar)$ by
$\pi_{g,n}$ and the Lie algebra of its unipotent completion over $\Ql$ by
$\p_{g,n}$. Denote the Lie algebra of $\pi_1(C,\xbar_j)^\un$ by $\p_j$ and the
Lie algebra of $\pi_1(\M_{g,1/k},\xbar_j)$ by $\g_{g,1}^{(j)}$.

Observe that the weight graded quotients of $\g_{g,n}$ and $\ghat_{g,n}$ are
$S_n\times\GSp(H)$-modules, where the $S_n$-action is induced by action of $S_n$
on $\cC_{g/k}^n$ that permutes the $n$ points. 

\begin{proposition}
\label{prop:exactness}
Suppose that $g\ge 2$. After applying weighted completion to this diagram, we
obtain the commutative diagram
$$
\xymatrix{
0 \ar[r] & \p_{g,n} \ar[r]\ar[d] & \g_{g,n}
\ar[r]\ar[d] & \g_g \ar@{=}[d]\ar[r] & 0 \cr
0 \ar[r] & \bigoplus_{j=1}^n \p_j \ar@{=}[d]\ar[r] &
\ghat_{g,n} \ar[r] \ar[d] & \g_g \ar[d]^{\text{diag}} \ar[r] & 0 \cr
0 \ar[r] & \bigoplus_{j=1}^n \p_j \ar[r] &
\bigoplus_{j=1}^n \g_{g,1}^{(j)} \ar[r] & (\g_g)^n \ar[r] & 0
}
$$
with exact rows. After applying $\Gr^W_\dot$, the rows are exact sequences of
$S_n\times\GSp(H)$-modules. Moreover, the same statements holds with $\g$
replaced by $\g^\geom$ in the diagram.
\end{proposition}

\begin{proof}[Sketch of Proof]
The conjugation action of $\pi_1(\M_{g,n/k},\etabar)$ on
$\pi_1(C^n-\Delta,\x_\etabar)$ induces induces an action of $\cG_{g,n}$ on
$H_1(\p_{g,n})$ which has weights $-1$ and $-2$, and is therefore negatively
weighted. The exactness of the first row now follows from
Proposition~\ref{prop:left_exactness} as $\p_{g,n}$ has trivial center
\cite[p.~201]{hattori-stallings}. The corresponding statement with $\g$ replaced
by $\g^\geom$ follows from the right exactness of relative completion and
because the composite $\p_{g,n} \to \ghat_{g,n}^\geom \to \ghat_{g,n}$ is
injective.

Right exactness of relative and weighted completion implies that the second row
is right exact. To prove that it is exact, we use the map $r_j : \cC_g^n\to
\M_{g,1}$ defined by $r_j : (x_1,\dots,x_n) \to x_j$. It induces a commutative
diagram
$$
\xymatrix{
 & \bigoplus_{j=1}^n \p_j \ar[r]\ar[d]_{pr_j} &
\ghat_{g,n} \ar[r]\ar[d] & \g_g \ar[r]\ar@{=}[d] & 0 \cr
0 \ar[r] & \p_j \ar[r] & \g_{g,1} \ar[r] & \g_g \ar[r] & 0
}
$$
where $pr_j$ denotes projection onto the $j$th factor.
Proposition~\ref{prop:section} implies that the bottom row of this diagram is
exact. It follows that the top row of this diagram (i.e., the second row of the
diagram in the statement) is left exact.

The exactness of the third row is a consequence of the exactness of the first
row of the diagram in the case $n=1$.
\end{proof}

Recall that for a fixed curve $C$, the $\Gr^W_\dot \p_j$ are naturally
isomorphic. Denote their common value by $\Gr^W_\dot \p$.

%\begin{corollary}
%\label{cor:extension}
%If $g\ge 3$, then the sequence
%$$
%\xymatrix{
%0 \ar[r] &  (\Gr^W_\dot\p)^n \ar[r] &
%\Gr^W_\dot\g_{g,n}^\geom \ar[r] & \Gr^W_\dot \ghat_g^\geom \ar[r] & 0
%}
%$$
%of $S_n\times\GSp(H)$-modules is exact and the inclusion
%$$
%\ker\{\Gr^W_\dot\p_{g,n}\to (\Gr^W_\dot\p)^n\} \to
%\ker\{\Gr^W_\dot\ghat_{g,n}^\geom\to \Gr^W_\dot\g_{g,n}^\geom\}
%$$
%is an $S_n\times\GSp(H)$-module isomorphism.
%\end{corollary}

The computations in \cite[\S 12]{hain:torelli} imply directly that
$\Gr^W_{-1}\p_{g,n}$ is the $S_n\times\GSp(H)$-module $H^n$. They also imply
that $\Gr^W_{-2}\p_{g,n}\to (\Gr^W_{-2}\p)^n$ is surjective with kernel
isomorphic to $\bigoplus_{i<j}\Ql(1)$, where the $ij$th factor is generated by
the logarithm\footnote{Every element of a prounipotent group over a field of
characteristic zero has a unique logarithm.} of a loop that encircles the $ij$th
diagonal $\Delta_{ij}$. Combined with the previous result, this gives:

\begin{lemma}
\label{lem:kernel}
If $g\ge 3$, then
$$
\Gr^W_r \u_{g,n} \to \Gr^W_r \uhat_{g,n} \text{ and }
\Gr^W_r \u_{g,n}^\geom \to \Gr^W_r \uhat_{g,n}^\geom
$$
are isomorphisms when $r=-1$ and is surjective when $r=-2$, both with kernel
isomorphic to the $S_n\times \GSp(H)$-module $\bigoplus_{i<j} \Ql(1)$, where
$S_n$ acts by permuting the factors.\footnote{In case there is any confusion,
the transposition $(i,j)$ acts trivially on the $ij$th factor.}
\end{lemma}

Suppose that $n\ge 1$. Set $\pi := \pi_1(C_\etabar,\xbar_1)$. Define $\p$ to be
the Lie algebra of $\pi^\un$. Then $\pi_1(\M_{g,n/k}[m],\etabar)$ acts on $\pi$
via the first projection $\pi_1(\M_{g,n/k}[m],\etabar) \to
\pi_1(\M_{g,1/k}[m],\xbar_1)$. This induces an action of $\cG_{g,n}$ on $\p$, so
that $\p$ also has a natural weight filtration. This action induces a Lie
algebra homomorphism $\g_{g,n} \to \Der\p$ which preserves the weight filtration
as it is $\cG_{g,n}$-equivariant.

\begin{proposition}
\label{prop:geom2arith}
If $g\ge 3$, the homomorphism $\u_{g,1}\to\Der\p$ induces natural homomorphisms
$$
\Gr^W_\dot \u_{g,1}^\geom \to \Gr^W_\dot \u_{g,1} \to \Gr^W_\dot \Der \p
\cong \Der \Gr^W_\dot \p.
$$
of Lie algebras in the category of $\GSp(H)$-modules. For all $n\ge 0$, the
natural map
$$
\Gr^W_{-1} \u_{g,n}^\geom \to \Gr^W_{-1}\u_{g,n}
$$
is an isomorphism. In weight $-2$ there is a short exact sequence
$$
0 \to \Gr^W_{-2} \u_{g,n}^\geom \to \Gr^W_{-2}\u_{g,n}
\to H^1(G_k,\Ql(1))^\ast\otimes\Ql(1) \to 0.
$$
\end{proposition}

\begin{proof}
Right exactness of weighted completion and the exactness of $\Gr^W_m$ imply that
the sequence
$$
\Gr^W_r\u^\geom_{g,n} \to \Gr^W_r\g_{g,n} \to \Gr^W_r\a_k \to 0
$$
is exact for all $r$. Recall from Example~\ref{ex:fields} that
$$
\Gr^W_{m}\a_k =
\begin{cases}
0 & r=-1 \cr
H^1(G_k,\Ql(1))^\ast\otimes\Ql(1) & r = -2.
\end{cases}
$$
The map $\Gr^W_r\u_{g,n}^\geom \to \Gr^W_r\u_{g,n}$ is therefore surjective
when $r=-1$ and has cokernel $H^1(G_k,\Ql(1))^\ast\otimes\Ql(1)$ when $r=-2$.

The computations in \cite[Sects 9, 10]{hain:torelli} imply that the homomorphism
$$
\Gr^W_r \u_{g,1}^\geom \to \Der \Gr^W_r\p
$$
is an isomorphism when $r=-1$ and injective when $r=-2$. This completes the
proof of the result when $n\le 1$.

When $n>1$, Proposition~\ref{prop:exactness} and the $n=1$ case imply that
$\Gr^W_r\uhat_{g,n}^\geom \to \Gr^W_r\uhat_{g,n}$ is injective. The result now
follows from Lemma~\ref{lem:kernel}.
\end{proof}

\section{Generators and Relations}
\label{sec:relations}

The proofs of our main results make essential use of the of the structure of
$\Gr^W_\dot\u_{g,n}^\geom/W_{-3}$ as a graded Lie algebra in the category of
$S_n\times \GSp(H)$-modules. In this section we compute presentations of these
Lie algebras. These results are refinements and elaborations of computations in
\cite[Sects.~11 \& 13]{hain:torelli}.

\subsection{Representations of $\Sp(H)$ and $\GSp(H)$}

This is a quick review of the representation theory of $\GSp(H)$ over $\Q$,
which is needed in subsequent sections. A good reference is the book
\cite{fulton-harris} of Fulton and Harris. Throughout $H$ is a finite
dimensional $\Q$-vector space which is endowed with a unimodular, skew symmetric
bilinear, $\GSp(H)$-invariant bilinear form $\theta : H\otimes H \to \Q(1)$.
Since both groups are split over $\Q$, all of their irreducible representations
are absolutely irreducible. The representations of $\GSp(H)$ and $\Sp(H)$ over a
field $F$ of characteristic zero are obtained from their representations over
$\Q$ by extension of scalars.

The standard cocharacter $\w : z \mapsto z^{-1}\id_H$ assigns a weight $w(V)$ to
each irreducible $\GSp(H)$-module $V$. The defining representation $H$ has
weight $-1$, which implies that $H^{\otimes r}$ has weight $-r$. As we recall
below, the irreducible $\GSp(H)$ modules are constructed from $H^{\otimes r}$
using the symplectic form $\theta : H^{\otimes 2} \to \Q(1)$ and the action of
the symmetric group $S_r$, which acts on $H^{\otimes r}$ by permuting the
factors.

Recall from Section~\ref{sec:gsp} that $\Q(r)$ denotes the 1-dimensional 
representation of $\GSp(H)$ with character $\tau^r : \GSp(H) \to \Gm$. Since the
composite of the character $\tau : \GSp(H) \to \Gm$ with the standard
cocharacter
$$
\Gm \stackrel{\w}{\longrightarrow} \GSp(H) \stackrel{\tau}{\longrightarrow} \Gm
$$
is $x \mapsto x^{-2}$, $\Q(r)$ has weight $-2r$. For any $\GSp(H)$-module $V$,
set
$$
V(r) = V\otimes \Q(r).
$$
We will call this a ``Tate twist'' of $V$. If $V$ has weight $w$, then $V(r)$
has weight $w-2r$.

For $1 \le i < j \le r$, define $\theta_{ij}: H^{\otimes r} \to H^{\otimes
(r-2)}(1)$ to be the map
$$
\theta_{ij} : x_1\otimes \dots \otimes x_r \mapsto \theta(x_i,x_j)
x_1 \otimes \dots \otimes \widehat{x_i} \otimes \dots \otimes \widehat{x_j}
\otimes \dots \otimes x_n 
$$
that contracts the $i$th and $j$th factors using $\theta$. Set
$$
H^\red{r} = \bigcap_{i<j} \ker \theta_{ij}.
$$
This is a $\GSp(H)\times S_r$ submodule of $H^{\otimes r}$.

The isomorphism classes of irreducible $S_r$-modules are parametrized by
partitions
$$
\lambda = [\lambda_1, \lambda_2, \dots, \lambda_s]
$$
of $r$ into $s\le r$ parts, where $\lambda_1 \ge \lambda_2 \ge \dots \ge
\lambda_s > 0$. Define $|\lambda|$ by
$$
|\lambda| := r = \lambda_1 + \lambda_2 + \dots + \lambda_s. 
$$
To see how a partition $\lambda$ of $r$ gives a $\GSp(H)$-module, choose an
irreducible $S_r$-module $W_\lambda$ that corresponds to the partition $\lambda$
and a non-zero element $w_\lambda$ of it. Set $S^\lambda H =
\Hom_{S_r}(W_\lambda,H^{\otimes r})$. This is a $\GSp(H)$-module which we will
identify with the image of the map
$\Hom_{S_r}(W_\lambda,H^{\otimes r})\to H^{\otimes r}$ that takes $\phi$ to
$\phi(w_\lambda)$. Define
$$
H_\lambda = S^\lambda H \cap H^\red{r}.
$$
This is a $\GSp(H)$-module of weight $-|\lambda|$.  Note that $H_\lambda(r)$ has
weight $-|\lambda|-2r$.

Set $g = \dim H/2$. The basic general facts we need from representation theory
are summarized in the following result, which can be deduced from standards
results, such as those presented in \cite{fulton-harris}.

\begin{theorem}
The representations $H_\lambda(r)$, where $r\in \Z$ and $\lambda$ is a partition
of a non-negative integer into $\le g$ parts, is a set of representatives of the
isomorphism classes of the irreducible $\GSp(H)$-modules. Each remains
irreducible when restricted to $\Sp(H)$; every irreducible $\Sp(H)$-module
arises in this way. Two irreducible $\GSp(H)$-modules $H_\lambda(r)$ and
$H_{\lambda'}(r')$ are isomorphic as $\Sp(H)$-modules if and only if $\lambda =
\lambda'$.
\end{theorem}

\subsection{A presentation of $\Gr^W_\dot\p$}

Here $\p$ denotes the Lie algebra of the unipotent completion of the fundamental
group of a smooth complex projective curve of genus $g\ge 1$. Its natural weight
filtration is defined over $\Q$ as $W_{-r}\p$ is the $r$th term of its lower
central series. We recall the presentation of $\Gr^W_\dot \p$. Recall that
$\thetadual$ is a map $\Q(1) \to \Lambda^2 H$. It will be regarded as a map
$\thetadual : \Q(1) \to \L_2(H)$ via the canonical identification $\L_2(H)\cong
\Lambda^2 H$. The next result follows directly from a fundamental result of
Labute \cite{labute}. It also follows from Hodge theory via the main result of
\cite{dgms} or by the more motivic argument in \cite[Thm.~5.8]{hain:torelli}.

\begin{proposition}
If $g\ge 1$, then the weight filtration of $\p$ equals its lower central series.
There is  a natural $\GSp(H)$-equivariant Lie algebra isomorphism
$$
\Gr^W_\dot \p \cong \L(H)/(\im\thetadual).
$$
Inner automorphisms act trivially on this isomorphism, so that it does not
depend on the choice of a basepoint.
\end{proposition}

A direct computation using the method of \cite[\S8]{hain:torelli} and the fact
(cf.\ \cite[(4.1)]{kohno-oda}) that $H^\dot(\Gr^W_\dot\p) \cong
H^\dot(\pi,\Q)$ yields:

\begin{corollary}
\label{cor:trivial_reps}
If $g\ge 1$ and $1\le r<6$, then $(\Gr^W_{-r}\p)^{\Sp(H)} = 0$.
\end{corollary}

\begin{proof}
When $g=1$, $\p$ is abelian and the result is trivially true. Suppose that $g\ge
2$. The trivial $\Sp(H)$-module can occur in $\Gr^W_{-r}\p$ only when $r$ is
even. So we need to check that there are no copies of the trivial representation
in $\Gr^W_{-2}\p$ and $\Gr^W_{-4}\p$. The proof of
\cite[Prop.~8.4]{hain:torelli} implies that for all $g\ge 2$, $\Gr^W_{-1}\p =
H$, $\Gr^W_{-2}\p = \Lambda^2 H/\im\thetadual$, and that $\Gr^W_{-3}\p$ is the
irreducible $\Sp(H)$-module that is the highest weight part of $H\otimes
\Lambda^2 H$, which is isomorphic to $H_{[2,1]}$. Each is an irreducible
$\Sp(H)$-module. The case $r=2$ follows. When $r=4$,
\cite[Cor.~8.3]{hain:torelli} implies that
$$
[\Gr^W_{-4}\p] = [\Gr^W_{-1}\p\otimes\Gr^W_{-3}\p] + [\Lambda^2\Gr^W_{-2}\p]
- [\Gr^W_{-2}\p\otimes \Lambda^2\Gr^W_{-1}\p] + [\Lambda^4 \Gr^W_{-1} \p]
$$
in the representation ring\footnote{The class of the $G$-module $V$ in the
representation ring of $G$ is denoted by $[V]$.} of $\Sp(H)$. Schur's Lemma
implies that the first two terms do not contain a copy of the trivial
representation and that the third term contains one. The decomposition
$$
\Lambda^4 H \cong
\begin{cases}
H_{[1^4]} \oplus H_{[1^2]} \oplus \Q & g\ge 4, \cr
H_{[1^2]} \oplus \Q & g = 3, \cr
\Q & g=2,
\end{cases}
$$
into irreducible $\Sp(H)$ modules implies that the fourth term contains one copy
of the trivial representation. The result follows.
\end{proof}

The same method can be used to show that $\dim (\Gr^W_{-6}\p)^{\Sp(H)} = 1$
for all $g\ge 2$.

\subsection{Frequently used $\GSp(H)$-modules}

Several $\GSp(H)$-modules play a significant role in this paper. All are of
weight $-1$ or $-2$ and are  Tate twists of the representations that correspond
to the partitions $[1]$, $[1^2]:=[1,1]$, and $[1^3]:=[1,1,1]$. We first define
integral models of each as we will need them in the genus 3 case.

The form $\theta$ will be a unimodular, skew-symmetric bilinear form on $H_\Z$
and gives rise to a $\GSp(H)$-equivariant mapping $\Lambda^2 H_A \to A(1)$ for
all  commutative rings $A$. The dual form $\thetadual$, defined in
Section~\ref{sec:gsp}, will be regarded both as an equivariant map $\thetadual :
A(1) \to \Lambda^2H_A$ and as an element of $\Lambda^2 H_A(-1)$. Set
$$
\Lambda^2_0 H_A := \Lambda^2 H_A/\im\thetadual
\text{ and }
\Lambda^3_0 H_A := (\Lambda^3 H_A)(-1)/\thetadual \wedge H_A.
$$
When $A$ is a field of characteristic zero, $\Lambda^2_0 H_A$ and $\Lambda^3_0
H_A$ are the irreducible $\GSp(H)$-modules $H_{[1^2]}$ and $H_{[1^3]}(-1)$.
These have weights $-2$ and $-1$, respectively. 

\begin{lemma}
\label{lem:split}
If $A$ is an integral domain of characteristic zero, then the short exact
sequence
\begin{equation}
\label{eqn:ses1}
\xymatrix{
0 \ar[r] & A(1)  \ar[r]^(.4)\thetadual & \Lambda^2 H_A \ar[r] &
\Lambda^2_0 H_A \ar[r] & 0
}
\end{equation}
of $\GSp(H_A)$-modules splits if and only if $g \in A^\times$. The short exact
sequence
\begin{equation}
\label{eqn:ses2}
\xymatrix{
0 \ar[r] &  H_A \ar[r]^(.33){\thetadual\wedge} & \big(\Lambda^3 H_A\big)(-1)
\ar[r] & \Lambda^3_0 H_A \ar[r] & 0
}
\end{equation}
of $\GSp(H_A)$-modules splits if and only if $g-1 \in A^\times$.
\end{lemma}

\begin{proof}
We will prove the lemma for the exact sequence (\ref{eqn:ses2}). The proof for
exact sequence (\ref{eqn:ses1}) is similar and is left to the reader. The map $c
: \Lambda^3 H_A \to H_A(1)$ defined by
$$
c: x \wedge y \wedge z \mapsto
\theta(x,y)z + \theta(y,z)x + \theta(z,x)y
$$
is $\GSp(H_A)$-equivariant. Its composition $c\circ j$ with the inclusion
$j:H_A(1) \hookrightarrow \Lambda^3H_A$ is $(g-1)$ times the identity of
$H_A(1)$. If $g-1\in A^\times$. Thus $c/(g-1)$ splits (\ref{eqn:ses2})
twisted by $A(1)$.

Suppose now that $s:\Lambda^3 H_A(-1) \to H_A$ is a splitting of
(\ref{eqn:ses2}). Denote the fraction field of $A$ by $F$. Since $F$ has
characteristic zero, $H_F$ and $\Lambda^3_0 H_F$ are irreducible
$\GSp(H_F)$-modules and the splitting of (\ref{eqn:ses2})$\otimes_A F$ is
unique, and so must be $c/(g-1)$. This, and the fact that $A$ is a domain, imply
that $s=c/(g-1)$. Since $\theta$ is unimodular, the image of $c/(g-1): \Lambda^3
H_A(-1) \to H_F$ is contained in $H_A$ if and only if $g-1\in A^\times$.
\end{proof}

The lemma implies that for all fields $F$ of characteristic zero
$$
\Lambda^2 H_F \cong A(1) \oplus \Lambda^2_0 H_F \text{ and }
\big(\Lambda^3 H_F\big)(-1) \cong H_F \oplus \Lambda^3_0 H_F
$$
We fix these isomorphisms to be the ones given by the unique $\GSp(H)$-invariant
splittings of the sequences in the lemma above. From this point on, we will
write $H$ instead of $H_F$ when $F$ is a field of characteristic zero. The
default field of characteristic zero in this section will be $\Q$. In later
sections it will be $\Ql$.

Another representation that will occur, but which plays a minor role, is 
$$
V_\tambo := H_{[2,2]}(-1),
$$
which has weight $-2$. The representation $H_{[2,2]}$ is the highest weight part
of the second symmetric power of $\Lambda^2_0 H$.

Kabanov's stability result \cite{kabanov} implies that when $g\ge 6$ the
decomposition of $\Lambda^2\Lambda^3_0 H$ is independent of $g$ in the sense
described in \cite[\S6]{hain:torelli}. Because of this, one can compute this
decomposition for all $g\ge 6$ by computing it when $g=6$. The following
computations were made using the computer program $\sf{LiE}$.

\begin{proposition}[{cf.\ \cite[Lem.~10.2]{hain:torelli}}]
\label{prop:reps}
If $g\ge 3$, then each irreducible $\GSp(H)$ module that occurs in $\Lambda^2
\Lambda^3_0 H$ occurs with multiplicity one. When $g=3$,
$$
\Lambda^2(\Lambda^3_0 H) = \Q(1) \oplus V_\tambo;
$$
when $g\ge 4$,
$$
\Lambda^2(\Lambda^3_0 H) \supset \Q(1) \oplus \Lambda^2_0 H
\oplus V_\tambo.
$$
Moreover, the representation $H\otimes \Lambda^3_0 H$ contains a unique
copy of $\Lambda^2_0 H$ for all $g\ge 3$.
\end{proposition}

\subsection{Presentation of $\Gr^W_\dot \u_{g,1}^\geom/W_{-3}$}

The Lie algebra $\Gr^W_\dot\u_{g,1}^\geom/W_{-3}$ can be computed by considering
its action on $\Gr^W_\dot\p$.

Our first task is to determine $\Gr^W_\dot\Der\p/W_{-3}$. Fix a field $F$ of
characteristic zero. Denote the free Lie algebra generated by the $F$-vector
space $V$ by $\L(V)$. 

The free Lie algebra $\L(V)$ is graded by bracket length:
$$
\L(V) \cong \bigoplus_{n\ge 1} \L_n(V).
$$
Observe that the derivation Lie algebra $\Der \L(H)$ is isomorphic to
$\Hom_F(H,\L(H))$ and is also graded:
$$
\Der\L(H) = \bigoplus_{n\ge 0} \Der_n \L(H),
$$
where $\Der_n \L(H) := \Hom(H,\L_{n+1}(H))$. Observe that there are natural
$\GSp(H)$-actions on $\L\big((\Lambda^3 H) (-1)\big)$ and $\Der \L(H)$. The
following fact is well-know, but we prove it because of its importance in this
paper. First note if $x,y\in H$, then $\theta(x,y)\thetadual \in \Lambda^2 H$. 

\begin{lemma}
If $g\ge 2$, there is a natural $\GSp(H)$-equivariant, graded Lie algebra
homomorphism 
$$
\deltatilde : \L\big((\Lambda^3 H)(-1)\big) \to \Der \L(H)
$$
such that
\begin{enumerate}

\item $\deltatilde(u)$ annihilates the image of $\thetadual$ for all $u \in
\L\big((\Lambda^3 H)(-1)\big)$;

\item $\deltatilde(x\wedge \thetadual) = \ad_x - \theta(x,\blank)\thetadual$
for all $x\in H$.

\end{enumerate}
Moreover
$$
\im\big\{\L_2\big((\Lambda^3 H) (-1)\big) \to \Hom(H,\L_3(H))\big\} \cong
V_\tambo \oplus \Lambda^2_0 H \oplus F(1).
$$
\end{lemma}

\begin{proof}
It suffices to define a $\GSp(H)$-equivariant linear mapping
$$
\Lambda^3 H \to \big[\Der_1 \L(H)\big](1) = \Hom(H,\L_2(H)(1)).
$$
This homomorphism is defined by
\begin{equation}
\label{eqn:johnson}
\deltatilde :
x \wedge y \wedge z : u \mapsto
\theta(u,x) [y,z] + \theta(u,y) [z,x] + \theta(u,z) [x,y].
\end{equation}
The identities are easily verified.
\end{proof}

Since each element of $\L\big((\Lambda^3 H)(-1)\big)$ annihilates $\im\thetadual
\in \L_2(H)$, $\deltatilde$ induces a $\GSp(H)$-equivariant graded Lie algebra
homomorphism
$$
\delta : \L\big((\Lambda^3 H)(-1)\big) \to \Der \Gr^W_\dot \p.
$$

\begin{theorem}[{\cite[Cor.~5.7 and \S 11]{hain:torelli}}]
\label{thm:ug1}
If $g\ge 3$, then there is a Lie algebra surjection
$$
q : \L\big((\Lambda^3 H)(-1)\big) \to \Gr^W_\dot \u_{g,1}^\geom
$$
such that the diagram
$$
\xymatrix@C=0pt{
\L\big((\Lambda^3 H)(-1)\big) \ar[rr]^q\ar[dr]_\delta &&
\Gr^W_\dot \u_{g,1}^\geom \ar[dl] \cr
& \Gr^W_\dot \Der \p
}
$$
commutes. It induces isomorphisms
\begin{equation}
\label{eqn:above}
\Gr^W_r \u_{g,1}^\geom \overset{\simeq}{\longrightarrow} \Gr^W_r\Der\p \cong
\begin{cases}
(\Lambda^3 H)(-1) & r = -1,\cr
V_\tambo \oplus \Lambda^2_0 H & r = -2.
\end{cases}
\end{equation}
\end{theorem}

The isomorphism $\Gr^W_{-1}\Der\p \cong \Lambda^3 H(-1)$ when $g\ge 3$ is a
manifestation of the Johnson homomorphism \cite{johnson:h1}. It is induced by
the homomorphism (\ref{eqn:johnson}). The composition of $\delta$ with inclusion
$i :H\hookrightarrow (\Lambda^3 H)(-1)$ defined by $x\mapsto x\wedge \thetadual$
takes $x\in H$ to the inner derivation $\ad_x$ of $\p$. Identify the copy of
$\Lambda^2_0 H$ in $\Gr^W_{-2}\Der\u_{g,1}^\geom$ with the image of the mapping
$$
\xymatrix{
\Lambda^2 H \ar[r]^(.35){\Lambda^2 i}&
\Lambda^2 \big((\Lambda^3 H)(-1)\big) \ar[r]^(.55){\bracket} &
\Gr^W_{-2}\u_{g,1}^\geom.
}
$$
induced by the bracket.

Since the isomorphisms (\ref{eqn:above}) factor through the map
$\Gr^W_r\u_{g,1}^\geom \to \Gr^W_r\u_{g,1}$, we conclude:

\begin{corollary}
If $g\ge 3$, then $\Gr^W_r \u_{g,1}^\geom \to \Gr^W_r \u_{g,1}$ is injective
when $r=-1,-2$.
\end{corollary}

We also need to understand the bracket
\begin{equation}
\label{eqn:bracket}
\Lambda^2 \Gr^W_{-1}\u_{g,1}^\geom \to \Gr^W_{-2}\u_{g,1}^\geom
\end{equation}
With respect to the decomposition $(\Lambda^3 H)(-1) = H\oplus \Lambda^3_0 H$,
the bracket has three components:
\begin{equation}
\label{eqn:bracket_cpts}
\Lambda^2 H \to \Gr^W_{-2}\Der \p,\quad H\otimes \Lambda^3_0 H \to
\Gr^W_{-2}\Der\p,\quad \Lambda^2(\Lambda^3_0 H) \to \Gr^W_{-2}\Der\p.
\end{equation}
Each is $\GSp(H)$-equivariant. These can be computed in $\Der\Gr^W_\dot\p$.

\begin{proposition}[{\cite[\S12]{hain:torelli}}]
\label{prop:bracket:g1}
When $g\ge 3$, the images of the three components of the bracket
(\ref{eqn:bracket}) are:
\begin{align*}
\im\big\{\Lambda^2 H \to \Gr^W_{-2}\Der \p\big\} &= \Lambda^2_0 H, \cr
\im\big\{H\otimes \Lambda^3_0 H \to \Gr^W_{-2}\Der \p\big\} &=
\Lambda^2_0 H,
\cr
\im\big\{\Lambda^2(\Lambda^3_0 H) \to \Gr^W_{-2}\Der\p\big\} &=
\begin{cases}
V_\tambo & g = 3,\cr
V_\tambo \oplus \Lambda^2_0 H & g\ge 4.
\end{cases}
\end{align*}
\end{proposition}

\begin{remark}
The genus 2 story is very different for several reasons, one of which is that
every genus 2 curve is hyperelliptic. On the representation theoretic side
$\Lambda^3 H \cong H(1)$; the image of the homomorphism $\L(\Lambda^3 H(-1)) \to
\Der\p$ is the set of inner derivations and the image of
$\Gr^W_\dot\u_{2,1}\to\Gr^W_\dot\Der\p$ is generated by this and by the copy of
$V_\tambo$ in $\Gr^W_{-2}\Der\p$.
\end{remark}

\subsection{Presentations of $\Gr^W_\dot \uhat_{g,n}^\geom/W_{-3}$ and
$\Gr^W_\dot \u_{g,n}^\geom/W_{-3}$}

The computations of the previous section combined with the results of
Section~\ref{sec:comp_arith_mcgs} yield computations of $\Gr^W_\dot
\uhat_{g,n}^\geom/W_{-3}$ and $\Gr^W_\dot \u_{g,n}^\geom/W_{-3}$.

Denote the image of $u\in \Lambda^3 H$ in $(\Lambda^3_0 H)(1)$ by $\ubar$.
For a positive integer $n$, define
$$
\Lambda^3_n H = 
\{(u_1,\dots,u_n)\in (\Lambda^3 H)^n : \ubar_1 = \dots = \ubar_n\}(-1) 
\cong \Lambda^3_0 H \oplus H^{\oplus n}
$$
This has weight $-1$. Note that $\Lambda^3_1 H = (\Lambda^3 H)(-1)$. The
following result is proved later in the section.

\begin{theorem}
\label{thm:presentation}
If $g\ge 3$, then for all $n\ge 0$ there are natural $S_n\times
\GSp(H)$-equivariant isomorphisms
$$
H_1(\uhat^\geom_{g,n}) \cong H_1(\u^\geom_{g,n}) \cong \Gr^W_{-1}\u^\geom_{g,n}
\cong \Lambda^3_n H.
$$
There is an exact sequence
$$
0 \to \Ql(1)^{\binom{n}{2}} \to \Gr^W_{-2}\u^\geom_{g,n}
\to \Gr^W_{-2}\uhat^\geom_{g,n}
\to 0
$$
of $S_n\times \GSp(H)$-modules and a $S_n\times \GSp(H)$-equivariant isomorphism
$$
\Gr^W_{-2}\u^\geom_{g,n} \cong V_\tambo \oplus
\big(\Lambda^2_0 H)^n \oplus \Ql(1)^{\binom{n}{2}}
$$
\end{theorem}

Our next task is to describe the bracket
$$
\bracket : \Lambda^2 \Gr^W_{-1}\u^\geom_{g,n} \to \Gr^W_{-2}\u^\geom_{g,n}.
$$
We do this by describing its ``matrix entries''. To this end, write
$$
\Gr^W_{-1}\u^\geom_{g,n} = \Lambda^3_n H =
\Lambda^3_0 H \oplus H_1 \oplus \dots \oplus H_n
$$
where the $j$th copy $H_j$ of $H$ corresponds to the $j$th point. Write
$$
\Gr^W_{-2}\u^\geom_{g,n} = V_\tambo \oplus \bigoplus_{i<j}\Ql(1)_{ij}
\oplus \bigoplus_{j=1}^n \Lambda^2_0 H_j.
$$
This isomorphism is chosen so that the bracket
$$
\Lambda^2 H_j \hookrightarrow
\Gr^W_{-2}\u^\geom_{g,n} \overset{\text{proj}}{\to} \Lambda^2_0 H_j
$$
is the quotient mapping and so that
$$
H\otimes H \cong H_i\otimes H_j \hookrightarrow
\Gr^W_{-2}\u^\geom_{g,n} \overset{\text{proj}}{\to} \Ql(1)_{ij}
$$  
is the polarization $\theta$. That this is possible follows from \cite[\S
13]{hain:torelli}. Let
$$
p_j : \Gr^W_{-2}\u^\geom_{g,n} \to \Lambda^2_0 H,\
q_{ij} : \Gr^W_{-2}\u^\geom_{g,n} \to \Ql(1),\
p_\tambo : \Gr^W_{-2}\u^\geom_{g,n} \to V_\tambo
$$
be the corresponding projections.

Note that
\begin{equation}
\label{eqn:decomp}
\Lambda^2 \Gr^W_{-1}\u^\geom_{g,n} \cong
\Lambda^2\Lambda^3_0 H
\oplus \bigoplus_{j=1}^n (H_j\otimes \Lambda^3_0 H)
\oplus \bigoplus_{j=1}^n \Lambda^2 H_j
\oplus \bigoplus_{i<j} H_i\otimes H_j
\end{equation}

Choose $\GSp(H)$-equivariant projections
\begin{gather*}
\cc : \Lambda^2\Lambda^3_0 H \to \Lambda^2_0 H,\
\dd : H\otimes\Lambda^3_0 H \to \Lambda^2_0 H,\cr
\ee : \Lambda^2 H \to \Lambda^2_0 H,\
\psi : \Lambda^2\Lambda^3_0 H \to \Ql(1).
\end{gather*}
Here, and in the following definition, $g\ge 3$ except in the definition of
$\cc$ where $g\ge 4$. Proposition~\ref{prop:reps} implies that each is unique up
to a scalar multiple. Denote the $\GSp(H)$-invariant projections
\begin{align}
\label{eqn:projns}
\Lambda^2\Gr^W_{-1}\u^\geom_{g,n} &\to \Lambda^2 \Lambda^3_0 H
\overset{\cc}{\to} \Lambda^2_0 H \qquad\qquad g\ge 4   \cr
\Lambda^2\Gr^W_{-1}\u^\geom_{g,n} &\to H_j\otimes\Lambda^3_0 H
\overset{\dd}{\to} \Lambda^2_0 H\cr
\Lambda^2\Gr^W_{-1}\u^\geom_{g,n} &\to \Lambda^2 H_j
\overset{\ee}{\to} \Lambda^2_0 H\cr
\Lambda^2\Gr^W_{-1}\u^\geom_{g,n} &\to H_i\otimes H_j
\overset{\ee}{\to} \Lambda^2_0 H\cr
\Lambda^2\Gr^W_{-1}\u^\geom_{g,n} &\to \Lambda^2H_i
\overset{\theta}{\to} \Ql(1) \cr
\Lambda^2\Gr^W_{-1}\u^\geom_{g,n} &\to H_i\otimes H_j
\overset{\theta}{\to} \Ql(1) \cr
\Lambda^2\Gr^W_{-1}\u^\geom_{g,n} &\to \Lambda^2 \Lambda^3_0 H
\overset{\psi}{\to} \Ql(1)
\end{align}
by $\cc$, $\dd_j$, $\ee_j$, $\ee_{ij}$, $\theta_i$, $\theta_{ij}$ and $\psi$,
respectively. For convenience, we define $\cc$ to be zero when $g=3$.
Proposition~\ref{prop:reps} and (\ref{eqn:decomp}) imply:

\begin{proposition}
\label{prop:projections}
If $g \ge 3$ and $n\ge 0$, then
$$
\Hom_{\Sp(H)}(\Lambda^2 \Gr^W_{-1}\u^\geom_{g,n},\Lambda^2_0 H)
$$
has basis
\begin{align*}
\{\dd_1,\dots,\dd_n,\ee_1,\dots,\ee_n\} \cup \{\ee_{ij}: 1\le i< j \le n\}
\quad & g=3 \cr
\{\cc,\dd_1,\dots,\dd_n,\ee_1,\dots,\ee_n\} \cup \{\ee_{ij}: 1\le i< j \le n\}
\quad & g\ge 4
\end{align*}
and
$$
\Hom_{\Sp(H)}(\Lambda^2 \Gr^W_{-1}\u^\geom_{g,n},\Ql)
$$
has basis
$$
\{\psi,\theta_{ij} : 1\le i<j\le n\} \cup \{\theta_1,\dots,\theta_n\}
$$
for all $g\ge 3$.
\end{proposition}

The bracket of $\Gr^W_\dot\u^\geom_{g,n}$ is determined by:

\begin{proposition}
\label{prop:bracket}
If $g\ge 3$, then, after rescaling $\psi$, $\cc$ and $\dd$ by non-zero constants
if necessary,\footnote{Explicit formulas for these projections can be deduced
from the formulas in \cite[\S11]{hain:torelli}, although we shall not need
them.}
\begin{align*}
p_j\circ\bracket &= \dd_j + \ee_j & g = 3 \cr
p_j\circ\bracket &= \cc + \dd_j + \ee_j & g \ge 4\cr
q_{ij}\circ\bracket &= \psi + \theta_{ij} & g\ge 3. \cr
\end{align*}
In addition, $p_\tambo\circ \bracket$ is non-zero and vanishes on
each $H_j\otimes \Lambda^3_0 H$ and each $H_i\otimes H_j$, where $i\le j$.
\end{proposition}

The bracket of $\Gr^W_\dot\uhat^\geom_{g,n}$ is obtained by ignoring the
$\Ql(1)$ factors.

\begin{proof}[Proof of Theorem~\ref{thm:presentation} and
Proposition~\ref{prop:bracket}]
Lemma~\ref{lem:kernel}, Theorem~\ref{thm:ug1} and
Proposition~\ref{prop:exactness} imply that
$$
\Gr^W_{-1}\u_{g,n}^\geom \cong \Gr^W_{-1}\uhat_{g,n}^\geom \cong
\Lambda^3_n H
$$
and that
$$
\Gr^W_{-2} \uhat^\geom_{g,n} \cong V_\tambo \oplus (\Lambda_0^2 H)^n
\text{ and }
\Gr^W_{-2} \u^\geom_{g,n} \cong V_\tambo \oplus (\Lambda_0^2 H)^n
\oplus \Ql(1)^{\binom{n}{2}}
$$
The bracket of $\Gr^W_\dot\uhat_{g,n}^\geom/W_{-3}$ is computed using the
homomorphism
$$
\Gr^W_\dot\uhat_{g,n}^\geom \to (\Gr^W_\dot\uhat^\geom_{g,1})^n
$$
induced by the inclusion $\cC_{g/k}^n \to (\M_{g,1/k})^n$, which is injective.
The computation of the bracket in this case follows directly from
Proposition~\ref{prop:bracket:g1}. The surjectivity of $\u^\geom_{g,n} \to
\uhat^\geom_{g,n}$ implies that, to compute the bracket in
$\Gr^W_\dot\u^\geom_{g,n}/W_{-3}$, we need only compute the $\Ql(1)$ component.
This is determined by the bracket in $\Gr^W_\dot\p_{g,n}$, which is computed in
\cite[\S 12]{hain:torelli}.
\end{proof}

\section{The Lie Algebra $\d_{g,n}$}
\label{sec:la_d}

In this section we associate a graded, 2-step nilpotent Lie algebra $\d(\u_T)$
to the base $T$ of a family of smooth projective curves with Zariski dense
monodromy representation $\pi_1(T,t_o)\to \GSp(H_\Ql)$. This Lie algebra is a
useful tool for studying the existence of rational points. When applied to
$\M_{g,n/k}$ or $\cC_{g/k}^n$, it gives the Lie algebra $\d_{g,n}$. The main
result of this section is that, when $g\ge 4$, sections of the natural
projection $\d_{g,n+1}\to \d_{g,n}$ correspond exactly to the $n$ tautological
sections of the universal curve over $\M_{g,n/k}$.

Unless mentioned to the contrary, all Lie algebras in this section will be over
the coefficient field $F$, which has characteristic zero. The most common
choices in this work are $F=\Q$, $\Ql$, $\C$.

\subsection{The functor $\d$ and the Lie algebra $\d_{g,n}$}

Suppose that $G$ is an extension of $\GSp(H)$ by a prounipotent group that is
negatively weighted with respect to the standard cocharacter $\w$, such as
$\cG_{g,n}$.

For a Lie algebra $\u$ in the category of $G$-modules, define $\d(\u)$ to be the
Lie algebra
$$
\d(\u) = \big(\Gr^W_\dot (W_{-1}\u/W_{-3})\big)/(\Lambda^2_0 H)^\perp
$$
where $(\Lambda^2_0 H)^\perp$ denotes the
$\GSp(H)$-invariant complement in $\Gr^W_{-2}\u$ of its
$\Lambda^2_0H$-isotypical component.
It is a two-step, graded, nilpotent Lie algebra in the category of
$\GSp(H)$-modules whose $r$th graded quotient is
$$
\d(\u)_r
\begin{cases}
\Gr^W_{-1}\u & r=-1,\cr
\Gr^W_{-2}\u/(\Lambda^2_0H)^\perp & r=-2,\cr
0 & r\neq -1,-2.
\end{cases}
$$

\begin{definition}
Suppose that $g\ge 3$ and $n\ge 0$. Define $\d_{g,n} = \d(\u_{g,n}^\geom)$. This
is a Lie algebra in category of $S_n\times \GSp(H)$-modules.
\end{definition}

This satisfies
$$
(\d_{g,n})_r =
\begin{cases}
\Lambda^3_n H & r=-1, \cr
(\Lambda^2_0 H)^n & r=-2.
\end{cases}
$$
Note also that $\d_{g,0}$ is an abelian Lie algebra isomorphic to $\Lambda^3_0
H$.

Proposition~\ref{prop:geom2arith} implies that for all $g\ge 3$ and $n\ge 0$,
the natural map $\u_{g,n}^\geom \to \u_{g,n}$ induces an
$S_n\times\GSp(H)$-equivariant Lie algebra isomorphism
\begin{equation}
\label{eqn:d_arith}
\d_{g,n} = \d(\u_{g,n}^\geom) \cong \d(\u_{g,n}).
\end{equation}

Denote the universal {\em complete} curve over $\M_{g,n/k}$ by $\cC_{g,n/k}$.
This is a Zariski open subset of the $(n+1)$st power $\cC_{g/k}^{n+1}$ of the
universal curve over $\M_{g/k}$. Index the copies of the universal curve by
integers $j$ between $0$ and $n$. Label the points so that the projection
$\cC_{g,n/k} \to \M_{g,n/k}$ takes $[C;x_0,\dots,x_n]$ to $[C;x_1,\dots,x_n]$.
This projection is $S_n$-equivariant, where $S_n$ acts by permuting the points
$\{x_1,\dots,x_n\}$. Note that $\M_{g,n+1/k}$ is a Zariski open subset of
$\cC_{g,n/k}$. A useful property of the functor $\d$, which is proved below, is
that, when $g\ge 3$, the inclusion $\M_{g,n+1}\hookrightarrow \cC_{g,n}$ induces
an isomorphism on $\d(\u)$.

Choose a geometric point $\etabar$ of $\M_{g,n+1/k}$. Assume that $\chi_\ell :
G_k \to \Zlx$ has infinite image. Denote the Lie algebra of the weighted
completion of $\pi_1(\cC_{g,n/k},\etabar)$ with respect to the monodromy
representation to $\GSp(H)$ by $\g_{\cC_{g,n}}$ and the Lie algebra of its
prounipotent radical by $\u_{\cC_{g,n}}$. Denote the Lie algebra of the relative
completion of $\pi_1(\cC_{g,n/\kbar},\etabar)$ by $\g_{\cC_{g,n}}^\geom$ and the
Lie algebra of its prounipotent radical by $\u_{\cC_{g,n}}^\geom$.

\begin{proposition}
\label{prop:d_invariance}
If $g\ge 3$ and $n\ge 0$, then the inclusion $\M_{g,n+1} \hookrightarrow
\cC_{g,n/k}$ induces $\GSp(H)$-equivariant Lie algebra isomorphisms
$$
\d_{g,n+1} := \d(\u_{g,n+1}^\geom) \cong \d(\u_{g,n+1})
\cong \d(\u_{\cC_{g,n}}^\geom) \cong\d(\u_{\cC_{g,n}})
$$
The first isomorphism is $S_{n+1}$-equivariant; the remaining two isomorphisms
are $S_{n}$-equivariant.
\end{proposition}

\begin{proof}
We first prove the arithmetic case. Consider the diagram
$$
\xymatrix{
\M_{g,n+1} \ar@{^{(}->}[r]\ar[dr]_{\pi'} & \cC_{g,n} \ar[d]^{\pi} \cr
& \M_{g,n}
}
$$
where $\pi$ and $\pi'$ both take $(x_0,\dots,x_n)$ to $(x_1,\dots,x_n)$. The
fiber of $\pi'$ is an $n$-punctured curve. Choose compatible basepoints for all
three spaces. Denote the unipotent completion of the fiber of $\pi'$ over the
basepoint of $\M_{g,n}$ by $\p'$. Then one has the commutative diagram
$$
\xymatrix{
0 \ar[r] & \p' \ar[r] \ar[d] & \u_{g,n+1} \ar[r]\ar[d] & \u_{g,n}
\ar[r]\ar@{=}[d] & 0 \cr
0 \ar[r] & \p \ar[r] & \u_{\cC_{g,n}} \ar[r] & \u_{g,n} \ar[r] & 0
}
$$
whose vertical maps are surjective. Exactness of the second row follows from
Proposition~\ref{prop:section}; exactness of the first row follows from
Proposition~\ref{prop:left_exactness} because $\p'$ is free and therefore has
trivial center, and because $H_1(\p)$ has weights $-1$ and $-2$. The diagram
implies that
$$
\Gr^W_{-2}\ker\{\u_{g,n+1} \to \u_{\cC_{g,n}}\}
= \Gr^W_{-2}\ker\{\p'\to \p\} \cong \Ql(1)^{n-1}.
$$
from which it follows that $\d(\u_{\cC_{g,n}})\cong \d(\u_{g,n+1}) \cong
\d_{g,n+1}$. The proof that $\d(\u_{\cC_{g,n}}^\geom) \cong
\d(\u_{g,n+1}^\geom)$ is similar and is left to the reader. The result now
follows from (\ref{eqn:d_arith}).
\end{proof}

\begin{corollary}
\label{cor:d_section}
Each section $x$ of the universal curve $\pi:\cC_{g,n/k} \to \M_{g,n/k}$ induces
a well defined section of $\d_{g,n+1} \to \d_{g,n}$.
\end{corollary}

\begin{proof}
Proposition~\ref{prop:section} implies that each section $x$ of the universal
curve induces a section $\sigma_x$ of $\pi_\ast :\cG_{\cC_{g,n}} \to \cG_{g,n}$
that is well defined up to conjugation by an element of $\cP := \ker \pi_\ast$.
The section $x$ therefore induces a section $d\sigma_x$ of $d\pi_\ast :
\g_{\cC_{g,n}} \to \g_{g,n}$ which is well defined up to addition of a section
of the form $\ad_u\circ\sigma_x$, where $u\in \p:=\ker d\pi_\ast$. Since
$\p=W_{-1}\p$, the sections $d\sigma_x$ and $d\sigma+\ad_u\circ\sigma_x$ induce
the same section of $\Gr^W_\dot\g_{\cC_{g,n}} \to \Gr^W_\dot\g_{g,n}$, and
therefore of $\Gr^W_\dot\u_{\cC_{g,n}}/W_{-3} \to \Gr^W_\dot\u_{g,n}/W_{-3}$.
The result now follows from Proposition~\ref{prop:d_invariance}.

Equality with the geometric case follows from
equation (\ref{eqn:d_arith}) for $\M_{g,n+1}$.
\end{proof}

\subsection{Splittings}

Elements of $\Lambda^3_n H$ will be denoted by $(v;u_1,\dots,u_n)$, where $v
\in \Lambda^3_0 H$ and $u_j \in H$. When the role of $v$ is clear, it will be
omitted. For all $n\ge 0$, the linear mapping
$$
\Gr^W_{-1} \d_{g,n+1} \to \Gr^W_{-1}\d_{g,n}
$$
defined by $(u_0,\dots,u_n) \mapsto (u_1,\dots,u_n)$ induces a
$\GSp(H)$-equivariant homomorphism $\e_n : \d_{g,n+1} \to \d_{g,n}$.

\begin{proposition}
\label{prop:d_sections}
If $g\ge 4$, then there are exactly $n$ $\GSp(H)$-invariant sections of $\e_n$.
When $n\ge 1$, these are defined by
$$
s_j : (u_1,\dots, u_n) \mapsto (u_j, u_1,\dots,u_n)\qquad j = 1,\dots,n.
$$
When $g=3$, the sections of $\e_n$ are $s_1,\dots,s_n$ and the section
$$
\zeta_n : (u_1,\dots,u_n) \mapsto (0,u_1,\dots,u_n).
$$
\end{proposition}

\begin{proof}
Suppose that $n\ge 1$ and that $g\ge 3$. Suppose that $s : \d_{g,n} \to
\d_{g,n+1}$ is a $\GSp(H)$-invariant section of $\e_n : \d_{g,n+1} \to
\d_{g,n}$. Since $H$ is an irreducible $\GSp(H)$-module, Schur's Lemma implies
that the restriction of $s$ to the subgroup $H^n$ of $\d_{g,n}$ is of the form
$$
s(u_1,\dots,u_n) = \big(\textstyle{\sum}_{j=1}^n a_j u_j, u_1,\dots,u_n\big)
$$
for some $a_1,\dots,a_n \in F$. Since $\e_n$ is the identity on $\Lambda^3_0 H$,
$s$ is as well. So, since the bracket
$$
H_j \otimes \Lambda^3_0 H \to \Lambda^2_0 H_j \qquad 1 \le j \le n
$$
is surjective, the map
$$
\Lambda_0^2 H_j \to \Lambda_0^2 H_0 \qquad 1 \le j \le n
$$
induced by $s$ must also be multiplication by $a_j$.

\begin{lemma}
\label{lem:cocycle}
If $g\ge 3$ and $v',v''\in \Lambda^3_0 H$ and that $u_1',u_1'',\dots,u_n',u_n''
\in H$, then the component of
$$
[s(v';u_1',\dots,u_n'),s(v'';u_1'',\dots,u_n'')]
-s\big([(v';u_1',\dots,u_n'),(v'';u_1'',\dots,u_n'')]\big)
$$
in $\Lambda^2_0 H_0$ is
\begin{equation}
\label{eqn:cocycle}
(1-\textstyle{\sum_j a_i})\cc(v'\wedge v'')
+ \sum_j(a_j^2-a_j)\ee_j(u_j'\wedge u_j'')
+ \sum_{i\neq j} a_i a_j \ee_{ij}(u_i'\wedge u_j'')
\end{equation}
\end{lemma}

\begin{remark}
This result is reinterpreted in terms of non-abelian cohomology in
Lemma~\ref{lem:delta}.
\end{remark}

Since $s$ is a section, the expression (\ref{eqn:cocycle}) must vanish. This
implies that $a_i^2 = a_i$ for all $i$ and that $a_ia_j = 0$ when $i\neq j$.
That is, each $a_i \in \{0,1\}$ and at most one $a_i$ is non-zero. When $g\ge
4$, $\cc\neq 0$. Consequently, we have the additional equation
$$
\sum_{j=1}^n a_j = 1
$$
which implies that exactly one of the $a_j$ is non-zero. This proves the result
when $g\ge 4$. When $g=3$, $\cc$ is zero and we also have the section $\zeta_n$
where all $a_j$ are zero.
\end{proof}

\begin{proof}[Proof of Lemma~\ref{lem:cocycle}]
Proposition~\ref{prop:bracket} implies that
\begin{multline*}
p_0\big([s(v';u_1',\dots,u_n'),s(v'';u_1'',\dots,u_n'')])
= \cc(v'\wedge v'') \cr
+ \sum_j a_j \big(\dd_j(v'\otimes u_j'') - \dd_j(v''\otimes u_j')\big)
+ \sum_{j} a_j^2 \ee_{j}(u_j'\wedge u_j'')
+ \sum_{i\neq j} a_ia_j \ee_{ij}(u_i'\wedge u_j'')
\end{multline*}
and that
\begin{multline*}
p_0 \circ s\big([(v';u_1',\dots,u_n'),(v'';u_1'',\dots,u_n'')]\big) \cr
=\sum_j a_j \big(\ee(u_j'\wedge u_j'') + \dd(v'\otimes u_j'')
- \dd(v''\otimes u_j') + \cc(v'\wedge v'')\big).
\end{multline*}
The expression (\ref{eqn:cocycle}) is their difference.
\end{proof}

\subsection{Genus 3}
\label{para:genus3}

As we have seen, the section $\zeta_n$ of $\d_{g,n+1}\to \d_{g,n}$ exists only
when $g=3$. This exceptional section occurs because $\Lambda^2 \Lambda^3_0 H$
does not contain $\Lambda^2_0 H$ when $g=3$. To circumvent this problem, we
consider a natural integral structure on $\Gr^W_{-1}\u_{g,n}$ that comes from
the image of the Torelli group in $\U_{g,n}$. This will allow us to prove that
$\zeta_n$ cannot be induced by a rational point.

Suppose that $A$ is an integral domain with fraction field $F$. Recall that
$$
\Lambda^3_0 H_A := (\Lambda^3 H_A)(-1)/\thetadual\wedge H_A \cong
\im\{\Lambda^3 H_A \to \Lambda^3_0 H_F\}.
$$
For a positive integer $n$, define a lattice in $\Lambda^3_n H_F$ by
$$
\Lambda^3_n H_A = 
\{(u_1,\dots,u_n)\in (\Lambda^3 H_A)^n : \ubar_1 = \dots = \ubar_n\}(-1).
$$
Define the projection $r : \Lambda^3_{n+1} H_A \to \Lambda^3_n H_A$ by
$r:(u_0,u_1,\dots,u_n) \mapsto (u_1,\dots, u_n)$.

The following is a straightforward generalization of Lemma~\ref{lem:split}.

\begin{lemma}
\label{lem:lattice}
Suppose that $g\ge 3$ and $n\ge 0$. If $A$ is an integral domain and $g-1 \notin
A^\times$, then $\zeta_n$ does not restrict to a section of the projection
$r:\Lambda^3_{n+1} H_A \to \Lambda^3_n H_A$. \qed
\end{lemma}

Now take the coefficient field $F$ to be $\Ql$. The homomorphism
(\ref{eqn:johnson}) induces natural isomorphisms
$$
\Gr^W_{-1}\u_{g,1} \cong \Lambda^3_1 H_\Ql \cong (\Lambda^3_1 H_\Zl)\otimes\Ql
$$
These induce a natural isomorphism
$$
\Gr^W_{-1}\u_{g,n} \cong (\Lambda^3_n H_\Zl)\otimes\Ql
$$
for all $n\ge 0$ such that the $j$th projection $r:\Lambda^3_{n+1}H_\Zl \to
\Lambda^3_n H_\Zl$ corresponds to the $j$th projection $\Gr^W_{-1}\u_{g,n+1} \to
\Gr^W_{-1}\u_{g,n}$.

\subsection{Summary}

The following result summarizes the results of Section~\ref{sec:la_d}.

\begin{proposition}
\label{prop:d_sections3}
Suppose that $n\ge 1$. If $g\ge 4$ and $\ell$ is arbitrary, or if $g=3$ and
$\ell = 2$, then the only $\GSp(H_\Ql)$-invariant sections of $\e_n : \d_{g,n+1}
\to \d_{g,n}$ that respect the integral structure on $\Gr^W_{-1}$ are the
sections $s_1,\dots,s_n$ induced by the tautological points. In particular,
there are no $\GSp(H)$-invariant sections of $\d_{g,1}\to\d_{g,0}$.
\end{proposition}

\begin{proof}
Corollary~\ref{cor:d_section} implies that each tautological section induces a
$\GSp(H)$-invariant section of $e_n$. When $n=1$, the result follows from
Lemma~\ref{lem:lattice} when $g=3$ and Proposition~\ref{prop:d_sections} when
$g\ge 4$.  The case $n>1$ follows from the case $n=1$ and naturality.
\end{proof}

\section{Topologically Ample Families of Curves}
\label{sec:top_amp}

Roughly speaking, a family $C\to T$ of smooth projective curves, where $T$ is a
smooth, geometrically connected $k$ variety, is topologically ample of type
$(g,n)$ if $C$ is the pullback of the universal curve along a map $T \to
\M_{g,n/k}$ that induces an isomorphism on fundamental groups. The actual
condition is weaker.

\begin{definition}
\label{def:top_ample}
Suppose that $k$ is a field of characteristic zero and that $T$ is a smooth,
geometrically connected $k$-variety. A family $C\to T$ of smooth projective
curves of genus $g\ge 3$ is {\em topologically ample over $k$ of type $(g,n)$} 
if there is a $k$-morphism $\phi : T \to \M_{g,n/k}$, called a {\em marking},
and a prime number $\ell$ satisfying:
\begin{enumerate}

\item $C$ is isomorphic to the pullback of the universal curve $\cC \to
\M_{g,n/k}$; equivalently, there is a morphism $\phihat : C \to \cC_{g,n/k}$
such that
$$
\xymatrix{
C \ar[r]^(.45)\phihat\ar[d] & \cC_{g,n/k}\ar[d] \cr
T \ar[r]^(.45)\phi & \M_{g,n/k}
}
$$
commutes;

\item the monodromy representation $\rho : \pi_1(T,\etabar) \to \GSp(H_\Ql)$ is
Zariski dense, where $\etabar$ is a geometric point of $T$ and
$H=\Het^1(C_{\etabar},\Ql(1))$;

\item the Lie algebra homomorphism $\phi_\ast : \u_T \to \u_{g,n}$ induced by
$\phi$ induces an isomorphism $\u_T/W_{-3}\to \u_{g,n}/W_{-3}$, where $\u_T$ and
$\u_{g,n}$ denote the Lie algebras of the prounipotent radicals of the weighted
completions of $\pi_1(T,\xbar)$ and $\pi_1(\M_{g,n/k},\phi(\xbar))$ with respect
to their homomorphisms to $\GSp(H_\Ql)$;

\item The group of sections $\Jac_{C/T}(T)$ of the relative jacobian $\Jac_{C/T}
\to T$ is finitely generated.

\end{enumerate}
When $g=3$, we also require that $\ell=2$ and that the image of the natural
homomorphism
$$
\ker\{\pi_1(T,\etabar)\to \GSp(H_{\Q_2})\} \to \Gr^W_{-1}\u_{g,n}
$$
equals the copy of $\Lambda^3_n H_{\Z_2}$ defined in
Paragraph~\ref{para:genus3}. A {\em  marked topologically ample curve of type
$(g,n)$} is a topologically ample curve $C\to T$ of type $(g,n)$ together with a
choice of marking $\phi : T \to \M_{g,n/k}$.
\end{definition}

Every marked topologically ample family of curves $C\to T$ of type $(g,n)$ has
$n$ tautological points $x_1,\dots,x_n$. These are the pullbacks of the
$n$-tautological sections of the universal curve over $\M_{g,n/k}$ along the
marking $\phi$.

Finite generation of $\Jac_{C/T}(T)$ occurs when $k$ is small and also when the
image of $\pi_1(T\otimes_k \kbar,\etabar) \to \GSp(H_\Ql)$ is large enough. It
fails, for example, when $k$ is algebraically closed and $T=\Spec k$.

\begin{proposition}
\label{prop:fg}
Suppose that $T$ is a geometrically connected, smooth $k$-variety. If either $k$
is finitely generated over $\Q$ or if $H^0(T\otimes_k \kbar,\H_\Ql) = 0$, then
$\Jac_{C/T}(T)$ is finitely generated.
\end{proposition}

\begin{proof}
The case when $k$ is finitely generated over $\Q$ follows from the Mordell-Weil
Theorem \cite{weil}. To prove the second assertion, it suffices, by standard
arguments, to prove it in the case $k=\C$. Denote the category of mixed Hodge
structures (MHS) by $\cH$ and the category of admissible variations of MHS over
$T$ by $\cH(T)$. In this case, the Abel-Jacobi mapping
$$
\Jac_{C/T}(T) \to \Ext^1_{\cH(T)}(\Z(0)_T,\H_\Z)
$$
is injective. It is proved in the Appendix that the sequence
$$
0 \to \Ext^1_\cH(\Z(0),H^0(T,\H_\Z)) \to \Ext^1_{\cH(T)}(\Z(0)_T,\H_\Z)
\to H^1(T,\H_\Z)
$$
is exact. Since $\H_\Z$ is torsion free, $H^0(T,\H_\Z)$ is a subgroup of
$H^0(T,\H_\Ql)$, and therefore vanishes. Since $T$ is homotopy equivalent to a
finite complex, $H^1(T,\H_\Z)$ is finitely generated. It follows that
$\Jac_{C/T}(T)$ is finitely generated. 
\end{proof}

Proposition~\ref{prop:level-indep} implies that when $g\ge 4$, the universal
curve over $\M_{g,n/k}[m]$ is topologically ample whenever there is a prime
number $\ell$ such that $\chi_\ell : G_k \to \Zlx$ has infinite image. When
$g=3$ it implies that $\M_{3,n/k}[m]$ is topologically ample when the image of
$\chi_2$ is infinite. This condition is satisfied, for example, when $k$ is a
number field. Recall that $\Mo_{g,n}[m]$ denotes the open subset of
$\M_{g,n}[m]$ of curves with trivial automorphism group. As indicated in
Section~\ref{sec:level}, it is a smooth quasi-projective variety whenever $g\ge
3$.

\begin{proposition}
Suppose that $g\ge 3$, $n\ge 0$, $m\ge 1$ and that either $g+n\ge 4$ or $m\ge
3$. If there exists a prime number $\ell$ such that the $\ell$-adic cyclotomic
character $\chi_\ell : G_k \to \Zlx$ has infinite image and if $T$ is a generic
ample subvariety of $\Mo_{g,n/k}[m]$ of dimension $\ge 2$, then the pullback of
the universal curve over $\M_{g,n}$ to $T$ is topologically ample of type
$(g,n)$.
\end{proposition}

\begin{proof}
If $m\ge 3$, then $\Mo_{g,n}[m] = \M_{g,n}[m]$. Proposition~\ref{prop:autos}
implies that if $g+n\ge 4$, then $\Mo_{g,n/k}[m]$ is a smooth quasi-projective
variety whose geometric fundamental group is isomorphic to the profinite
completion of the mapping class group $\G_{g,n}[m]$. Choose a projective
imbedding $\Mo_{g,n/k}[m]\hookrightarrow \P^N_{k}$. Denote the
Zariski closure of the image by $X$. Set
$$
r = N-(3g-3+n) + d.
$$
When $d\ge 0$, there is a Zariski open subset $U_d$ of the Grassmannian
$G_r(\P^N)$ consisting of the $r$ planes in $\P^N$ that intersect $X$ and all of
the strata of $X-\M_{g,n}$ transversely. If $L\in U_d(k)$, then $L\cap X$ is a
subvariety of $X$ of dimension $d$.

Observe that $U_d(k)$ is non-empty. If $d\ge 2$ and $L\in U_d(k)$, then, by
\cite[Thm.~p.~150]{smt}, $\M_{g,n}\cap L$ is geometrically connected $k$-variety
of dimension $d$ and the inclusion
$$
T := \Mo_{g,n/k}[m]\cap L \hookrightarrow \Mo_{g,n/k}[m]
$$
induces an isomorphism on geometric fundamental groups. It thus induces an
isomorphism on \'etale fundamental groups. Let $C\to T$ be the restriction of
the universal curve over $\M_{g,n/k}$ to $T$. Since the $\ell$-adic cyclotomic
character $\chi_\ell$ has infinite image, this implies that the representations
$\pi_1(T)\to \GSp(H)$ and $\pi_1(C\otimes_k \kbar)\to \Sp(H)$ are Zariski dense.
Proposition~\ref{prop:fg} implies that $\Jac_{C/T}(T)$ is finitely generated. It
follows that $C/T$ is topologically ample of type $(g,n)$.
\end{proof}

\begin{remark}
\label{rem:coho}
Suppose that $C/T$ is topologically ample of type $(g,n)$ with $g\ge 3$. Denote
the lisse sheaf of $\Ql$ vector spaces over $T$ that corresponds to the
$\GSp(H)$-module $V$ by $\V$. Proposition~\ref{prop:homology} implies that if
$V$ has negative weight, then one has isomorphisms
\begin{multline*}
\Het^1(T,\V) \cong \Hom_{\GSp(H)}(H_1(\u_T),V) \cr
\cong \Hom_{\GSp(H)}(H_1(\u_{g,n}),V) \cong \Het^1(\M_{g,n/k},\V).
\end{multline*}
In particular, $\Het^1(T,\Ql(1)) \cong \Het^1(\M_{g,n/k},\Ql(1)) = 0$.
\end{remark}

Recall that $\cC_{g,n/k}$ denotes the universal complete curve over
$\M_{g,n/k}$.

\begin{proposition}
\label{prop:ample_curve}
If $C/T$ is topologically ample over $k$ of type $(g,n)$, then $\phihat$ induces
an isomorphism $\phihat_\ast: \u_C/W_{-3} \overset{\simeq}{\to}
\u_{\cC_{g,n/k}}/W_{-3}$ such that the diagram
$$
\xymatrix{
\u_C/W_{-3} \ar[r]_\simeq^{\phihat_\ast} \ar[d] & \u_{\cC_{g,n/k}}/W_{-3}
\ar[d] \cr
\u_T/W_{-3}  \ar[r]_\simeq^{\phi_\ast} & \u_{g,n}/W_{-3}
}
$$
commutes.
\end{proposition}

\begin{proof}
This an immediate consequence of the commutativity of the diagram
$$
\xymatrix{
0 \ar[r] & \p \ar[r]\ar@{=}[d] & \u_C \ar[r]\ar[d] & \u_T \ar[d]^\cong_\phi
\ar[r] & 0 \cr
0 \ar[r] & \p \ar[r] & \u_{\cC_{g,n}} \ar[r] & \u_{g,n} \ar[r] & 0
}
$$
whose rows are exact by Proposition~\ref{prop:section}.
\end{proof}

The homomorphism $\phi_\ast : \u_T \to \u_{g,n}$ induces an isomorphism
$\phi_\ast : \d(\u_T) \to \d_{g,n}$. The next result is a direct consequence of
Propositions~\ref{prop:d_invariance} and \ref{prop:ample_curve}. The proof of
the last assertion is almost identical to that of Corollary~\ref{cor:d_section}
and is left to the reader.

\begin{corollary}
\label{cor:d_curve}
If $C/T$ is topologically ample over $k$ of type $(g,n)$, then there is a
$\GSp(H)$-equivariant isomorphism $\d(\u_C) \cong \d_{g,n+1}$
such that the diagram
$$
\xymatrix{
\d(\u_C) \ar[d]\ar[r]^\simeq & \d_{g,n+1}\ar[d] \cr
\d(\u_T) \ar[r]^\simeq & \d_{g,n}
}
$$
commutes. Each section $x\in C(T)$ induces a $\GSp(H)$-invariant section of
$\d(\u_C) \to \d(\u_T)$ which preserves the integral lattice on $\Gr^W_{-1}$.
\end{corollary}

\section{The Class of a Rational Point}
\label{sec:kappa}

In this section we associate a cohomology class $\kappa_x \in H^1(T,\H)$ to a
section $x \in C(T)$ of a family of smooth projective curves $C \to T$. In the
universal case --- viz., the universal curve over $\M_{g,1}$ --- this class is
easily expressed in terms of the $\GSp(H)$-structure of $\Gr^W_{-1}\u_{g,1}$.

Suppose that $T$ is a smooth, geometrically connected $k$-scheme and that
$f:X\to T$ is a family of smooth projective varieties of relative dimension $d$
over $T$.  The polarization gives an isomorphism $R^1 f_\ast\Ql(1)\cong
R^{2d-1}f_\ast\Ql(d)$. By well-known arguments (cf.\ \cite[\S
4]{hain-matsumoto:cycles}), each relative algebraic cycle $\xi$ in $X$ of
dimension $0$ over $T$ that is homologically trivial on each geometric fiber of
$f$, determines a well defined class
$$
c_\xi \in \Het^1(T,R^{2d-1}f_\ast\Ql(d)) \cong \Het^1(T,R^1 f_\ast\Ql(1)).
$$
In particular, if $x_1,\dots,x_n$ are sections of a family $C \to T$ of smooth
projective curves of genus $g\ge 3$ and $d_1,\dots,d_n$ are integers which sum
to zero, then we have the class of the relative divisor $\xi = \sum d_jx_j$.

Denote the relative canonical divisor of $C/T$ by $\w_{C/T}$. If $x\in C(T)$,
then one has the algebraic 0-cycle class $(2g-2)x - \w_{C/T}$, which is
homologically trivial on each geometric fiber. Its class
$$
\kappa_x \in \Het^1(T,\H_\Ql)
$$
is well defined. In particular, one has the class
$$
\kappa \in \Het^1(\M_{g,1/k},\H_\Ql)
$$
of the tautological section of the universal curve over $\M_{g,1/k}$. It is
universal in the sense that each $x\in C(T)$ is classified by a $k$-morphism
$\phi : T \to \M_{g,1/k}$ and $\kappa_x = \phi^\ast \kappa$. Denote the class of
the $j$th tautological point of the universal curve over $\M_{g,n/k}$ by
$\kappa_j$.

\begin{proposition}
If $g\ge 3$, $n\ge 0$, and $m\ge 1$, then for all fields of $k$ characteristic
zero
$$
\Het^1(\M_{g,n/k}[m],\H_\Ql)
= \Ql \kappa_1 \oplus \Ql \kappa_2 \oplus \dots \oplus \Ql \kappa_n
$$
\end{proposition}

\begin{proof}
When $n=1$, this follows directly form \cite[Thm.~6.9]{hain-matsumoto:cycles}.
The general case follows easily from this, naturality, and the computation
\cite[Prop.~5.2]{hain:msri}.
\end{proof}

If $\pi: A \to T$ is a family of polarized abelian varieties over a smooth
$k$-variety $T$, then the Abel-Jacobi mapping induces a homomorphism
$$
A(T) \to \Het^1(T,R^{2g-1}\pi_\ast\Ql(g)) \cong \Het^1(T,R^1\pi_\ast\Ql(1)).
$$

A direct proof of Corollary~\ref{cor:kim} below was communicated to me by
Minhyong Kim; the current formulation, using the lemma, was contributed by the
referee.

\begin{lemma}
If $\pi : A \to T$ is a family of polarized abelian varieties over a noetherian
scheme $T$, then the kernel of the $\ell$-adic Abel-Jacobi map
$$
A(T) \to \Het^1(T,R^1\pi_\ast\Zl(1))
$$
is the subgroup $\bigcap_n \ell^n A(T)$ of $\ell^\infty$-divisible points.
\end{lemma}

\begin{proof}
The connecting homomorphism of the long exact cohomology sequence associated to
the exact sequence of \'etale sheaves
$$
0 \to A[\ell^n] \to A \overset{\ell^n}{\to} A \to 0
$$
induces an injection $A(T)/\ell^n A(T) \hookrightarrow \Het^1(T,A[\ell^n])$.
Since $T$ is noetherian, \cite[Thm~1.1, p.~233]{sga} implies that
$H^0(T,A[\ell^n])$ is finite for each $n\ge 1$. This implies that
$\varprojlim^1 H^0(T,A[\ell^n])$ vanishes and that the natural surjection
$$
\Het^1(T,\varprojlim_n A[\ell^n]) \to \varprojlim_n \Het^1(T, A[\ell^n])
$$
is an isomorphism. The result now follows as the Weil pairing induces an
isomorphism $R^1\pi_\ast\Zl(1) \to \varprojlim_n A[\ell^n]$.
\end{proof}

\begin{corollary}
\label{cor:kim}
If the group $A(T)$ of sections of $\pi:A\to T$ is finitely generated, then the
kernel of
$$
A(T) \to \Het^1(T,R^1\pi_\ast\Ql(1))
$$
is finite for all $\ell$.
\end{corollary}

The assumption that $A(T)$ be finitely generated is necessary, as can be seen by
considering the case where $T$ is the spectrum of an algebraically closed field.

We can apply this to the family of jacobians $\pi:\Jac_{C/T} \to T$ associated
to the family of curves $C\to T$.

\begin{corollary}
\label{cor:torsion}
Suppose that $x$ and $y$ are sections of $C$ over $T$. If $\Jac_{C/T}(T)$ is
finitely generated, then $\kappa_x = \kappa_y$ implies that $x-y$ is torsion in
$\Jac_{C/T}(T)$.
\end{corollary}

\begin{proof}
Since $R^1 \pi_\ast \Zl(1)$ is naturally isomorphic to $\H_\Zl$, the diagram
$$
\xymatrix{
C(T) \ar[dr]\ar[d] \cr
\Jac_{C/T}(T) \ar[r] & \Het^1(T,\H_\Ql)
}
$$
commutes, where the left-hand vertical map takes $x$ to the class of
$(2g-2)x-\w_{C/T}$ whose image under the horizontal map is $\kappa_x$. Since the
bottom map is injective mod torsion, $\kappa_x = \kappa_y$ implies that the
divisor $(2g-2)(x-y)$ is torsion.
\end{proof}

We conclude this section by relating the class $\kappa_x$ of a section of $C\to
T$ to the structure of $H_1(\u_T)$. Proposition~\ref{prop:homology} implies that
there is an isomorphism
$$
\Het^1(T,\H) \cong \Hom_{\GSp(H)}(H_1(\u_T),\H).
$$

The $\GSp(H)$-equivariant projection $h:\Lambda^3_1 H \to H$ is the twist of the
map $\Lambda^3 H \to H(1)$ defined by $x\wedge y\wedge z \mapsto \theta(x,y)z +
\theta(y,z)x + \theta(z,x)y$.

The following assertion follows from \cite[Prop~6.5]{hain-matsumoto:cycles} and
its proof.

\begin{proposition}
\label{prop:kappa_projn}
If $g\ge 3$, the $\GSp(H)$-invariant homomorphism
$$
H_1(\u_{g,1}) \to \Gr^W_{-1} \u_{g,1} \cong \Lambda^3_1 H \overset{2h}{\to} H
$$
corresponds to the universal class $\kappa$ under the isomorphism
$$
\Het^1(\M_{g,1/k},\H_\Ql) \cong \Hom_{\GSp(H)}(H_1(\u_{g,1}),H).
$$
\end{proposition}

Let $h_j : \Lambda^3_n H \to H$ be the $\GSp(H)$-invariant homomorphism
$$
\xymatrix{
\Lambda^3_n H \ar[r] & (\Lambda^3_1 H)^n \ar[r]^(.55){pr_j}
& \Lambda^3_1 H \ar[r]^(.55)h & H
} 
$$

\begin{corollary}
\label{cor:kappa}
When $g\ge 3$ and $n\ge 1$, the class $\kappa_j \in \Het^1(\M_{g,n/k},\H_\Ql)$
corresponds to the homomorphism $2h_j$ under the isomorphism
$$
\Het^1(\M_{g,n/k},\H_\Ql) \cong \Hom_{\GSp(H)}(H_1(\u_{g,n}),H).
$$
\end{corollary}

Recall from the proof of Lemma~\ref{lem:split} that the composite $h\circ
(\_\wedge \thetadual) : H\to \Lambda^3_1 H \to H$ is $(g-1)\id_H$. It follows
that the projection $\Lambda^3_n H \to H$ that corresponds to the projection
$$
\Lambda^3_n H = \Lambda^3_0 H \oplus H^n \to H
$$
onto the $j$th copy of $H$ is $h_j/(g-1)$. This projection corresponds to
$\kappa_j/(2g-2)$.

\begin{remark}
One of the issues in trying to prove the section conjecture is to determine
which classes in $\Pic^1 C$ are the classes of rational points. To provide a
context for the results of the next section, we mention that when $C$ is the
pullback of the universal curve to $\Spec k(\M_{g,n/k}[m])$
$$
(\Pic^0 C)\big(k(\M_{g,n/k}[m])\big) \cong
\bigoplus_{j=1}^n \Z(\kappa_j/\epsilon_m) \oplus
\bigoplus_{1\le j<j'\le n} \Z[x_j-x_{j'}] \oplus (\Z/m\Z)^{2g}
$$
where $\epsilon_m$ is $2$ when $m$ is even and $1$ when $m$ is
odd.\footnote{More precisely, when $m$ is even the relative canonical bundle has
$2^{2g}$ square roots (``theta characteristics''), any two of which differ by a
point of order 2. By $\kappa_j/2$ we mean the class of $(g-1)x_j - \alpha$ for a
choice of theta characteristic $\alpha$. This is well defined mod 2-torsion.}
This is proved in \cite[Thm.~12.3]{hain:msri} when $k = \C$; it implies the
result for all algebraically closed fields $k$ of characteristic zero. The case
for a general field $k$ of characteristic zero follows as all the generators are
represented by $\Q$-rational divisors.
\end{remark}

\section{Computation of $C(k(T))$}

Given a family of curves $C\to T$, one encounters the problem of distinguishing
the sections of $\Pic^1 C \to T$ that correspond to $k(T)$-rational points of
$C$ from those that do not. This distinction can be made when $C/T$ is
topologically ample.

\begin{lemma}
\label{lem:no_sect}
If $g\ge 3$ and $n\ge 1$, there is no $\GSp(H)$-equivariant Lie algebra
homomorphism
$$
\Gr^W_\dot \u_{g,n}^\geom/W_{-3} \to \Gr^W_\dot \u_{g,2}^\geom/W_{-3}
$$
that induces the map $(v;u_1,\dots,u_n) \mapsto (v;u_1,u_1)$ on $\Gr^W_{-1}$.
\end{lemma}

\begin{proof}
Suppose that such a homomorphism $\phi : \Gr^W_\dot \u_{g,n}^\geom/W_{-3} \to
\Gr^W_\dot \u_{g,2}^\geom/W_{-3}$ exists. Denote the $ij$th copy ($1\le  i < j
\le n$) of $\Ql(1)$ in $\Gr^W_{-2}\u_{g,n}^\geom$ by $\Ql(1)_{ij}$. Note that
there is a unique copy of $\Ql(1)$ in $\Gr^W_{-2}\u_{g,2}^\geom$. Denote the
$j$th copy of $H$ in $\Gr^W_{-1}\u_{g,n}^\geom$ by $H_j$.
Proposition~\ref{prop:bracket} implies that
$$
\xymatrix{
\Lambda^2 \Lambda^3_0 H \ar[r]^(.42){\text{inclusion}} &
\Lambda^2 \Gr^W_{-1}\u_{g,2}^\geom \ar[r]^(.53)\bracket &
\Gr^W_{-2}\u_{g,2}^\geom \ar[r]^(.55){q_{12}} & \Ql(1)
}
$$ 
is surjective. When $n=1$, there is no copy of $\Ql(1)$ in $\Gr^W_{-2}\u_{g,n}$,
from which it follows that there is no such homomorphism $\phi$ when $n=1$. 

Now suppose that $n>1$. In this case Proposition~\ref{prop:bracket} implies
that the image of
\begin{equation}
\label{eqn:inv_bracket}
\xymatrix{
\Lambda^2 \Lambda^3_0 H \ar[r]^(.42){\text{inclusion}} &
\Lambda^2 \Gr^W_{-1}\u_{g,n}^\geom \ar[r]^(.53)\bracket &
\Gr^W_{-2}\u_{g,n}^\geom \ar[r]^(.46){\oplus q_{ij}} &
\bigoplus_{i<j}\Ql(1)_{ij}
}
\end{equation} 
is the diagonal copy of $\Ql(1)$. It follows that $\phi$ induces a surjective
homomorphism
$$
\big(\Gr^W_{-2}\u_{g,n}^\geom\big)^{\Sp(H)} \to
\big(\Gr^W_{-2}\u_{g,2}^\geom\big)^{\Sp(H)}.
$$
On the other hand, Proposition~\ref{prop:bracket} implies that, when $i<j$, the
component
$$
\xymatrix{
H_i\otimes H_j \ar[r]^(.42){\text{inclusion}} &
\Lambda^2 \Gr^W_{-1}\u_{g,n}^\geom \ar[r]^(.53)\bracket &
\Gr^W_{-2}\u_{g,n} \ar[r]^{q_{ij}} & \Ql(1)_{ij}
}
$$
of the bracket  is surjective. Since $j>1$, $H_j$ is contained in the kernel of
$\phi$. These two facts imply that the restriction of $\phi$ to $\Ql(1)_{ij}$ is
zero whenever $1\le i<j \le n$, which contradicts the non-triviality of
(\ref{eqn:inv_bracket}). It follows that there is no such homomorphism $\phi$
when $n>1$, which completes the proof.
\end{proof}

Topologically ample curves have the nice property that their $T$-rational
points are determined by their classes in $H^1(T,\H_\Ql)$.

\begin{proposition}
\label{prop:rigidity}
Suppose that $C/T$ is a topologically ample curve of type $(g,n)$ where $g\ge 3$
and $n\ge 1$. If $x\in C(k(T))$ and $\kappa_x=\kappa_j$, then $x$ is the $j$th
tautological point $x_j$.
\end{proposition}

\begin{proof}
We may assume that $j=1$. Corollary~\ref{cor:torsion} implies that $[x]-[x_1]$
is torsion in $(\Pic^0C)(k(T))$. Denote this torsion point by $\delta$. If
$\delta = 0$, then $x=x_1$ and we are done. By replacing $T$ by $T-Z$ if
necessary, where $Z$ is a subvariety of $T$ of codimension $\ge 2$, results of
Section~\ref{sec:points} imply that we may assume that $x\in C(T)$. If $\delta
\neq 0$, then the sections $x$ and $x_1$ of $C$ are disjoint, and there is a
$k$-morphism $T \to \M_{g,2}$ defined by $x \mapsto [C_x;x_1,x]$. This induces a
$\GSp(H)$-equivariant homomorphism $\phi : \Gr^W_\dot \u_T \to \Gr^W_\dot
\u_{g,2}$. Since $C/T$ is topologically ample, the marking induces an
identification $\u_T/W_{-3} \cong \u_{g,n}/W_{-3}$. Since $\kappa_x = \kappa_1$,
the induced mapping
$$
\xymatrix{
\Gr^W_{-1}\u_T \ar[r]^{\phi_\ast}\ar[d]^\cong &
\Gr^W_{-1}\u_{g,2} \ar[d]^\cong \cr
\Lambda^3_0 H \oplus H^n \ar[r]^\psi & \Lambda^3_0 \oplus H^2
}
$$
on $\Gr^W_{-1}$ is defined by $\psi : (v;u_1,\dots,u_n) \mapsto (v;u_1,u_1)$.
But by Lemma~\ref{lem:no_sect}, there is no such homomorphism. It follows that
$\delta=0$ and that $x=x_1$.
\end{proof}

\subsection{Proof of Theorem~\ref{thm:rational}}
If $C(k(T))$ is empty, $n$ must be zero and the result is true.  Suppose that
$x\in C(k(T))$. Corollary~\ref{cor:splitting} implies that $x$ induces a
splitting of the homomorphism $\pi_1(C,\etabar_C) \to \pi_1(T,\etabar_T)$ for
some choice of basepoints. Corollary~\ref{cor:d_curve} implies that $x$ induces
a $\GSp(H)$-equivariant splitting of $\d(\u_C) \to \d(\u_T)$ which, when $g=3$,
preserves the $\Z_2$ lattice structure on $\Gr^W_{-1}$.
Proposition~\ref{prop:d_sections3} implies that $n\ge 1$ and that this section
equals the section induced by the tautological point $x_j$ for some $j\in
\{1,\dots,n\}$. Corollary~\ref{cor:kappa} implies that $\kappa_x = \kappa_j$.
Proposition~\ref{prop:rigidity} and Remark~\ref{rem:stix} now imply that
$x=x_j$.

\section{Non-abelian Cohomology}
\label{sec:nab}

This section is a brief review of the non-abelian cohomology of proalgebraic
groups developed in \cite{hain:nab}, which was inspired by, and is a variant of,
the non-abelian cohomology developed by Kim in \cite{kim:coho}. This non-abelian
cohomology is the principal tool used in the proof of Theorem~\ref{thm:sec_pos}.
This version of non-abelian cohomology should be useful in studying rational
points of curves and other non-abelian varieties as it is somewhat computable,
as we demonstrate in Section~\ref{sec:thm3}.

Suppose that $F$ is a field of characteristic zero and that $\cG$ is a
negatively weighted extension of a connected, reductive $F$-group $R$ with
respect to a non-trivial central cocharacter $\w : \Gm \to R$. Suppose that
$\cP$ is a prounipotent $F$-group with trivial center. Fix an {\em outer} action
$\phi : \cG \to \Out \cP$ of $\cG$ on $\cP$. This induces an action of $\cG$ on
$H_1(\cP)$, so that $H_1(\cP)$ has a natural weight filtration. We will assume
that this is negatively weighted --- i.e., $H_1(\cP)=W_{-1}H_1(\cP)$ --- and
that each of its weight graded quotients is finite dimensional. Under these
conditions, $\cP$ has a natural weight filtration with finite dimensional
quotients and the group of weight filtration preserving automorphisms
$\Aut_W\cP$ of $\cP$ is a proalgebraic group. We will define a non-abelian
cohomology scheme $\Hnab^1(\cG,\cP)$, which will represent sections of an
extension of $\cG$ by $\cP$.

The condition that $\cP$ have trivial center implies that this extension is
uniquely determined by the outer action $\phi$. Indeed, since $\cP$ has trivial
center, the sequence
$$
1 \to \cP \to \Aut_W \cP \to \Out_W \cP \to 1,
$$
where the first map takes an element of $\cP$ to the corresponding inner
automorphism, is an exact sequence of proalgebraic $F$-groups. Pulling this
extension back along $\phi$ gives an extension
\begin{equation}
\label{eqn:extn-phi}
1 \to \cP \to \cGhat_\phi \to \cG \to 1
\end{equation}
of $\cG$ by $\cP$. We do not make the conventional assumption that this sequence
splits, nor do we distinguish any one section when there is one.

The assumption that $H_1(\cP)$ be negatively weighted implies $\cGhat_\phi$ is
also a negatively weighted extension of $R$ with respect to $\w$. It also
implies that $W_{-m}\cP$ is a normal subgroup of $\cGhat_\phi$ for all $m\in
\Z$.

Denote the Lie algebra of $\cG$ by $\g$. It acts on the weight graded quotients
of $\p$. Frequently we will impose the finiteness condition
\begin{equation}
\label{eqn:finiteness}
\dim H^1(\g,\Gr^W_{-m} \p) < \infty
\end{equation}
for $m\in\Z$ in a suitable range. The following is our analogue of
\cite[Prop.~2]{kim:coho}.

\begin{theorem}
\label{thm:representability}
Suppose that $N>1$ and that the finiteness condition (\ref{eqn:finiteness})
holds when $1\le m < N$. Under the assumptions above, there is an affine
$F$-scheme of finite type, $\Hnab^1(\cG,\cP/W_{-N}\cP)$, which represents the
functor that takes an $F$-algebra $A$ to the set
$$
\{\text{sections of }(\cGhat_\phi/W_{-N}\cP)\otimes_F A \to \cG\otimes_F A\}/
\text{conjugation by }\cP(A),
$$
where $\cP(A)$ acts on the sections by conjugation.
\end{theorem}

\begin{corollary}
If the finiteness condition (\ref{eqn:finiteness}) holds for all $m\ge 1$, then
the affine $F$-scheme
$$
\Hnab^1(\cG,\cP) := \varprojlim \Hnab^1(\cG,\cP/W_{-N}\cP)
$$
represents the functor
$$
\{\text{sections of }\cGhat_\phi\otimes_F A \to \cG\otimes_F A\}/
\text{conjugation by }\cP(A).
$$
\end{corollary}

There is a non-abelian analogue of an exact sequence which aids in the
computation of $\Hnab^1(\cG,\cP)$. The next result is our analogue of the claim
on page~641 of \cite{kim:coho}.

\begin{theorem}
\label{thm:exactness}
Suppose that $N>1$. If the finiteness condition (\ref{eqn:finiteness}) is
satisfied  when $1\le m \le N$, then
\begin{enumerate}

\item there is a morphism
$\delta : \Hnab^1(\cG,\cP/W_{-N}) {\to} H^2(\g,\Gr^W_{-N}\p)$
of $F$-schemes,

\item there is a principal action of $H^1(\g,\Gr^W_{-N}\p)$ on
$\Hnab^1(\cG,\cP/W_{-N-1})$,

\item for all $F$-algebras $A$,  $\Hnab^1(\cG,\cP/W_{-N})(A)$ is a principal
$H^1(\g,\Gr^W_{-N}\p)(A)$ set over the $A$-rational points $(\delta^{-1}(0))(A)$
of the scheme $\delta^{-1}(0)$.

\end{enumerate}
\end{theorem}

This ``exact sequence'' is represented by the diagram:
{\small
$$
\xymatrix{
\Hnab^1(\cG,\cP/W_{-N-1}) \save !L(.85)="Hnab";
"Hnab",\ar@(ul,dl)"Hnab"_{{H^1(\g,\Gr^W_{-N}\p)}}\restore
\ar[r]^\pi &
\Hnab^1(\cG,\cP/W_{-N}) \ar[r]^{\delta} & H^2(\g,\Gr^W_{-N}\p)
}
$$
}
in which $\pi$ is the projection that takes a section to its quotient by
$\Gr^W_{-N}\cP$.

\begin{remark}
In our version of non-abelian cohomology, we do not need to specify a splitting
of the extension (\ref{eqn:extn-phi}), or even to have one at all. Since this
extension may not split, $\Hnab^1(\cG,\cP)$ may be the empty scheme. Even when
the extension does split, there may be no preferred section and we do not view
these cohomology groups as pointed objects, as is customary. This point of view
is based on the computations in Section~\ref{sec:thm3}, which suggest that
preferring any one section over another maybe unnatural in some contexts.
\end{remark}

\subsection{The scheme $\Hnab^1(\cG,\cP/W_{-N-1})$} 

In Section~\ref{sec:thm3}, we will need to compute the ``connecting morphism''
$\delta$ in situations of interest. Before we explain how to compute $\delta$,
we need to explain why $\Hnab^1(\cG,\cP/W_{-m-1}\cP)$ is a scheme. Full details
can be found in \cite{hain:nab}.

Denote $\cGhat_\phi$ by $\cGhat$. For each $m>0$ and set $\cGhat_m =
\cGhat/W_{-m-1}\cP$. Denote the Lie algebras of $\cG$, $\cGhat$, $\cGhat_m$,
$\cP$ and $\cP/W_{-m-1}$ by $\g$, $\ghat$, $\ghat_m$, $\p$ and $\p_m$,
respectively. Choose a lift $\what : \Gm \to \cGhat$ of $\w$ to $\cGhat$. Denote
its composition with $\cGhat \to \cGhat_m$ by $\what_m$ and with $\cGhat\to\cG$
by $\what_0$. Note that $\what$ splits the weight filtration of all
$\cGhat$-modules. In particular, it splits the weight filtration of each of
these Lie algebras. 

Fix $N\ge 1$. Let $A$ be an $F$-algebra. Define a section $s$ of
$\cGhat_N\otimes_F A \to \cG\otimes_F A$ to be {\em $\what$-graded} if it
commutes with $\what$ in the sense that $\what_N=s\circ \what_0$. Equivalently,
its derivative $ds : \g\otimes_F A \to \ghat\otimes_F A$ is $\what$-graded in
the sense that it commutes with the $\Gm$-actions on $\g$ and $\ghat$ induced by
$\what$.

One has the maps
\begin{multline}
\label{eqn:maps_sects}
\{\what\text{-graded sections of } \ghat_N\otimes_F A \to \g\otimes_F A\} \cr
\cong
\{\what\text{-graded sections of } \cGhat_N\otimes_F A \to \cG\otimes_F A\} \cr
\hookrightarrow
\{\text{sections of } \cGhat_N\otimes_F A \to \cG\otimes_F A\} \cr
\overset{\Gr^W_\dot\! d}{\to}
\{\text{graded sections of }
\Gr^W_\dot\ghat_N\otimes_F A \to \Gr^W_\dot\g\otimes_F A\}
\end{multline}
where the last map takes a section $s$ to the graded Lie algebra section induced
by its derivative $ds$. Note that the composition of these three maps is the
isomorphism induced by the $\what$-splitting of the weight filtration.

It is not difficult to show (cf.\cite[Prop.~4.3]{hain:nab}) that if the
finiteness assumption (\ref{eqn:finiteness}) holds for all $m\le N$, then the
functor that takes an $F$-algebra $A$ to the set
$$
\{\text{graded sections of }
\Gr^W_\dot\ghat_N\otimes_F A \to \Gr^W_\dot\g\otimes_F A\}
$$
is represented by an affine $F$-scheme of finite type, which we shall denote by
$\Sect_\dot(\Gr^W_\dot\g,\Gr^W_\dot\p/W_{-N-1})$. It is a subscheme of an
affine space that is a principal homogeneous space over
$\Hom_R(\Gr^W_\dot H_1(\g),\Gr^W_\dot\p_N)$. The finiteness assumption implies
that this is finite dimensional.

\begin{proposition}[{\cite[Cor.~4.7]{hain:nab}}]
\label{prop:gr_sect}
If $N\ge 0$, then, for all $F$-algebras $A$ the maps (\ref{eqn:maps_sects})
induce bijections
\begin{multline*}
%\label{eqn:maps_sects}
\{\what\text{-graded sections of } \ghat_N\otimes_F A \to \g\otimes_F A\} \cr
\overset{\simeq}{\to}
\{\text{sections of } \cGhat_N\otimes_F A \to \cG\otimes_F A\}/\cP(A) \cr
\overset{\Gr^W_\dot\! d}{\to}
\{\text{graded sections of }
\Gr^W_\dot\ghat_N\otimes_F A \to \Gr^W_\dot\g\otimes_F A\}
\end{multline*}
\end{proposition}

This follows from the observation \cite[Lem.~4.2]{hain:nab} that each
$\cP(A)$-conjugacy class of sections contains a unique $\what_N$-graded section.
It follows that the functor that takes $A$ to the set of $\cP(A)$ conjugacy
classes of sections of $\cGhat_N\otimes_F A \to \cG\otimes_F A$ is represented
by the affine $F$-scheme $\Sect_\dot(\Gr^W_\dot\g,\Gr^W_\dot\p/W_{-N-1})$. That
is, $\Hnab^1(\cG,\cP/W_{-N-1}\cP)=
\Sect_\dot(\Gr^W_\dot\g,\Gr^W_\dot\p/W_{-N-1})$.

\subsection{The connecting morphism $\delta$}
\label{sec:delta}

Full details of the construction of $\delta$ and its properties can be found in
\cite{hain:nab}.

Suppose that $m\ge 0$. Identify $\Hnab^1(\cG,\cP_m)$ with the scheme of graded
sections of $\Gr^W_\dot\ghat_m \to \Gr^W_\dot\g_\dot$. Suppose that $N\ge 1$.
Set $\z=\Gr^W_{-N}\p$. Then we have the extension of graded Lie algebras
$$
0 \to \z \to \Gr^W_\dot\g_N \to \Gr^W_\dot\ghat_{N-1} \to 0.
$$
Suppose that $\sigma$ is a graded section of $\Gr^W_\dot\ghat_{N-1} \to
\Gr^W_\dot\g$. Define $\sigmatilde$ to be the $R$-linear section
$$
\xymatrix{
\Gr^W_\dot\g \otimes_F A \ar[r]^(.45)\sigma &
\Gr^W_\dot\g_{N-1}\otimes_F A \ar[r]^{r\otimes A} &
\Gr^W_\dot\ghat_N \otimes_F A
}
$$
of $\Gr^W_\dot\ghat_N\otimes_F A \to \Gr^W_\dot\g\otimes_F A$, where $r$ is the
unique, $R$-linear, graded section of $\Gr^W_\dot\ghat_N \to \Gr^W\ghat_{N-1}$.
Note that, in general, $r$ is not a Lie algebra homomorphism. The obstruction to
$\sigmatilde$ being a Lie algebra homomorphism is the function $h_\sigma :
\Lambda^2 \Gr^W_\dot\g_N \to \z$ defined by
$$
h_\sigma(x\wedge y) := [\sigmatilde(x),\sigmatilde(y)]- \sigmatilde([x,y]).
$$
It is an $R$-invariant 2-cocycle of weight $0$ in the Chevalley-Eilenberg
cochain complex $C^\dot(\Gr^W_\dot\g,\z)\otimes_F A$. Define
$$
\delta : \Hnab^1(\cG,\cP_N)(A) \to H^2(\g,\z)\otimes_F A
$$
by taking $\sigma$ to the class of $h_\sigma$ in $H^2(\Gr^W_\dot\g,\z)\otimes_F
A \cong H^2(\g,\z)\otimes_F A$. This is induced by a morphism of $F$-schemes.

If $\delta(\sigma) \neq 0$, $\sigma$ does not lift to a graded Lie algebra
section of $\Gr^W_\dot\ghat_N\otimes_F A \to \g\otimes_F A$. If $\delta(\sigma)
= 0$, then, by \cite[Lem.~3.3]{hain:nab}, there is an $R$-invariant 1-cochain
$$
v \in \Gr^W_0 C^1(\Gr^W_\dot\g,\z)\otimes_F A
$$
such that $\delta(v)=h_\delta$. Then $\sigmatilde + v : \Gr^W_\dot\g\otimes_F A
\to \Gr^W_\dot\ghat_N\otimes_F A$ is a graded Lie algebra section of
$\Gr^W_\dot\ghat_N\otimes_F A \to \Gr^W_\dot\g\otimes_F A$ that lifts $\sigma$.

\section{Setup for Proofs of Theorems~\ref{thm:zero_points} and
\ref{thm:sec_pos}}
\label{sec:setup}

The following three sections are devoted to the proofs of
Theorems~\ref{thm:zero_points} and \ref{thm:sec_pos}. This section sets up the
notation to be used in these sections and  establishes some basic results.

Suppose that $g$, $n$ and $m$ are non-negative integers satisfying $g\ge 2$,
$n\ge 0$  and $m\ge 1$.\footnote{Later, we will assume that $g\ge 5$. But, for
the time being, we consider smaller genera so that some of the preliminary
results can be presented in their natural generality.} Suppose that $k$ is a
field of characteristic zero that contains $\bmu_m(\kbar)$, the $m$th roots of
unity. Fix a primitive $m$th root of unity. We then have the geometrically
connected moduli space $\M_{g,n/k}[m]$.

Set $K=k(\M_{g,n}[m])$. Fix an algebraic closure $\Kbar$ of $K$. Denote the
algebraic closure of $k$ in $\Kbar$ by $\kbar$. Set
$$
G_k = \Gal(\kbar/k),\quad G_K = \Gal(\Kbar/k),\text{ and }
G_K^\geom := \Gal(\Kbar/\kbar).
$$
There is an exact sequence
\begin{equation}
\label{eqn:extn}
1 \to G_K^\geom \to G_K \to G_k \to 1.
\end{equation}

Let $C$ be the restriction of the universal (complete) curve $\cC \to
\M_{g,n/k}[m]$ to its generic point $\Spec K$. Choose a geometric point $\xbar$
of $C\otimes_K\Kbar$ that lies over its generic point. Set
$$
\pi = \pi_1(C\otimes_K\Kbar,\xbar)
$$
Then one has the extension
$$
1 \to \pi \to \pi_1(C,\xbar) \to G_K \to 1.
$$

Suppose that $\ell$ is a prime number. Set
$$
H = H_\Ql = \Het^1(C\otimes_K\Kbar,\Ql(1)).
$$
The monodromy representation $\rho : G_K \to \GSp(H)$ has Zariski dense image if
and only if the image of the $\ell$-adic cyclotomic character $\chi_\ell : G_k
\to \Zlx$ is infinite. Suppose that this is the case. Define the weight of a
representation $G_K \to \GSp(H) \to \Aut V$ to be its weight as a
$\GSp(H)$-module With this convention, $H$ is a $G_K$-module of weight $-1$. 

Denote the weighted completions of $G_K$ with respect to $\rho$ by $\cG_K$ and
the weighted completion of $\pi_1(C,\etabar)$ with respect to $\pi_1(C,\etabar)
\to G_K \to \GSp(H)$ by $\cG_C$. Denote their Lie algebras by $\g_K$ and $\g_C$,
respectively. Denote their prounipotent radicals by $\U_K$ and $\U_C$, and their
Lie algebras by $\u_K$ and $\u_C$, respectively. One has the exact sequence of
$\Ql$-groups
$$
1 \to \pi^\un \to \cG_C \to \cG_K \to 1
$$
where $\pi^\un$ denotes the continuous unipotent completion $\pi^\un_{/\Ql}$ of
$\pi$ over $\Ql$. Denote the Lie algebra of $\pi^\un$ by $\p$. Since $\g_K$,
$\g_C$ and $\p$ are $\cG_C$-modules via the conjugation action, all have
compatible natural weight filtrations that possess the exactness properties
described in Proposition~\ref{prop:weight}.

Basic naturality properties of weighted completion imply:

\begin{proposition}
For each $\Ql$-algebra $A$, there is a group extension
$$
1 \to
\pi^\un(A) \to \Gtilde_C^A \to G_K \to 1
$$
and an inclusion
$$
\xymatrix{
1 \ar[r] & \pi \ar[r]\ar[d] & \pi_1(C,\xbar) \ar[r]\ar[d] &
G_K \ar[r]\ar@{=}[d] & 1\cr
1 \ar[r] & \pi^\un(A) \ar[r] & \Gtilde_C^A \ar[r] & G_K \ar[r] & 1
}
$$
of extensions, where the left hand vertical map is the canonical inclusion.
\end{proposition}

\begin{proof}
Define $\Gtilde_C^A$ so that
\begin{equation}
\label{eqn:pullback}
\xymatrix{
\Gtilde_C^A \ar[r]\ar[d] & G_K \ar[d] \cr
\cG_C(A) \ar[r] & \cG_K(A)
}
\end{equation}
is a pullback square. The map $\Gtilde_C^A \to G_K$ is surjective because the
lower map is surjective. To see this, choose a splitting of $\cG_C \to \GSp(H)$
in the category of $\Ql$-groups. It descends to a splitting of $\cG_K \to
\GSp(H)$. These splittings induce compatible isomorphisms $\cG_C(A) \cong
\GSp(H_A)\ltimes \U_C(A)$ and $\cG_K(A) \cong \GSp(H_A)\ltimes \U_K(A)$. Since
$\U_C\to\U_K$ is a surjective homomorphism of prounipotent groups, the map
$\U_C(A) \to \U_K(A)$ is surjective, which implies that $\cG_C(A)\to\cG_K(A)$
is surjective.
\end{proof}

When $n\ge 1$, $\Gtilde_C^A$ is isomorphic to the group $G_K\ltimes \pi^\un(A)$
described in the introduction.

Since $W_{-r}\pi^\un$ is the $r$th term of its lower central series, each
$W_{-r}\pi^\un$ is a normal subgroup of $\cG_C$. Consequently,
$W_{-r}\pi^\un(A)$ is a normal subgroup of $\Gtilde_C^A$. One thus has the
truncated extensions
$$
1 \to (\pi^\un/W_{-r}\pi^\un)(A) \to \Gtilde_C^A/(W_{-r}\pi^\un(A))
\to G_K \to 1.
$$
for each extension field $A$ of $\Ql$. Define $\Hnab^1(G_K,\pi^\un/W_{-r})(A)$
to be the set of sections of $\Gtilde_C^A \to G_K$, modulo conjugation by
$\pi^\un(A)/W_{-r}$.

\begin{proposition}
\label{prop:nab_iso}
Suppose that $r \in [1,\infty]$. If $H^1(G_K,\Gr^W_{-s}\pi^\un)$ is finite
dimensional when $l\le s < r$, then the restriction mapping
$$
\Hnab^1(\cG_K,\pi^\un/W_{-r})(\Ql) \to \Hnab^1(G_K,\pi^\un/W_{-r})(\Ql)
$$
is a bijection.
\end{proposition}

Conditions that ensure the finite dimensionality of $H^1(G_K,\Gr^W_{-r}\pi^\un)$
are given in Section~\ref{sec:lie_coho_comps}. The condition is satisfied, for
example, when $k$ is a number field and $g\ge 3$.

\begin{proof}
Proposition~\ref{prop:homology} implies that $H^1(G_K,\Gr^W_{-s}\pi^\un)\cong
H^1(\cG_K,\Gr^W_{-s}\pi^\un)$.  The finite dimensionality of
$H^1(G_K,\Gr^W_{-s}\pi^\un)$ thus ensures the existence of the scheme
$\Hnab^1(\cG_K,\pi^\un/W_{-r})$. Suppose that $s$ is a section of
$\Gtilde^\Ql_C/W_{-r}\pi^\un(\Ql) \to G_K$. Then the universal mapping property
of weighted completion implies that
$$
\xymatrix{
G_K \ar[r]^(0.3)s & \Gtilde^{\Ql}_C/W_{-r}\pi^\un(\Ql) \ar[r] &
(\cG_C/W_{-r}\pi^\un)(\Ql)
}
$$
induces a section $\sigma$ of $\cG_C \to \cG_K/W_{-r}\pi^\un(\Ql)$ such that the
diagram
$$
\xymatrix{
\Gtilde^\Ql_C/W_{-r}\pi^\un(\Ql) \ar[d] & G_K \ar[l]_(0.32)s \ar[d] \cr
\cG_C(\Ql)/W_{-r}\pi^\un(\Ql) & \cG_K(\Ql) \ar[l]_(0.32)\sigma
}
$$
commutes. On the other hand, since (\ref{eqn:pullback}) is a pullback square, so
is its reduction mod $W_{-r}\pi^\un(\Ql)$. Consequently, each section $\sigma$
of $(\cG_C/W_{-r}\pi^\un)(\Ql) \to \cG_K$ induces a section $s$ of
$\Gtilde^\Ql_C/W_{-r}\pi^\un(\Ql) \to G_K$ such that the diagram above commutes.
The results follows as two sections $s_1$ and $s_2$ of $\Gtilde^\Ql_C \to G_K$
are conjugate by an element of $\pi^\un(\Ql)$ if and only if the corresponding
sections $\sigma_1$ and $\sigma_2$ of $\cG_C \to \cG_K$ are conjugate by an
element (necessarily the same mod $W_{-r}\pi^\un(\Ql)$) of $\pi^\un(\Ql)$.
\end{proof}

\section{Cohomology Computations}
\label{sec:coho_comps}

Theorems~\ref{thm:zero_points} and \ref{thm:sec_pos} are proved using the
non-abelian cohomology described in Section~\ref{sec:nab}. In order to compute
$\Hnab^1(\cG_K,\pi^\un)$, we need to compute $H^1(\g_K,\Gr^W_{-r}\p)$ and to
bound $H^2(\g_K,\Gr^W_{-r}\p)$ for each $r>0$. This we do in the next section.
But first we need to compute and/or bound the cohomology groups $H^\dot(G_K,V)$
in low degrees for all negatively weighted $\GSp(H)$-modules $V$ as these
control the cohomology groups $H^\dot(\g_K,\Gr^W_{-r}\p)$ in low degree. In this
section we do this by investigating the degree to which the restriction mapping
$$
\Het^\dot(\M_{g,n/k}[m],\V) \to H^\dot(G_K,V)
$$
is an isomorphism. Remarkably, it is an isomorphisms in degree 1 and an
injection in degree 2, except when the pullback of $\V$ to $\M_{g,n/\kbar}[m]$
contains a trivial local system.

Recall that $\cC_{g/k}^n[m]$ denotes the $n$th power of the universal (complete)
curve $\cC_g$ over $\M_{g/k}[m]$. By convention, $\cC_{g/k}^0[m]=\M_{g/k}[m]$.
Note that $\M_{g,n/k}[m]$ is a Zariski open subset of $\cC_{g/k}^n[m]$.

The irreducible $\GSp(H_\Ql)$-module $H_\lambda(r)$ gives rise to a $G_K$-module
of weight $-|\lambda|-2r$ by composition with $\rho$. Denote the local system
over $\M_{g,n/k}[m]$ corresponding to the $\GSp(H)$-module $V$ by $\V$. Call an
irreducible $\GSp(H_\Ql)$ module $V$ {\em geometrically non-trivial} if it is
non-trivial when restricted to $\Sp(H_\Ql)$.

The following key result follows from \cite[Cor.~6.2]{hain:density} and standard
comparison theorems.

\begin{theorem}
\label{thm:gysin}
Suppose that $k$ is an algebraically closed field of characteristic zero.
Suppose that $n\ge 0$, $m\ge 1$, and that $g=3$ or $g\ge 5$. If $U$ is a Zariski
open subset of $\cC_{g/k}^n[m]$, then for all non-trivial, irreducible,
geometrically non-trivial representations $V$ of $\GSp(H)$, the map
$$
\Het^j(\cC_g^n[m],\V) \to \Het^j(U,\V)
$$
induced by the inclusion $U\hookrightarrow \cC_g^n[m]$ is an isomorphism when
$j=0,1$ and an injection when $j=2$.
\end{theorem}

Since $\M_{g,n/k}[m]$ is a Zariski open subset of $\cC_g^n[m]$, the result
remains true when $\cC_g^n[m]$ is replaced by $\M_{g,n}[m]$ in the statement
of the Theorem.

\begin{remark}
\label{rem:variant}
Theorem~\ref{thm:gysin} holds in a more general situation, which is described in
the following paragraph. It allows the reduction of the transcendence degree of
the field $K$ in Theorem~\ref{thm:sec_pos} from $3g-3+n$ to $3$. The key point
is that a more general version of the previous theorem implies that the
cohomological computations in this section hold in this more general situation.
For clarity of exposition, we do not work out the details in this more general
situation, but leave them to the interested reader.

Suppose that $k$ is an algebraically closed field of characteristic zero. First
let $M$ be a generic linear section of $\M_{g/k}[m]$ of dimension $\ge 3$.
Denote the restriction of the $n$th power of the universal curve to $M$ by
$C_M^n$. For a suitable imbedding $C_M^n \to \P^N$,
\cite[Thm.~5.1]{hain:density} implies (cf.\ \cite[Thm.~6.1]{hain:density}) that
a generic section $T$ of $C_M^n$ by a complete intersection of sufficiently high
multi-degree will have the property that for all non-trivial, irreducible,
geometrically non-trivial representations $V$ of $\GSp(H)$, and all Zariski open
subsets $U$ of $T$, the map $\Het^j(T,\V) \to \Het^j(U,\V)$ induced by the
inclusion $U\hookrightarrow T$ is an isomorphism when $j=0,1$ and an injection
when $j=2$. \qed
\end{remark}

Set $K=k(\M_{g,n/k}[m])$ and let $C/K$ be the restriction of the universal curve
to $\Spec K$.

\begin{proposition}
\label{prop:coho_comp}
Suppose that $g\ge 5$ and that $k$ is a field of characteristic zero for which
the image of the $\ell$-adic cyclotomic character $\chi_\ell : G_k \to \Zlx$ is
infinite.
\begin{enumerate}

\item\label{item1}
If $r\neq 1$, then $H^1(G_K,\Ql(r)) \cong H^1(G_k,\Ql(r))$.

\item\label{item2}
If $V$ is a geometrically non-trivial, irreducible representation of
$\GSp(H_\Ql)$ of negative weight, then
$$
H^1(G_K,V) \cong \Het^1(\M_{g,n/k}[m],\V) \cong
\begin{cases}
\Ql^n & V = H_\Ql;\cr
\Ql & V = \Lambda^3_0 H_\Ql;\cr
0 & \text{otherwise}.
\end{cases}
$$

\item\label{item3}
If $V$ is a $\GSp(H)$-module satisfying $V^{\Sp(H)}=0$, then there is a natural
inclusion
$$
H^1(G_k,H^1(G_K^\geom,V))\hookrightarrow H^2(G_K,V).
$$

\item\label{item4}
If $V$ is a $\GSp(H)$-module satisfying $V^{\Sp(H)}=0$, then the restriction
mapping $\Het^2(\M_{g,n/k}[m],\V) \to H^2(G_K,V)$ is injective.

\end{enumerate}
\end{proposition}

\begin{proof}
The extension (\ref{eqn:extn}) gives rise to a spectral sequence
\begin{equation}
\label{eqn:ss}
E_2^{s,t} = H^s(G_k,H^t(G_K^\geom,V)) \Rightarrow H^{s+t}(G_K,V).
\end{equation}
There is therefore an exact sequence
\begin{multline}
\label{eqn:left_exact}
0 \to H^1(G_k,H^0(G^\geom_K,V)) \to H^1(G_K,V) \cr
\to H^0(G_k,H^1(G_K^\geom,V)) \to H^2(G_k,H^0(G^\geom_K,V)).
\end{multline}
Observe that
\begin{equation*}
%\label{eqn:limit}
H^\dot(G_K^\geom,V) \cong \varinjlim_U \Het^\dot(U,\V),
\end{equation*}
where $U$ ranges over the Zariski open subsets of $T:=\M_{g,n/\kbar}[m]$ that
are defined over $k$.

The Gysin sequence implies that if $U= T-Z$, then there is
an exact sequence
$$
0 \to \Het^1(T,\Ql) \to \Het^1(U,\Ql) \to \Ind_{G_{k_S}}^{G_k}\Ql(-1)
$$
of $G_k$-modules, where $S$ is the set of components of $Z$ that have
codimension $1$ in $T$ and $G_{k_S}$ is the kernel of the action of $G_k$ on
$S$. Since $\Het^1(T,\Q)=0$ for all $g\ge 3$, \cite[Prop.~5.2]{hain:msri}, this
implies that $\Het^1(U,\Ql)$ is a $G_k$-submodule of
$\Ind_{G_{k_S}}^{G_k}\Ql(-1)$. Therefore $H^0(G_k,\Het^1(U,\Ql(r)))$ vanishes
for all $r\neq 1$, and $H^0(G_k,H^1(G_K^\geom,\Ql(r)))$ thus vanishes as well.
Assertion (\ref{item1}) now follows from the exactness of
(\ref{eqn:left_exact}).

It suffices to prove (\ref{item3}) and (\ref{item4}) when $V$ is a geometrically
non-trivial, irreducible $\GSp(H)$-module. For the remainder of the proof, we
will assume this to be the case. This assumption implies that
$H^0(G_K^\geom,V)=0$, so that the bottom row of the spectral sequence
(\ref{eqn:ss}) vanishes. Theorem~\ref{thm:gysin} and
\cite[Proposition~5.2]{hain:msri} imply that $\Spec K \to U \to T$ induce
isomorphisms
\begin{gather*}
H^1(G_K^\geom,H_\Ql(r)) \cong \Het^1(U,\H_\Ql(r))
\cong \Het^1(T,\H_\Ql(r)) \cong \Ql(r)^n\cr
H^1(G_K^\geom,\Lambda^3_0 H_\Ql(r)) \cong \Het^1(U,\Lambda^3_0\H_\Ql(r))
\cong\Het^1(T,\Lambda^3_0\H_\Ql(r))
\cong \Ql(r)
\end{gather*}
and $H^1(G_K,V) = \Het^1(U,\V)=\Het^1(T,\V)=0$ for all other geometrically
non-trivial, irreducible $\GSp(H_\Ql)$-modules $V$. Assertion (\ref{item2})
now follows directly from (\ref{eqn:left_exact}).

As observed above, the fact that $V$ is a geometrically non-trivial, irreducible
representation implies that the bottom row of the spectral sequence
(\ref{eqn:ss}) vanishes. Consequently, there is an exact sequence
\begin{multline}
\label{eqn:deg_2}
0 \to H^1(G_k,H^1(G_K^\geom,V)) \to H^2(G_K,V) \cr
\to H^0(G_k,H^2(G_K^\geom,V))
\to H^2(G_k,H^1(G_K^\geom,V)).
\end{multline}
Assertion (\ref{item3}) follows.

To prove (\ref{item4}), consider the spectral sequence of \'etale cohomology
that converges to $\Het^\dot(T,\V)$. Since $V$ is a
geometrically non-trivial, irreducible $\GSp(H)$-module,
$H^0(T\otimes_k\kbar,\V)=0$. The map $\Spec K \to T$ induces a map of spectral
sequence, and thus also a map of exact sequences
$$
\xymatrix{
0 \ar[r] & H^1(G_k,\Het^1(T\otimes_k \kbar,\V)) \ar[d]^\cong\ar[r] &
\Het^2(T,\V) \ar[r]\ar[d] & H^0(G_k,\Het^2(T\otimes_k\kbar,\V)) \ar@{^{(}->}[d]
\cr
0 \ar[r] & H^1(G_k,H^1(G_K^\geom,V)) \ar[r] & H^2(G_K,V) \ar[r] &
H^0(G_k,H^2(G_K^\geom,V))
%% ) %% to balance the parenthesis above
}
$$
whose second row is part of (\ref{eqn:deg_2}). Theorem~\ref{thm:gysin} implies
that the left-hand vertical map is an isomorphism and that the right-hand
vertical map is an inclusion. Assertion (\ref{item4}) follows.
\end{proof}

Each $\Gr^W_{-r}\pi^\un$ is a $\GSp(H)$-module. The $G_K$ action on it
factors through the monodromy representation $G_K \to \GSp(H_\Ql)$. Since the
image of $G_K^\geom$ is contained in $\Sp(H)$, the geometric invariants
$(\Gr^W_{-r}\pi^\un)^{\Sp(H)}$ is a $G_k$-module which vanishes when $r$
is odd and is isomorphic to a sum of copies of $\Ql(t)$ when $r=2t$ is even.

\begin{corollary}
\label{cor:essentials}
Suppose that $g\ge 5$. If $r\ge 2$, then
$$
H^1(G_K,\Gr^W_{-r}\pi^\un)
\cong H^1(G_k,(\Gr^W_{-r}\pi^\un)^{\Sp(H)})
$$
In particular, when $r>1$ is odd, then $H^1(G_K,\Gr^W_{-r}\pi^\un)$
vanishes. If $r\ge 2$, then the cup product
$$
H^1(G_K,H) \otimes H^1(G_K,\Gr^W_{-2r}\pi^\un)
\to H^2(G_K,\Gr^W_{-2r-1}\pi^\un),
$$
induced by the commutator $H\otimes \Gr^W_{-2r}\pi^\un \to
\Gr^W_{-2r-1}\pi^\un$, is injective.
\end{corollary}

\begin{proof}
Suppose that $r\ge 2$. Write
$$
\Gr^W_{-r}\pi^\un = (\Gr^W_{-r}\pi^\un)^{\Sp(H)} \oplus V
$$
where $V$ is a $\GSp(H)$-module. Since $V$ has weight $<-1$ and since each of
its irreducible components is geometrically non-trivial,
Proposition~\ref{prop:coho_comp} implies that $H^1(G_K,V)=0$. It follows that
$$
H^1(G_K,\Gr^W_{-r}\pi^\un) \cong
H^1(G_K,(\Gr^W_{-r}\pi^\un)^{\Sp(H)}).
$$
Since $(\Gr^W_{-r}\pi^\un)^{\Sp(H)}$ vanishes when $r$ is odd and when $r=2$,
and is a direct sum of copies of $\Ql(t)$ when $r=2t$,
Proposition~\ref{prop:coho_comp} implies that, when $r\ge 2$,
$$
H^1(G_K,(\Gr^W_{-r}\pi^\un)^{\Sp(H)}) \cong
H^1(G_k,(\Gr^W_{-r}\pi^\un)^{\Sp(H)}).
$$
This proves the first assertion.

Schur's Lemma implies that all $\Sp(H)$-submodules of $H\otimes
(\Gr^W_{-r}\pi^\un)^{\Sp(H)}$ are of the form $H\otimes B$, where $B$ is a
subspace of $(\Gr^W_{-r}\pi^\un)^{\Sp(H)}$. Since the Lie algebra
$\Gr^W_\dot\pi^\un$ is generated by $H=\Gr^W_{-1}\pi^\un$ and has trivial center
\cite[p.~201]{hattori-stallings}, the commutator map $H\otimes
(\Gr^W_{-r}\pi^\un)^{\Sp(H)} \to \Gr^W_{-r-1}\pi^\un$ is injective. Since
$\Gr^W_{-r-1}\pi^\un$ is a semi-simple $G_K$-module,
$$
H^2\big(G_K,H\otimes (\Gr^W_{-r-1}\pi^\un)^{\Sp(H)}\big)
\to H^2(G_K,\Gr^W_{-r-1}\pi^\un)
$$
is injective. To prove the last assertion, it thus suffices to show that the cup
product
\begin{equation}
\label{eqn:cup}
H^1(G_K,H) \otimes H^1\big(G_K,(\Gr^W_{-r}\pi^\un)^{\Sp(H)}\big)
\to H^2(G_K,H\otimes (\Gr^W_{-r}\pi^\un)^{\Sp(H)})
\end{equation}
is injective.

Proposition~\ref{prop:coho_comp}(\ref{item2}) implies that $H^1(G_K,H) =
H^1(G_K^\geom,H)$. This is a trivial $G_k$-module. The cup product
(\ref{eqn:cup}) is thus the composite of the mappings:
\begin{align*}
& H^1(G_K,H) \otimes H^1\big(G_k,(\Gr^W_{-r}\pi^\un)^{\Sp(H)}\big) \cr
&\overset{\simeq}{\to}
H^1(G^\geom_K,H) \otimes H^1\big(G_k,(\Gr^W_{-r}\pi^\un)^{\Sp(H)}\big) \cr
&\cong
H^1\big(G_k,H^1(G^\geom_K,H)\otimes(\Gr^W_{-r}\pi^\un)^{\Sp(H)}\big) \cr
&\cong
H^1\big(G_k,H^1(G^\geom_K,H\otimes (\Gr^W_{-r}\pi^\un)^{\Sp(H)})\big) \cr
&\subseteq
H^2(G_K,H\otimes (\Gr^W_{-r}\pi^\un)^{\Sp(H)}).
\end{align*}
which is injective as Proposition~\ref{prop:coho_comp}(\ref{item4}) implies
that the last mapping is injective.
\end{proof}

\section{Computation of $H^\dot(\g_K,\Gr^W_\dot\p)$}
\label{sec:lie_coho_comps}

The computation of the cohomology groups $H^\dot(\g_K,\Gr^W_{-r}\p)$ in low
degrees is needed for the application of the non-abelian cohomology of
Section~\ref{sec:nab} to the proofs of Theorems~\ref{thm:zero_points} and
\ref{thm:sec_pos}. Throughout this section, we assume that $\chi_\ell : G_k \to
\Zlx$ has infinite image.

Since $\GSp(H)$ is connected, Proposition~\ref{prop:homology} and
Remark~\ref{rem:invariants} imply that
$$
H^j(\g_{g,n},V) \cong \Hom_{\GSp(H)}(H_j(\u_{g,n}),V))
$$
for all finite dimensional $\GSp(H)$-modules $V$. In particular, the projections
$$
\cc,\ \dd_j,\ \ee_j,\ \ee_{ij} \in \Hom_{\GSp(H)}(H_2(\u_{g,n}),\Lambda^2_0 H)
$$
defined in display (\ref{eqn:projns}), Section~\ref{sec:relations}, can be
regarded as elements of $H^2(\g_{g,n},\Lambda^2_0 H)$. We will also regard them
as elements of $H^2(\g_K,\Lambda^2_0 H)$ via the natural homomorphism
$$
H^2(\g_{g,n},\Lambda^2_0 H) \to H^2(\g_K,\Lambda^2_0 H).
$$

\begin{proposition}
\label{prop:lie_coho}
If $g\ge 5$ and $n\ge 0$, then
$$
H^1(\g_K,\Gr^W_{-r}\p) \cong
\begin{cases}
\Ql^n & r = 1,\cr
0 & 2 < r < 6 \text{ and } r\text{ odd } > 6, \cr
H^1(G_k,(\Gr^W_{-r}\pi^\un)^{\Sp(H)}) & r>2 \text{ even.}
\end{cases}
$$
When $r=2$, $H^2(\g_K,\Gr^W_{-2}\pi^\un)$ is isomorphic to
$H^2(\g_{g,n},\Lambda^2_0 H)$ and is generated by the classes $\cc$, $\dd_j,\
\ee_j$, ($1\le j \le n$), $\ee_{ij}$, ($1\le i\le j \le n$), which are subject
to
the relations
$$
\cc + \dd_j + \ee_j = 0,\quad j=1,\dots,n.
$$
Consequently $\{\cc,\ee_i,\ee_{ij}:1\le i \le j \le n\}$ is a basis of
$H^2(\g_K,\Lambda^2_0 H)$.
\end{proposition}

\begin{proof}
The exponential map induces a $\cG_K$-module isomorphism $\Gr^W_{-r}\p \to
\Gr^W_{-r}\pi^\un$. We will denote this coefficient module by
$\Gr^W_{-r}\pi^\un$ in group cohomology, and by $\Gr^W_{-r}\p$ in Lie algebra
cohomology. Proposition~\ref{prop:coho_comp} implies that the natural map
$$
H^j(\g_K,\Gr^W_{-r}\p) \to H^j(G_K,\Gr^W_{-r}\pi^\un)
$$
is an isomorphism when $j=1$ and injective when $j=2$. The computation of
$H^1(\g_K,\Gr^W_{-r}\p)$ follows directly from Proposition~\ref{prop:coho_comp},
Corollary~\ref{cor:trivial_reps}, and Corollary~\ref{cor:essentials}.

Proposition~\ref{prop:homology} implies that $\Spec K \to \M_{g,n/k}$ induces an
inclusion
$$
\Het^2(\M_{g,n/k},\Lambda^2_0 H) \hookrightarrow H^2(G_K,\Lambda^2_0 H)
$$
and therefore an injection
\begin{equation}
\label{eqn:h2}
H^2(\g_{g,n},\Lambda^2_0 H) \to H^2(\g_K,\Lambda^2_0 H).
\end{equation}
Apply $\Hom_{\GSp(H)}(\blank,\Lambda^2_0 H)$ to the exact sequence in
Lemma~\ref{lem:ses} with $\u = \u_K$ and $\u_{g,n}$ to obtain the commutative
diagram
{\tiny
$$
\xymatrix@C=16pt{
0 \ar[r] &
\Hom_{\GSp(H)}(\Gr^W_{-2}\u_{g,n},\Lambda^2_0 H) \ar[r]^{\bracket^\ast}\ar[d] &
\Hom_{\GSp(H)}(\Lambda^2 \Gr^W_{-1}\u_{g,n},\Lambda^2_0 H) \ar[r]\ar[d]^\cong &
H^2(\g_{g,n},\Lambda^2_0 H)\ar@{^{(}->}[d] \ar[r] & 0
\cr
0 \ar[r] &
\Hom_{\GSp(H)}(\Gr^W_{-2}\u_K,\Lambda^2_0 H) \ar[r]^{\bracket^\ast} &
\Hom_{\GSp(H)}(\Lambda^2 \Gr^W_{-1}\u_K,\Lambda^2_0 H) \ar[r] &
H^2(\g_K,\Lambda^2_0 H) \ar[r] & 0
%% ) %% to balance parenthesis above
}
$$
}
The rows are exact because $H^1(\g_K,\Lambda^2_0 H)$ and
$H^1(\g_{g,n},\Lambda^2_0 H)$ both vanish when $g\ge 5$. This implies that
(\ref{eqn:h2}) is an isomorphism, which implies that the left-hand vertical map
is an isomorphism. The result now follows from
Propositions~\ref{prop:projections}, which implies that
$\{\cc,\dd_i,\ee_i,\ee_{ij}:1\le i \le j \le n\}$ is a basis of the middle
group, and Proposition~\ref{prop:bracket}, which implies that the relations are
$\{\cc + \dd_i + \ee_i = 0:1\le i \le n\}$.
\end{proof}

Since $H^1(G_K,\Gr^W_{-r}\pi^\un)$ is finite dimensional if and only if
$H^1(G_k,\Ql(r/2))$ is finite dimensional, the previous Proposition and
Corollary~\ref{cor:essentials} imply:

\begin{proposition}
\label{prop:fd}
If $1\le r < 6$ or if $r$ is odd, then for all fields $k$ of characteristic
zero, $H^1(\g_K,\Gr^W_{-r}\p)$ is finite dimensional. If $r=2s\ge 6$ is even and
if $H^1(G_k,\Ql(s))$ is finite dimensional, then $H^1(\g_K,\Gr^W_{-r}\p)$ is
finite dimensional.
\end{proposition}

This allows us to apply the non-abelian cohomology of Section~\ref{sec:nab} when
$k$ is a number field or a finite extension of $\Qp$ provided $p\neq \ell$. The
following result is needed for the computation of the connecting morphism in
non-abelian cohomology.

\begin{proposition}
\label{prop:injective}
If $g\ge 5$ and $n\ge 1$, then the map
$$
H^1(\g_K,\Gr^W_{-1}\p) \otimes H^1(\g_K,\Gr^W_{-r}\p)
\to H^2(\g_K,\Gr^W_{-r-1}\p),
$$
induced by the bracket $\Gr^W_{-1}\p\otimes \Gr^W_{-r}\p \to \Gr^W_{-r-1}\p$, is
injective.
\end{proposition}

\begin{proof}
This follows from Corollary~\ref{cor:essentials} and the fact that the diagram
$$
\xymatrix{
H^1(\g_K,\Gr^W_{-1}\p) \otimes H^1(\g_K,\Gr^W_{-r}\p)
\ar[d]_\cong\ar[r] & H^2(\g_K,\Gr^W_{-r-1}\p) \ar@{^{(}->}[d] \cr
H^1(G_K,\Gr^W_{-1}\pi^\un) \otimes H^1(G_K,\Gr^W_{-r}\pi^\un)
\ar[r] & H^2(G_K,\Gr^W_{-r-1}\pi^\un)
%% ) %% to balance the parenthesis above
}
$$
commutes, and the fact that the left-hand vertical map is an isomorphism and the
right-hand vertical map is an inclusion.
\end{proof}

\section{Proofs of Theorems~\ref{thm:zero_points} and \ref{thm:sec_pos}}
\label{sec:thm3}

We are now ready to combine the non-abelian cohomology of Section~\ref{sec:nab}
with the subsequent computations to prove Theorems~\ref{thm:zero_points} and
\ref{thm:sec_pos}. Throughout we assume that $g\ge 5$ and that $r\ge 1$ and that
the image of $\chi_\ell : G_k \to \Zlx$ is infinite. We will assume that
$H^1(\g_K,\Gr^W_{-s}\p)$ is finite dimensional whenever $1\le s \le r$ to ensure
that the scheme $\Hnab^1(G_K,\pi^\un/W_{-r-1})$ is defined. Conditions under
which this is the case are given in Proposition~\ref{prop:fd}. They are
satisfied, for example, when $k$ is a number field or a finite extensions of
$\Qp$ where $p\neq \ell$.

When $n\ge 1$, denote by $s_j$ the section of $\cG_C \to \cG_K$ induced by the
$j$th tautological point $x_j$ of $C$. There are thus maps
$$
\iota_r : \{s_1,\dots,s_n\}\to \Hnab^1(\cG_K,\pi^\un/W_{-r-1})(\Ql)
$$
for each $r>0$. When $n=0$, we consider $\{s_1,\dots,s_n\}$ to be the empty set.
The maps $\iota_r$ are compatible in the sense that
$$
\xymatrix{
\{s_1,\dots,s_n\} \ar[r]_(0.4){\iota_{r+1}} \ar@/^1pc/[rr]^{\iota_r} &
H^1(\cG_K,\pi^\un/W_{-r-2}) \ar[r]_{\text{quot}_\ast}
& H^1(\cG_K,\pi^\un/W_{-r-1})
}
$$
commutes.

When $r=1$, $H^1(\cG_K,\pi^\un/W_{-2})$ is a principal homogeneous space over
$$
H^1(\g_K,\Gr^W_{-1}\p) = H^1(\g_K,H)
= \Ql\kappa_1\oplus \dots \oplus \Ql\kappa_n.
$$
The first task it to fix a natural identification
$$
H^1(\g_K,\Gr^W_{-1}\p) \overset{\simeq}{\to} \Hnab^1(\cG_K,\pi^\un/W_{-2}).
$$
Proposition~\ref{prop:gr_sect} implies that elements of
$\Hnab^1(\cG_K,\pi^\un/W_{-2})$ correspond to $\GSp(H)$-invariant sections of
$\Gr^W_{-1}\u_C \to \Gr^W_{-1}\u_K$. Theorem~\ref{thm:presentation} implies that
these correspond to the $\GSp(H)$-invariant sections of the map $\epsilon_n :
\Lambda^3_{n+1} H \to \Lambda^3_n H$ that takes $(v;u_0,u_1,\dots,u_n)$ to
$(v;u_1,\dots,u_n)$. We identify $\Hnab^1(\cG_K,\pi^\un/W_{-2})$ with
$\Aff_\Ql^n$ by identifying the $\GSp(H)$-invariant section
$$
(v;u_1,\dots,u_n) \mapsto (v:\textstyle{\sum_{j=1}^n a_j u_j,u_1,\dots, u_n})
$$
of $\epsilon_n$ with $(a_1,\dots,a_n)$. Corollary~\ref{cor:kappa} now implies:

\begin{lemma}
With this identification,
$$
(2g-2)\iota_1(s_j) = \kappa_j \in \Hnab^1(\cG_K,\pi^\un/W_{-2})(\Ql),
$$
where $1\le j \le n$. Consequently, $\iota_r$ is injective for each $r\ge 1$.
\end{lemma}

The computation of $\delta:\Hnab^1(\cG_K,\pi^\un/W_{-2}) \to
H^2(\g_K,\Gr^W_{-2}\p)$ is a reformulation of Lemma~\ref{lem:cocycle}.

\begin{lemma}
\label{lem:delta}
If $g\ge 5$, then the connecting morphism
$$
\delta : \Hnab^1(\cG_K,\pi^\un/W_{-2}) \to H^2(\g_K,\Gr^W_{-2}\p)
$$
is given by
$$
\delta\big(\textstyle{\sum_{j=1}^n a_j \kappa_j}\big)/(2g-2) =
\textstyle{\sum_{i<j} a_i a_j \ee_{ij}}
 + \textstyle{\sum_i (a_i^2 - a_i)\ee_i}
+ (1-\textstyle{\sum_i a_i})\cc.
$$
\end{lemma}

\begin{proof}
This follows directly from the construction of $\delta$ given in
Section~\ref{sec:delta}, the cohomology computation
Proposition~\ref{prop:lie_coho}, and the formula for the bracket given in
Lemma~\ref{lem:cocycle}.
\end{proof}

\begin{corollary}
\label{cor:isom_r_is_3}
If $g\ge 5$, then
$$
\iota_2 : \{s_1,\dots,s_n\} \to \Hnab^1(\cG_K,\pi^\un/W_{-3})(\Ql)
$$
is a bijection.
\end{corollary}

\begin{proof}
Theorem~\ref{thm:exactness} implies that $H^1(\cG_K,\pi^\un/W_{-3})$ is the
inverse image of $0$ under
$$
\delta : H^1(\cG_K,H) \to H^2(\g_K,\Gr^W_{-2}\p).
$$
The previous result and the Chinese Remainder Theorem imply that
\begin{align*}
\delta^{-1}(0) &=
\big\{
\textstyle{(2g-2)^{-1}\sum_{j=0}^n a_j\kappa_j} :
a_i a_j = \delta_{i,j}a_j,\ \textstyle{\sum_i a_i} = 1
\big\}
\cr
&= \{\kappa_1/(2g-2),\dots,\kappa_n/(2g-2)\}\times \Spec\Ql.
\end{align*}
The result follows as $\iota_2(s_j)=\kappa_j/(2g-2)$.
\end{proof}

\begin{proof}[Proof of Theorem~\ref{thm:zero_points}]
By Proposition~\ref{prop:generic_pt}, a section of $\pi_1(C,\xbar) \to G_K$
induces a section of $\cG_C/W_{-3}\pi^\un \to \cG_K$. But the previous result
implies that there are no such sections when $n=0$.
\end{proof}

Another consequence is the following version of the Section Conjecture. The
theorem shows that Kim's program \cite{kim:program} works well in the case where
the curve is the generic curve of type $(g,n)$ with a level $m\ge 1$ structure
and $g\ge 5$.

\begin{theorem}
Suppose that $k$ is a field of characteristic zero and that $\ell$ is a prime
number for which the image of the $\ell$-adic cyclotomic character $\chi_\ell$
is infinite. Let $K=k(\M_{g,n}[m])$ and $C$ be the restriction of the universal
curve over $\M_{g,n}[m]$ to its generic point. If $g\ge 5$ and $n\ge 0$, then
$C(K) = \{x_1,\dots,x_n\}$ and the function
$$
\sect^\un_r : C(K) \to \Hnab^1(G_K,\pi^\un/W_{-r-1})(\Ql)
$$
that takes a rational point $x$ to the conjugacy class of the section
$s_{r,x}^\un$ of
$$
\Gtilde_C^\Ql/W_{-r-1}\pi^\un(\Ql) \to G_K
$$
induced by $s_x$ is a bijection when $2\le r \le 5$. If $H^1(G_k,\Ql(t))$
vanishes for all $t\ge 3$ --- such as when $k$ is a finite extension of $\Q_p$,
$p\neq \ell$ \cite[VII,\S3]{coho-fields} --- then $\sect^\un_r$ is a bijection
for all $r\in [3,\infty]$.
\end{theorem}

\begin{proof}
We know from Theorem~\ref{thm:rational} that $C(K)=\{x_1,\dots,x_n\}$. From
Proposition~\ref{prop:nab_iso} we know that $\Hnab^1(\cG_K,\pi^\un/W_{-r-1}) \to
\Hnab^1(G_K,\pi^\un/W_{-r-1})$ is an isomorphism. The result then follows from
the exact sequence of Theorem~\ref{thm:exactness} and the computations of
Proposition~\ref{prop:lie_coho}.
\end{proof}

Additional work is required when $H^1(G_k,\Ql(t))$ does not vanish for all $t\ge
3$. This is the case when $k$ is a number field, for example.

\begin{proposition}
Suppose that $r\ge 2$ and that $H^1(G_k,\Ql(t))$ is finite dimensional whenever
$3\le 2t \le r$. If $g \ge 5$, then there are natural isomorphisms
\begin{align*}
\Hnab^1(G_K,\pi^\un/W_{-r-1})
&\cong \Hnab^1(\cG_K,\pi^\un/W_{-r-1}) \cr
&\cong \{s_1,\dots,s_n\} \times H^1\big(\g_K,\Gr^W_{-r}\p\big) \cr
&\cong
{
\begin{cases}
\{s_1,\dots,s_n\}\times \Spec\Ql & r \text{ odd,} \cr
\{s_1,\dots,s_n\}\times H^1\big(G_k,\big(\Gr^W_{-r}\pi^\un\big)^{\Sp(H)}\big)
& r \text{ even.}
\end{cases}
}
\end{align*}
where $\iota_m(s_j)$ corresponds to $s_j$ when $r$ is odd and to $(s_j,0)$
when $r$ is even. 
When $r = 2t$, there is a commutative diagram.
{\small
$$
\xymatrix{
\{s_1,\dots,s_n\}\times H^1(\g_K,\Gr^W_{-2t}\p)
\ar[r]^{\iota_1\times\id}\ar@{<->}[d]_{\cong} &
H^1(\g_K,\Gr^W_{-1}\p) \otimes H^1(\g_K,\Gr^W_{-2t}\p)
\ar[d]^{\text{\sf cup}} \cr
\Hnab^1(\cG_K,\pi^\un/W_{-2t}) \ar[r]^\delta & H^2(\g_K,\Gr^W_{-2t-1}\p)
}
$$
}
where $\text{\sf cup}$ is induced by the bracket $\Gr^W_{-1}\p \otimes
\Gr^W_{-2t}\p \to \Gr^W_{-2t-1}\p$.
\end{proposition}

\begin{proof}
We prove this result by induction on $r$ using the exact sequence of
Theorem~\ref{thm:exactness}. The case $r=2$ is established in
Corollary~\ref{cor:isom_r_is_3}. Observe that the composition
$$
\xymatrix{
\{s_1,\dots,s_n\} \ar[r]^(0.37){\iota_r}
& H^1(\cG_K,\pi^\un/W_{-r-1})(\Ql) \ar[r]^(.55)\delta &
H^2(\g_K,\Gr^W_{-r-1}\p)
}
$$
of $\iota_r$ with the connecting morphism $\delta$ is trivial for all $r > 1$.
Since $H^1(\g,\Gr^W_{-r}\p)=0$ for $2\le r < 6$, the result also holds when
$r<6$.

Next observe that if the result is true for $r=2t-1>2$, it is also true for
$r=2t$ as the exact sequence of non-abelian cohomology implies that
$\Hnab^1(\cG,\pi/W_{-2t-1})$ is a principal $H^1(\g_K,\Gr^W_{-2t}\p)$ bundle
over $\Hnab^1(\cG_K,\pi^\un/W_{-2t-2}\pi^\un)$, which, by induction, is
$\{s_1,\dots,s_n\}$.

To prove that if the result is true for $r=2t$, then it is true for $r=2t+1$, it
suffices to prove the last assertion as then Proposition~\ref{prop:injective}
then implies that $\delta$ is an imbedding on each $s_j\times
H^1(\g_K,\Gr^W_{-2t}\p)$ that takes $(s_j,0)$ to $0$. But this follows from the
description of $\delta$ given in Section~\ref{sec:delta}.
\end{proof}

Combining Theorem~\ref{thm:rational} with the fact that $\Hnab^1(G_K,\pi^\un)$
is the projective limit of the $\Hnab^1(G_K,\pi^\un/W_{-r})$, we obtain:

\begin{theorem}
Suppose that $k$ is a field of characteristic zero and that $\ell$ is a prime
number for which the image of the $\ell$-adic cyclotomic character $\chi_\ell$
is infinite. Let $K=k(\M_{g,n}[m])$ and $C$ be the restriction of the
universal curve over $\M_{g,n/k}[m]$ to its generic point. If $g\ge 5$ and $n\ge
0$ and if $H^1(G_k,\Ql(s))$ is finite dimensional for all $s\ge 3$, then $C(K) =
\{x_1,\dots,x_n\}$ and the function
$$
\sect^\un_\infty : C(K) \to \Hnab^1(G_K,\pi^\un)(\Ql)
$$
that takes a rational point to the conjugacy class of the section of
$$
\Gtilde_C^\Ql/W_{-r-1}\pi^\un(\Ql) \to G_K
$$
induced by $s_x$ is a bijection.
\end{theorem}

This completes the proof of Theorem~\ref{thm:sec_pos}. With minor modifications,
one can prove the following generalization.

Let $T$ be a section of dimension $\ge 3$ of $\M_{g,n/k}[m]$ by a generic
complete intersection as described in Remark~\ref{rem:variant}. Suppose that
$T$ is defined over $k$ and geometrically connected. Denote the restriction of
the universal curve over $\M_{g,n/k}[m]$ to $T$ by $C \to T$. Set $K=k(T)$.

\begin{theorem}
\label{thm:variant}
Suppose that $\ell$ is a prime number and that the image of the $\ell$-adic
cyclotomic character is infinite. If $g\ge 5$, $n\ge 0$ and $m\ge 1$ and if
$H^1(G_k,\Ql(r))$ is finite dimensional for all $r>1$, then
$$
C(K) \to \Hnab^1(G_K,\pi^\un)(\Ql)
$$
is an isomorphism. 
\end{theorem}

The proof is left the interested reader. The main point is that the computations
of Section~\ref{sec:coho_comps} hold in this more general situation. The rest of
the argument is identical.

\appendix

\section{A Hodge Theoretic Lemma}

Suppose that $T$ is a complex algebraic manifold and that $\H$ is a polarized
variation of Hodge structure (PVHS) of negative weight. Suppose that the
$\Z$-local system underlying $\H$ is torsion free. Here we will construct the
sequence
\begin{equation}
\label{seqce}
0 \to \Ext^1_\cH(\Z(0),H^0(T,\H_\Z)) \overset{j}{\to}
\Ext^1_{\cH(T)}(\Z(0)_T,\H_\Z)
\overset{\delta}{\to} H^1(T,\H_\Z)
\end{equation}
used in the proof of Proposition~\ref{prop:fg} and prove that it is exact. The
proof below is short and direct. A stronger result can be deduced from general
results proved in \cite{vmhs} in which the image of $\delta$ is shown the be the
appropriate set of Hodge classes.

The map $\delta$ takes a VMHS $\V$ that is an extension
$$
0 \to \H \to \V \to \Z(0)_T \to 0
$$
to the corresponding extension of local systems, which is an element of
$$
\Ext^1_{\cL(T)}(\Z,\H) \cong H^1(T,\H)
$$
where $\cL(T)$ denotes the category of $\Z$-local systems over $T$.

The map $j$ is defined by pushout. By the theorem of the fixed part for PVHS,
$H^0(T,\H)$ \cite{schmid} has a pure Hodge structure and the inclusion
$H^0(T,\H)\hookrightarrow H_t$ is a morphism of HS for all $t\in T$. Thus
$H^0(T,\H)$ may be regarded as a constant sub PVHS of $\H$ over $T$. An
extension
\begin{equation}
\label{const_extn}
0 \to H^0(T,\H) \to V \to \Z(0) \to 0
\end{equation}
in the category of MHS can be viewed as a constant VMHS over $T$. One can
push it out along the inclusion $H^0(T,\H) \hookrightarrow \H$ to obtain
an extension of $\Z(0)_T$ by $\H$:
$$
\xymatrix{
0 \ar[r] & H^0(T,\H) \ar[r]\ar[d] & V \ar[r]\ar[d] & \Z(0) \ar[r]\ar[d] & 0 \cr
0 \ar[r] & \H \ar[r] & \V \ar[r] & \Z(0)_T \ar[r] & 0
}
$$
The map $j$ takes the extension (\ref{const_extn}) to this extension.
Injectivity will follow from the computation below.

We'll say that an extension
$$
0 \to \H \to \V \to \Z(0)_T \to 0
$$
in the category $\cH(T)$ of admissible VMHS over $T$ is {\em topologically
trivial} if it splits in the category $\cL(T)$ of local systems over $T$. The
kernel of $\delta$ consists of those extensions $\V$ that are topologically
trivial.

To prove exactness, it suffices to show that every topologically trivial
extension of $\Z(0)_T$ by $\H$ arises from an element of
$\Ext^1_\cH(\Z(0),H^0(T,\H_\Z))$ by the pushout construction above. Suppose that
$\V$ is an extension of $\Z(0)_T$ by $\H$ in $\cH(T)$. Then the Theorem of the
Fixed Part for VMHS \cite{steenbrink-zucker} implies that $H^0(T,\V)$ has a
natural MHS. The exact sequence
\begin{equation}
\label{extn}
0 \to \H \to \V \to \Z(0)_T \to 0
\end{equation}
gives a long exact sequence
$$
0 \to H^0(T,\H) \to H^0(T,\V) \to \Z(0) \overset{d}{\to} H^1(T,\H).
$$
Then $\V$ is a topologically trivial extension of $\Z$ by $\H$ if and only
if $d=0$. Thus, when $\V$ is a topologically trivial extension, we have
the short exact sequence
\begin{equation}
\label{invar_extn}
0 \to H^0(T,\H) \to H^0(T,\V) \to \Z(0) \to 0
\end{equation}
which we regard as an element of $\Ext^1_\cH(\Z(0),H^0(T,\H_\Z))$. It may be
regarded as a subvariation of $\V$ over $T$. This construction defines a
function
$$
r : \ker \delta \to \Ext^1_\cH(\Z(0),H^0(T,\H_\Z))
$$
which is easily seen to be a homomorphism. It is also easy to check that $r\circ
j$ is the identity. To prove exactness of the sequence (\ref{seqce}), it
suffices to show that $r$ is injective. But this is clear, for if the extension
(\ref{invar_extn}) is split, then one has a Hodge splitting $s : \Z(0)_T \to
H^0(T,\V)$ of it. But, regarding $H^0(T,\V)$ as a constant subvariation of $\V$,
we see that $s$ gives a splitting of (\ref{extn}) in $\cH(T)$. This implies
that $\ker r = 0$.

\end{document}